\documentclass[preprint,3p,12pt]{elsarticle}

\usepackage[utf8]{inputenc}
\usepackage[T1]{fontenc}
\usepackage{calc}
\usepackage{subfigure}
\usepackage{titlesec}
\usepackage{graphicx}
\usepackage{mathtools}
\usepackage{xcolor}
\usepackage{multirow}
\usepackage{comment}
\usepackage{ulem}
\usepackage{url}

\usepackage[linesnumbered,ruled,vlined]{algorithm2e}
\usepackage{amssymb}
\usepackage{amsthm}

\usepackage{lineno}
\newtheorem{thm}{Theorem}
\usepackage{amsthm}

\journal{Elsevier}
\SetKwInput{KwInput}{Input}                
\SetKwInput{KwOutput}{Output}              

\graphicspath{{img/}}

\begin{document}

\begin{frontmatter}


  \title{Convolution tensor decomposition for efficient high-resolution solutions to the Allen-Cahn equation}


  \author[label1]{Ye Lu\corref{cor1}}
  \cortext[cor1]{Corresponding author}
  \ead{yelu@umbc.edu}
  \author[label1]{Chaoqian Yuan}
  \author[label2]{Han Guo}

  \address[label1]{Department of Mechanical Engineering, University of Maryland Baltimore County, Baltimore, USA}
\address[label2]{ACI Renault Le Mans, France}


  \begin{abstract}
This paper presents a convolution tensor decomposition based model reduction method for solving the Allen-Cahn equation. The Allen-Cahn equation is usually used to characterize  phase separation  or the motion of anti-phase boundaries in materials.  {Its} solution is time-consuming when high-resolution meshes and large time scale integration are involved. To resolve these issues, the convolution tensor decomposition method is developed, in conjunction with a stabilized semi-implicit scheme for time integration. The development enables a powerful computational framework for high-resolution solutions of Allen-Cahn problems, and allows the use of relatively large time increments for time integration without violating the discrete energy law. To further improve the efficiency and robustness of the method, an adaptive algorithm is also proposed.   Numerical examples have confirmed the efficiency of the method in both 2D and 3D problems. Orders-of-magnitude speedups were obtained with the method for high-resolution problems, compared to  {the} finite element method. The proposed computational framework opens numerous opportunities for  simulating complex microstructure formation in materials on large-volume high-resolution meshes at a deeply reduced computational cost. 
  
  \end{abstract}


\end{frontmatter}


\newcommand{\revision}[2]{\sout{#1} \textcolor{red}{(#2)}}

\section{Introduction}
The Allen-Cahn (AC) equation, originally introduced by Allen and Cahn \cite{allen1979microscopic}, is a nonlinear reaction-diffusion equation that can be used to describe the process of phase separation in multi-component alloy systems or the so-called antiphase boundary motion in crystals. It has been widely used, sometimes under the name of phase field method, in various scientific and engineering areas, such as computational biology \cite{inan2020analytical} and material sciences \cite{zheng2014phase,yu2018phase,pinomaa2019quantitative,chadwick2021development}. However, solving the AC equation has been computationally challenging,  when a fine  discretization/resolution (in space and time) is needed to resolve the spatial and temporal length scales of the interface motion. This has been marked as an important issue with the phase field method for  modeling the detailed microstructure formation in materials, which consequently limits the size or dimension of computational domains with phase field simulations \cite{chadwick2021development,flint2018prediction,yang2021phase}. 

Various techniques have been developed for accelerating the solutions to the AC equation. An important step is to have an efficient time integration scheme for solving the nonlinear and transient AC equation. Existing schemes can be categorized into three main families: explicit, implicit, and semi-implicit. Explicit time integration has been appealing for its simplicity in terms of implementation. In explicit schemes \cite{chadwick2021development,shen2010numerical}, all the nonlinear terms are treated as constant for each time step, which prevents the extensive iterations for updating the solution variable during each time step. However, it is well known that explicit schemes suffer from the constraint  {to} extremely small time steps and do not satisfy the discrete energy requirement in the AC equation \cite{shen2010numerical}. On the contrary, implicit schemes \cite{condette2011spectral} can meet the energy requirement automatically and can enable relatively larger time steps, but require solving a nonlinear equation at each time step, which causes concerns on the convergence and computational expenses.  {In this regard, works have been done to improve the efficiency of implicit solvers (e.g., \cite{graser2013time,graser2015nonsmooth})}.  Semi-implicit schemes \cite{yang2009error} can be a good compromise between explicit and implicit schemes. In theses schemes, the nonlinear terms are still treated as constant like the explicit ones, whereas the other terms remain the same as the implicit ones. Therefore, semi-implicit schemes can have some advantages of both methods. Nevertheless, additional stabilization  to semi-implicit schemes is  needed to guarantee an unconditional satisfaction of the discrete energy requirement of the AC equation, leading to the so-called stabilized semi-implicit schemes \cite{shen2010numerical}. The stabilized semi-implicit time integration is free of time step constraint for energy stability and allows to solve a system of equations with  constant coefficients for each time step, which is very efficient and easy to implement.

We remark that there are some other time integration techniques that can have similar advantages to the stabilized semi-implicit schemes, such as convex splitting schemes \cite{shen2012second,wise2009energy},   {the} invariant energy quadratization approach \cite{zhao2017numerical,pan2023novel},  and  {the} scalar auxiliary variable approach \cite{shen2019new}. Since the focus of this work is not to study the different time integration methods, the stabilized semi-implicit time integration \cite{shen2010numerical} should be an adequate choice for our work for its simplicity in terms of implementation.

In addition to the time integration, the computational cost related to the spatial resolution is another important and remaining issue that limits the applications  of the AC equation. Conventional finite difference or finite element (FE) methods can be used for the spatial discretization, but both can be  expensive for high-resolution meshes. The  {rapid} growth of the computational cost with respect to the increasing degrees of freedom (DoFs) is usually observed with the conventional numerical methods. Model reduction (also called reduced order modeling) techniques  can overcome this limit and have been considered promising for solving complex large size problems.

In the literature, model reduction techniques  have been developed for both linear and nonlinear problems \cite{chinesta2017model,liu2018model,lu2024extended}. A relatively simple but efficient way to construct reduced order models is  homogenization or the so-called clustering approach that groups the DoFs of a full order system into fewer DoFs. Although the local physical information may be missed due to the reduced DoFs, the homogenization/clustering based reduced order models have demonstrated a great success in  applications for multi-scale modeling of materials \cite{liu2016self,li2019clustering}. Another type of model reduction method relies on the projection of the physical systems onto some (reduced order) latent spaces, based on data or preliminary knowledge about the potential solutions. This  includes reduced basis method \cite{maday2002reduced,hesthaven2016certified}, proper orthogonal decomposition (POD) \cite{willcox2002balanced,goury2016automatised,kerfriden2011bridging,lu2018space}, and hyper reduction \cite{ryckelynck2009hyper,carlberg2013gnat,zhang2017efficient,lu2020adaptive,kaneko2021hyper}, etc.  The proper generalized decomposition (PGD) method \cite{ladeveze2010latin,boucinha2014ideal,ammar2006new,chinesta2011short,lu2021reduced} can be viewed as an extension of POD, based on an assumption of separation of variables. This is a powerful technique for space-time and high-dimensional problems \cite{giacoma2016efficient,chinesta2013pgd}. Some other variants or further extensions of the PGD method include tensor decomposition \cite{zhang2022hidenn}, preconditioned PGD \cite{tang2024solving}, and extended tensor decomposition \cite{lu2024extended}, etc.

In the context of the AC  {problem}, very little work has been done on model reduction. Existing work mostly used POD \cite{song2016reduced,uzunca2017energy,zhou2019reduced,li2021numerical,wu2023certified} or the so-called generalized multiscale FE method \cite{tyrylgin2021multiscale} to construct reduced order models. These methods usually rely on an offline precomputed snapshot database, which causes concerns on the accuracy and  efficiency for high-resolution problems. This motivated us to develop a more efficient model reduction framework for solving the AC equation, especially for high-resolution solutions.

We rely on a recently developed convolution tensor decomposition (CTD) method \cite{lu2023convolution,li2023convolution}. The CTD method is based on a convolution FE (CFE) approximation \cite{lu2023convolution} and the tensor decomposition (or PGD) method. Here we do not distinguish the tensor decomposition and PGD, since they both share the same idea of separation of variables, sometimes called Canonical tensor decomposition \cite{kolda2009tensor} in a discrete form. Hence  {the proposed CTD method can also be called} convolution PGD (CPGD). The CTD method is expected to provide smooth and accurate solutions with  convergence to  pure CFE  method \cite{lu2023convolution} when increasing the number of modes, while keeping a high efficiency for high-resolution problems. No database is needed for the CTD method. Applications to high-resolution topology optimization have demonstrated its superior performance \cite{li2023convolution}. In this work, we further develop the CTD method for solving the AC equation.  {For the sake of the efficiency,} we propose to leverage the stabilized semi-implicit time integration scheme, as mentioned earlier, so that the nonlinear terms in the AC equation become constant coefficients for each time step.  {This also simplifies the implementation of CTD for the nonlinear AC equation.} Furthermore, an adaptive solution algorithm,  {which allows to perform adaptively a hybrid FE/CTD computation}, is developed to further enhance the efficiency and robustness of the  model reduction method when the number of modes becomes large.  Numerical examples including 2D and 3D cases have confirmed the efficiency of the method on high-resolution meshes. We believe that the developed CTD method has a great potential to enable high-performance phase field simulations of materials.

This paper is organized as follows. Section 2 presents the formulation of the AC equation,   {and} the energy stability requirement and different time integration schemes are  {also} discussed. Section 3 presents  the proposed CTD method and an adaptive algorithm. Section 4 presents some 2D and 3D numerical examples of the AC problems. Finally, the paper closes with some concluding remarks.  {Please notice that many details and derivations have been moved to Appendices for a better readability of the paper.}

\section{Problem formulation}
\subsection{Strong form}
The AC equation \cite{allen1979microscopic} is a second-order parabolic partial differential equation, which reads
\begin{equation}
\label{eq:AC}
    \frac{d u}{dt}+L\left (\frac{d \mathcal{F}}{du}-\kappa\Delta u \right)=0 {,} \quad \quad  \text{in}\ \Omega\times [0,T] {,}
\end{equation}
where $\Omega$ is the domain of interest, $T$ is the final time,   {$u : \Omega \rightarrow \mathbb{R}$} is an unknown non-conservative field (usually varying in $[-1,1]$ for describing different phases and their interface) and needs to be found with appropriate  initial and  boundary conditions, $L\in \mathbb{R}^+$ is the mobility, $\mathcal{F}(u)$ is a given bulk free energy,  $\kappa\in \mathbb{R}^+$ is the gradient energy parameter that controls the thickness of the interface between two phases.

The AC problem \eqref{eq:AC} can be considered as a gradient flow problem with the energy functional below
\begin{equation}
\displaystyle
        E {(u)}=\int_\Omega (\mathcal{F}{(u)}+\frac{1}{2} \kappa (\nabla u)^2)\ d \Omega {.}   
\end{equation}
Multiplying Eq. \eqref{eq:AC} with  $\left (\frac{d \mathcal{F}}{du}-\kappa\Delta u \right)$, we can find that 
\begin{equation}
\displaystyle
        \frac{\partial E}{\partial t}=-L\int_\Omega \left (\frac{d \mathcal{F}}{du}-\kappa\Delta u \right)^2\ d \Omega {,}   \quad\quad \forall t\in [0,T] {,}
\end{equation}
which implies that $ \frac{\partial E}{\partial t} \leq 0, \forall t\in [0,T]$. The discrete energy law is then 
\begin{equation}
\displaystyle
\label{eq:energylaw}
        E(t^{k+1})-E(t^{k})\leq 0 {,}
\end{equation}
where $t^{k}$ denotes the  {$k$-th} time step. This discrete energy law should be satisfied during the solution procedure of the AC problem \eqref{eq:AC}.  {In our work, a numerical solution is called "energetically stable" if it satisfies the discrete energy law.}

 {Different choices can be adopted here for the bulk free energy function $\mathcal{F}$. Examples can be found in \cite{gokieli2003discrete, gomez2011provably,tang2016implicit, ma2021mixed,song2016reduced}. In this work, we adopted the definition below
\begin{equation}
\label{eq:energyfunction}
   {\mathcal{F}(u)=a_0 (u^2-1)^2} {,}
\end{equation}
where $a_0$ is an energy coefficient. It can be seen that due to the definition of $\mathcal{F}$, the AC problem is usually nonlinear.}

\subsection{Weak form}
The AC equation can be solved by a FE method based on the weak form of the problem. Assuming the homogeneous Neumann boundary condition applies to the problem, i.e., $\nabla u\cdot \boldsymbol{n} =0$ on $\partial \Omega$ with $\boldsymbol{n}$ denoting the normal vector to the boundary, we can obtain weak form of the problem as below
\begin{equation}
\label{eq:AC-weak}
    \int_\Omega \delta u\ \frac{d u}{dt}\ d \Omega + \int_\Omega \delta u\  L w(u) \ d \Omega +\int_\Omega  \nabla \delta u\cdot L\kappa\nabla u \ d \Omega =0 {,} \quad \quad  \forall t \in  [0,T] {,}
\end{equation}
where $w(u)= \frac{d \mathcal{F}}{du}$, $\delta u$ is the test function.  {The above equation should hold for $\forall \delta u$ from a suitable function space, e.g., a Sobolev space with appropriate constraints}. 
\subsection{Semi-discretized form}
The discretization of problem \eqref{eq:AC-weak} can be done by considering the FE approximation
\begin{equation}
    {u_h}(x,y,z)=\sum_{i\in {A}}{N}_{{i}}(x,y,z) u_i=\boldsymbol{N}\boldsymbol{u} {,}
\end{equation}
where $(x,y,z)$ denotes the spatial coordinate,  $N_i$ is the nodal shape function, $u_i$ is the nodal solution,  {$A$ contains the supporting nodes of the FE shape functions}, $\boldsymbol{N}$ and $\boldsymbol{u}$ are the shape function and solution vectors formed by the $N_i$ and $u_i$. With this definition, the problem \eqref{eq:AC-weak} becomes
\begin{equation}
\label{eq:AC-discrete}
    \int_\Omega \delta \boldsymbol{u}^T \boldsymbol{N}^T\ \frac{d {u}}{dt}\ d \Omega + \int_\Omega \delta\boldsymbol{u}^T \boldsymbol{N}^T\  L w(\boldsymbol{u}) \ d \Omega +\int_\Omega  \delta\boldsymbol{u}^T \boldsymbol{B}^T L\kappa \boldsymbol{B}\boldsymbol{u} \ d \Omega =0 {,} \quad \quad  \forall t \in  [0,T] {,}
\end{equation}
where $\boldsymbol{B}$ is the gradient of the shape function,  {i.e., $\boldsymbol{B}=[\frac{\partial \boldsymbol{N}}{\partial x}\  \frac{\partial \boldsymbol{N}}{\partial y}\  \frac{\partial \boldsymbol{N}}{\partial z}]^T$}.  {By considering that Eq. \eqref{eq:AC-discrete} holds for arbitrary $\delta \boldsymbol{u}$, the final} semi-discretized form reads
\begin{equation}
\label{eq:AC-discrete-final}
    \int_\Omega  \boldsymbol{N}^T\ \frac{d {u}}{dt}\ d \Omega + \int_\Omega \boldsymbol{N}^T\  L w(\boldsymbol{u}) \ d \Omega +\int_\Omega  \boldsymbol{B}^T L\kappa \boldsymbol{B}\boldsymbol{u} \ d \Omega =0 {,} \quad \quad  \forall t \in  [0,T] {.}
\end{equation}

\subsection{Time integration}
Now a time integration scheme needs to be used for solving the problem \eqref{eq:AC-discrete-final}. In general, three types of schemes can be defined, i.e., explicit, implicit, and semi-implicit schemes. The explicit scheme is the simplest one and can be written as 
\begin{equation}
\label{eq:AC-discrete-expl}
    \int_\Omega  \boldsymbol{N}^T\ \frac{u^{k+1}-u^{k}}{\Delta t}\ d \Omega + \int_\Omega \boldsymbol{N}^T\  L w(\boldsymbol{u}^k) \ d \Omega +\int_\Omega  \boldsymbol{B}^T L\kappa \boldsymbol{B}\boldsymbol{u}^k \ d \Omega =0  {,}
\end{equation}
or 
\begin{equation}
    \int_\Omega  \boldsymbol{N}^T\boldsymbol{N}\ \frac{\boldsymbol{u}^{k+1}-\boldsymbol{u}^{k}}{\Delta t}\ d \Omega + \int_\Omega \boldsymbol{N}^T\  L w(\boldsymbol{u}^k) \ d \Omega +\int_\Omega  \boldsymbol{B}^T L\kappa \boldsymbol{B}\boldsymbol{u}^k \ d \Omega =0  {,}
\end{equation}
where $u^{k+1}$ is the unknown solution at the current step $t^{k+1}$, $u^{k}$ is the solution computed at the previous time step $t^k$,  {and $\Delta t= t^{k+1}-t^k$ is the time increment}. This formulation provides many advantages in terms of implementation and solution efficiency for each time step. However, very small time increments are usually required and can result in unreasonable computational costs when a large time scale is considered. Furthermore, the  energy law \eqref{eq:energylaw} is not guaranteed with  {the  solutions obtained by the explicit scheme}.

The implicit scheme can automatically guarantee the energy law and can be  {written} as
\begin{equation}
\label{eq:AC-discrete-impl}
    \int_\Omega  \boldsymbol{N}^T\ \frac{u^{k+1}-u^{k}}{\Delta t}\ d \Omega + \int_\Omega \boldsymbol{N}^T\  L w(\boldsymbol{u}^{k+1}) \ d \Omega +\int_\Omega  \boldsymbol{B}^T L\kappa \boldsymbol{B}\boldsymbol{u}^{k+1} \ d \Omega =0 {.} 
\end{equation}
The main challenge with this type of scheme is to solve the nonlinear problem,  {since Eq. \eqref{eq:AC-discrete-impl} needs to be solved many times while updating $w(\boldsymbol{u}^{k+1})$ at each time step.} 

A compromise between the implicit and explicit schemes is the semi-implicit scheme, which reads
\begin{equation}
\label{eq:AC-discrete-semi}
    \int_\Omega  \boldsymbol{N}^T\ \frac{u^{k+1}-u^{k}}{\Delta t}\ d \Omega + \int_\Omega \boldsymbol{N}^T\  L w(\boldsymbol{u}^{k}) \ d \Omega +\int_\Omega  \boldsymbol{B}^T L\kappa \boldsymbol{B}\boldsymbol{u}^{k+1} \ d \Omega =0 {.} 
\end{equation}
In this case, the nonlinear term $w(\boldsymbol{u}^{k})$ becomes constant for each time step. The problem becomes solving a linear system of  {equations} and allows the use of relatively  {larger} time steps than that of explicit scheme. This feature is very attractive. However, the energy law \eqref{eq:energylaw} still imposes a step size constraint for this scheme \cite{shen2010numerical}. It can be shown that the following condition has to be satisfied so that the semi-implicit solutions are energetically stable 
\begin{equation}
\label{eq:AC-discrete-semi-step}
    \Delta t \leq \frac{2}{L\ \underset{{t\in [0, T]}}{\max}(w')} {,}
\end{equation}
where $w'=\frac{d w}{d u}$. This undesired constraint can be overcome by a stabilized semi-implicit formulation presented in the next section.

\subsection{Stabilized semi-implicit formulation}
We adopt here a stabilized semi-implicit formulation developed in \cite{shen2010numerical} for overcoming the time step constraint due to the energy requirement. Other stabilizing techniques \cite{hu2009stable,wise2009energy} can also be considered for the same purpose. Testing the different stabilizing techniques is out of the scope of the paper. Hence, we only present the one used in this work, which reads
\begin{equation}
\label{eq:AC-discrete-semi-stable}
    \int_\Omega  \boldsymbol{N}^T\ (u^{k+1}-u^{k})(\frac{1}{\Delta t}+\alpha L)\ d \Omega + \int_\Omega \boldsymbol{N}^T\  L w(\boldsymbol{u}^{k}) \ d \Omega +\int_\Omega  \boldsymbol{B}^T L\kappa \boldsymbol{B}\boldsymbol{u}^{k+1} \ d \Omega =0  {,}
\end{equation}
where $\alpha$ is a stabilizing coefficient.  {To simplify the notation, we omit the superscript of $\boldsymbol{u}^{k+1}$ and let $\boldsymbol{u}$ denote the solution at the step $k+1$. Then,} the final discretized equation reads 
\begin{equation}
\label{eq:AC-discrete-semi-stable-K}
    \bold{K} \boldsymbol{u}=\bold{Q} {,}
\end{equation}
with
\begin{equation}
\begin{cases}
        \bold{K}=\int_\Omega (1/{\Delta t}+\alpha L)\boldsymbol{N}^T\boldsymbol{N}\ d\Omega+\int_\Omega \boldsymbol{B}^T L\kappa \boldsymbol{B}\ d\Omega {,}\\
        \bold{Q}= \int_\Omega (1/{\Delta t}+\alpha L)\boldsymbol{N}^T\boldsymbol{N}\boldsymbol{u}^{k}\ d\Omega-\int_\Omega \boldsymbol{N}^T\  L w(\boldsymbol{u}^{k}) \ d \Omega {.}
\end{cases}
\end{equation}

It can be demonstrated that this formulation is  {energetically stable for arbitrary $\Delta t$}, i.e., the energy law \eqref{eq:energylaw} is satisfied for arbitrary $\Delta t$, if $\alpha$  is chosen in the following way
\begin{equation}
\label{eq:AC-discrete-semi-stable-step}
    \alpha \geq \frac{\max(w')}{2} {.}
\end{equation}
The proof can be found in  \ref{apdx:proof}. Since $w'$ is usually bounded with the definition of free energy and the varying range of $u$, it is easy to know the appropriate $\alpha$ once the energy function is defined.  {In the case of  {unbounded} $w'$, we can either limit the range of $u$ or use another stabilization/time integration technique (e.g., \cite{zhao2017numerical,pan2023novel}). The CTD method descriped in the next section can still be applied in the same way.} 

It should be noted that the advantage of "energetically stable"  does not mean $\Delta t$ can be arbitrarily large without influencing the accuracy of the final solution. The accuracy is still influenced by the time discretization in this case, and  {the stabilized semi-implicit scheme is of the same order of accuracy, for the same $\Delta t$, as the original semi-implicit scheme \eqref{eq:AC-discrete-semi},  {as explained in \cite{shen2010numerical}}. For an accurate time integration, a convergence study on the time discretization $\Delta t$ is still needed.} This is reasonable as we do not want the time increment to be limited by anything other than the discretization error.  Higher order semi-implicit formulations can be considered for improved accuracy, but they are not the focus of this work.

In summary, the stabilized semi-implicit formulation allows to release the constraint on the step size due to the energy requirement.  {Compared to} the explicit formulation, relatively  {larger} time steps can be used for solving \eqref{eq:AC-discrete-semi-stable}, which significantly reduces the computational cost for time integration. Nevertheless, solving \eqref{eq:AC-discrete-semi-stable} can still be expensive if high-resolution meshes are used for the spatial discretization. This issue can be addressed by developing the CTD formulation, as presented in the next section.

\section{The proposed CTD formulation for the AC equation}
The CTD formulation includes three ingredients: separation of variables, convolution approximation, and a dedicated solution algorithm.
\subsection{Separation of variables}
The CTD method developed in this work is based on a full separation of spatial variables, which assumes the solution function has the following decomposition
\begin{equation}
\label{eq:CTD}
    u(x,y,z)=\underbrace{\sum_{m=1}^{M} u^{(m)}_x(x)u^{(m)}_y(y)u^{(m)}_z(z)}_{u^{\text{CTD}}} {,}
\end{equation}
where $u(x,y,z)$ denotes the solution of a 3D AC problem, $u^{\text{CTD}}$ is an approximated  solution provided by the CTD method, $M$ is the total number of modes in the CTD solution,  $u^{(m)}_x, u^{(m)}_y, u^{(m)}_z$ are some unknown 1D functions related to  $x,y,z$ directions, respectively. The sum of the products of these 1D functions $u^{(m)}_x, u^{(m)}_y, u^{(m)}_z$ is supposed to provide an approximated solution to $u$. If the number of modes is large enough, the CTD solution should converge to the "exact" solution. Computing the unknown functions $u^{(m)}_x, u^{(m)}_y, u^{(m)}_z$ is essential for the CTD solution. They should be determined by solving the  AC equation,  together with the number of modes  $M$.  The details of the solution algorithm  will be presented later. Please note that the decomposition used here is the same as that in PGD \cite{ammar2006new,ghnatios2019advanced}. Hence, the CTD method can also be called CPGD \cite{lu2023convolution}. We do not distinguish the two methods in this work. 

The advantages of such  {a} decomposition  {are} multi-fold. First, the decomposition enables solving a 3D problem at the costs of a series of 1D problems, since the unknowns  of the problem become  {three} 1D functions $u^{(m)}_x, u^{(m)}_y, u^{(m)}_z$, rather than a 3D function $u(x,y,z)$.  {The computational time and the RAM (random-access memory) requirement during the online computation}  can be significantly reduced. Second, unlike conventional FE methods, whose DoFs grow  {rapidly}  with respect to the mesh resolution, the decomposition enables a linear growth of DoFs. For example, if the mesh has $n$ nodes in each $x,y,z$ direction, the total number of DoFs for a conventional FE analysis is $n^3$, which grows  {rapidly} with $n$, whereas the number of DoFs for CTD is  {$3nM$}, which grows just linearly with $n$. This makes a huge advantage for solving high-resolution (large $n$) problems. Third, the storage space for the solution data can be largely reduced with the decomposition, as only the 1D functions $u^{(m)}_x, u^{(m)}_y, u^{(m)}_z$ need to be stored.  In this sense, the CTD solution data  is automatically in a compressed data format.  

We remark that Eq. \eqref{eq:CTD} assumes regular Cartesian meshes are used for the spatial discretization. This is usually the case for material microstructure modeling. In the case of irregular  {geometries, for which a Cartesian mesh cannot be directly applied}, a geometric mapping can be introduced to handle the problem. An example of such  {a} mapping can be derived from  {the techniques developed for} isogeometric analysis \cite{hughes2005isogeometric,maquart20203d}, as presented in \cite{lu2023convolution,zhang2023isogeometric,kazemzadeh2023nurbs}.  {We will incorporate these techniques in our future work for irregular geometries.}

\subsection{Convolution approximation}
Convolution approximation is another important component in the CTD method. Here, we presents the idea of the so-called CFE approximation \cite{lu2023convolution}.  {Let us consider first a FE approximation written in the natural (local) coordinate system $\boldsymbol{\xi}$}
\begin{equation}
\label{eq:FE_local}
    \begin{aligned}           
    {u^e_h}(\boldsymbol{\xi})=&\sum_{i\in A^e}{N}_{{i}}(\boldsymbol{\xi})u_i,
    \end{aligned}
\end{equation}
where $A^e$ is the local support of an element and is defined as the original FE support when presenting the CFE approximation. For a 2D linear element, $A^e$ has 4 nodes, as shown in the left figure in \figurename~\ref{fig:convapprox}. In general, the 4 nodes can only allow to construct a linear approximation. If higher order elements are needed, more nodes need to be included into $A^e$. This is the usual way to increase the  approximation order in FE.

{CFE uses a different way to increase the approximation order and smoothness. The idea is to keep the same FE mesh (element connectivity and nodes) while using the neighboring nodes of the element to construct high order approximations, with the help of a convolution operation and a smooth function $W$. This smooth function $W$ is called convolution patch function and has a support defined by the so-called convolution patch, as shown in the second figure in \figurename~\ref{fig:convapprox}. In this way, a smooth approximation can be constructed by performing the convolution between the FE approximation and the  patch function $W$, which reads}
\begin{equation}
\displaystyle
\label{eq:convolution}
     u^c (\boldsymbol{\xi})=(u^e_h*W)(\boldsymbol{\xi})=\int_{\Omega_{\text{patch}}} u^e_h(\boldsymbol{\eta})W(\boldsymbol{\xi-\boldsymbol{\eta}})d\boldsymbol{\eta},
\end{equation}
{where $\Omega_{\text{patch}}$ is the convolution patch domain, which can span outside an element with a predefined patch size. Since $W$ can be an arbitrarily high order smooth function, the resultant approximation $u^c$ can also be {of} arbitrarily high order. Considering a discrete form of the above convolution operation \eqref{eq:convolution}, we can have the CFE approximation conceptually written as} 
\begin{equation}
\label{eq:CFEM_concept}
    \begin{aligned}           
    {u^{\text{CFE}}}(\boldsymbol{\xi})=\sum_{k\in A^e_s} \tilde{{N}}_k(\boldsymbol{\xi}) u_k {,}\\
    \end{aligned}
\end{equation}
where $A^e_s$ is the expanded support for the CFE shape functions $\tilde{{N}}_k$, as shown in right figure in \figurename~\ref{fig:convapprox}, $u_k$ is the associated nodal solution. This support $A^e_s$ contains the nodes in $\Omega_{\text{patch}}$, including not only the original FE support $A^e$, but also the nodes surrounding the element. This can be applied to all the elements within the computational domain by moving the convolution patch throughout the domain. For the elements on the boundary, the patch can be reduced to fit the boundary by ignoring the part outside the domain.
\begin{figure}[htbp]
\centering 
{\includegraphics[scale=0.35]{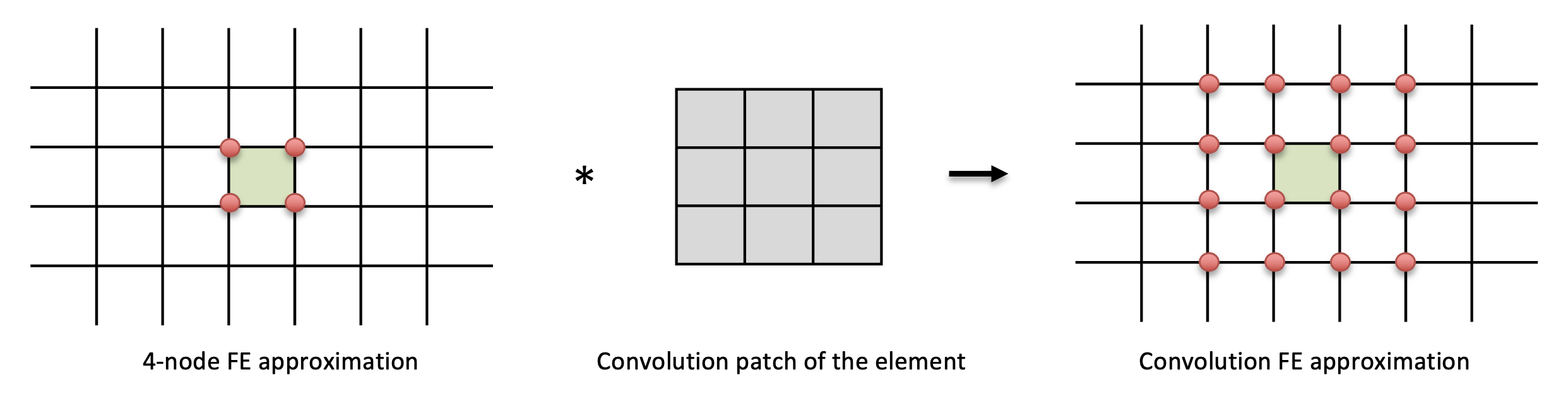}}
\caption{Supporting nodes for  FE and CFE approximations, left: $A^e$, right: $A^e_s$,  {the square in the left figure represents the domain of an element}}
\label{fig:convapprox}
\end{figure}

{In order to get the precise definition for $\tilde{{N}}_k$, we can consider the following equation \cite{lu2023convolution}} 
\begin{equation}
\label{eq:CFEM}
    \begin{aligned}           
    {u^{\text{CFE}}}(\boldsymbol{\xi})=&\sum_{i\in A^e}{N}_{{i}}(\boldsymbol{\xi})\sum_{j\in A^i_s} {W}^{\boldsymbol{\xi}_i}_{a,j} (\boldsymbol{\xi}) u_j=\sum_{k\in A^e_s} \tilde{{N}}_k(\boldsymbol{\xi}) u_k {,}\\
    \end{aligned}
\end{equation}
where $N_i$ is  {a} standard FE shape function, ${W}^{\boldsymbol{\xi}_i}_{a,j}$ is  {the} radial basis interpolation function centered at $\boldsymbol{\xi}_i$ (see e.g., \ref{apdx:radialbasis}), $A^e$ is the original FE support, $A^i_s$ is the nodal patch defined in the radial basis interpolation, $A^e_s=\underset{i\in A^e}{\bigcup}  A^i_s$ contains all the supporting nodes in the convolution patch.  {\figurename~\ref{fig:nodalpatch} illustrates an example of forming the convolution patch using the nodal patch. The final CFE shape function $\tilde{{N}}_k$ can be obtained by developing the double-summation and combining the terms with common coefficients.  For better understanding about how to construct  the convolution approximation, readers can be referred to  \ref{apdx:1DCFEshapefunction}. }

\begin{figure}[htbp]
\centering 
{\includegraphics[scale=0.23]{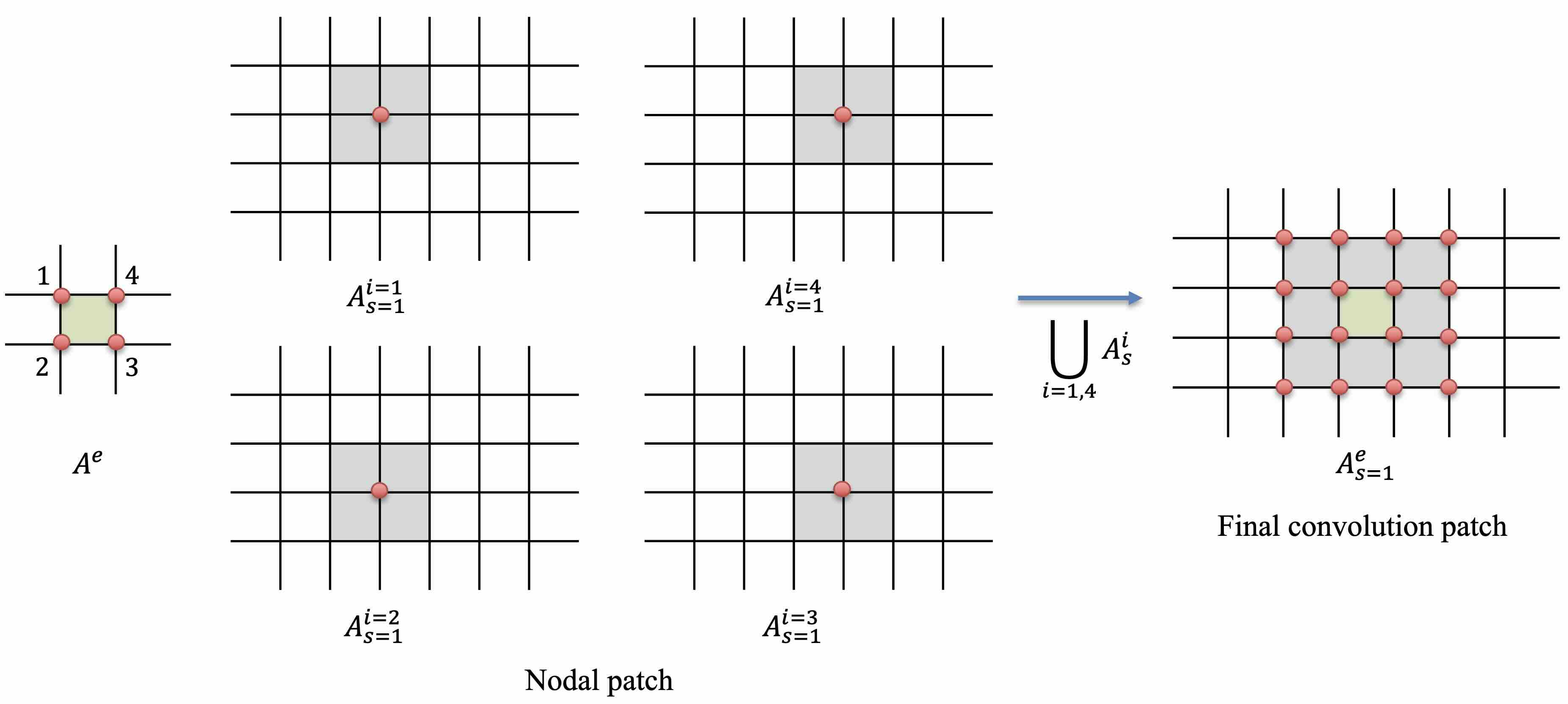}}
\caption{ {Examples of supporting nodes for $A^e$, $A_s^i$, and $A^e_s$, $s$ is the patch size that defines the number of layers outside an element}}
\label{fig:nodalpatch}
\end{figure}

{In the CFE approximation}, three controllable parameters can be defined: polynomial order $p\in \mathbb{Z}^+$ of $W$, dilation parameter $a \in \mathbb{R}^+$, and patch size $s\in \mathbb{Z}^+$. The polynomial order $p$ controls the convergence rate of the convolution approximation. The dilation $a$ controls the window size {(influencing domain)} of kernel functions in $W$. The patch size $s$ defines the number of layers of elements outside the original FE element.  {\figurename~\ref{fig:nodalpatch}} is an example for $s=1$.

An example of the 1D CFE shape functions in the natural coordinate system $\xi\in [-1,1]$ can be found in \figurename~\ref{fig:1Dshape}. The CFE shape functions have a clearly better smoothness than the linear FE, even with $p=1$. To better visualize the global shape of these functions and the influence of different control parameters, we also illustrate the CFE shape functions in the global coordinate system $x\in [-5,5]$ with different $p$ and $a$, as shown in \figurename~\ref{fig:CFEshape-global}. We can see that the CFE shape functions automatically satisfy the Kronecker delta property,  {i.e., $\tilde{{N}}_i(\boldsymbol{\xi}_j)=\delta_{ij}$},  and the partition of unity,  {i.e., $\sum_{i} \tilde{{N}}_i(\boldsymbol{\xi})=1$}, by construction.  Again, the detailed construction of these functions can be found in \ref{apdx:radialbasis} and \ref{apdx:1DCFEshapefunction}.

\begin{figure}[htbp]
\centering 
{\includegraphics[scale=0.35]{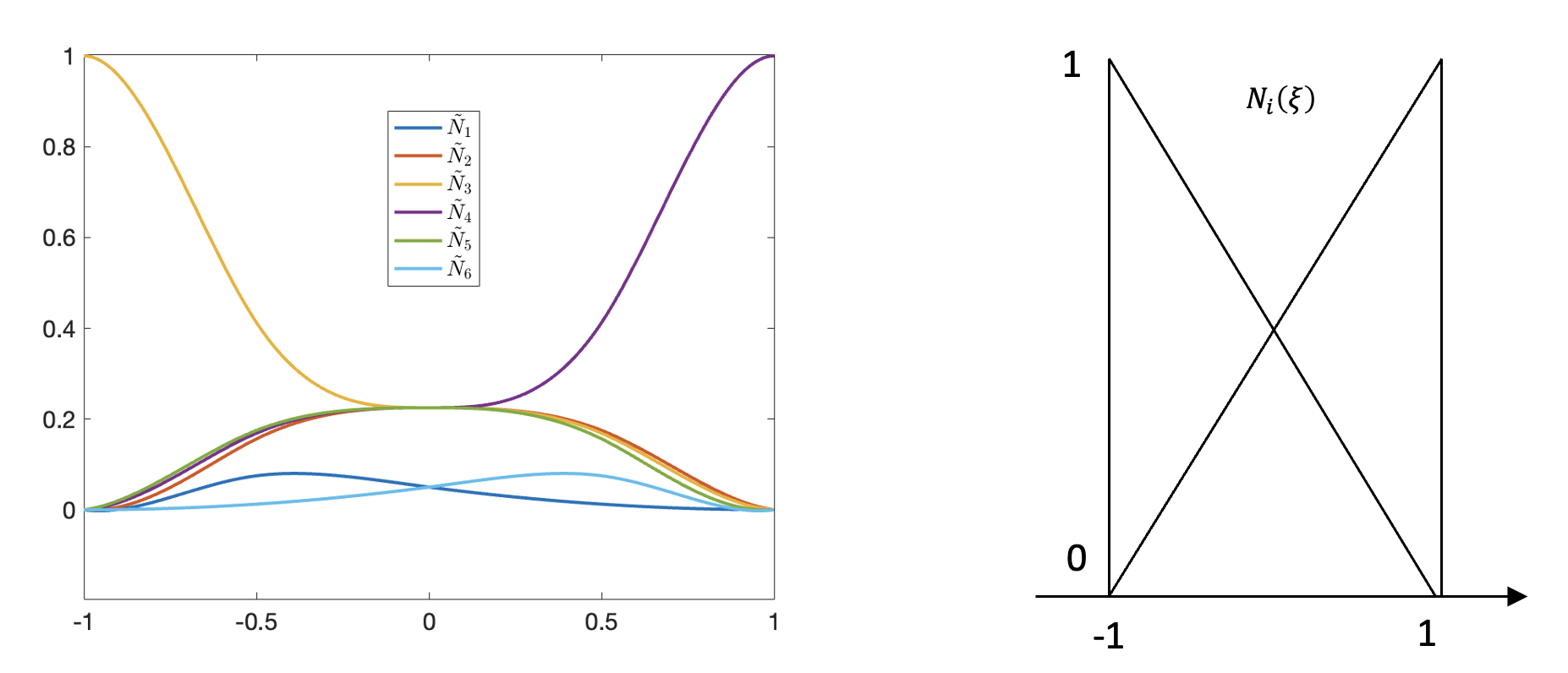}}
\caption{1D CFE and FE shape functions in the natural coordinate system $\xi\in [-1,1]$, left: CFE shape functions $\tilde{{N}}_k({\xi})$ with $s=2, p=1, a=1$, right: linear FE shape functions ${{N}}_i({\xi})$  }
\label{fig:1Dshape}
\end{figure}

\begin{figure}[htbp]
\centering
\subfigure[$p=1, a=1$]{\includegraphics[scale=0.36]{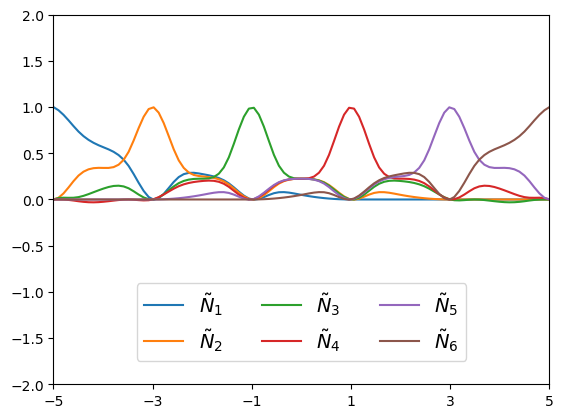}}
\subfigure[$p=1, a=2$]{\includegraphics[scale=0.36]{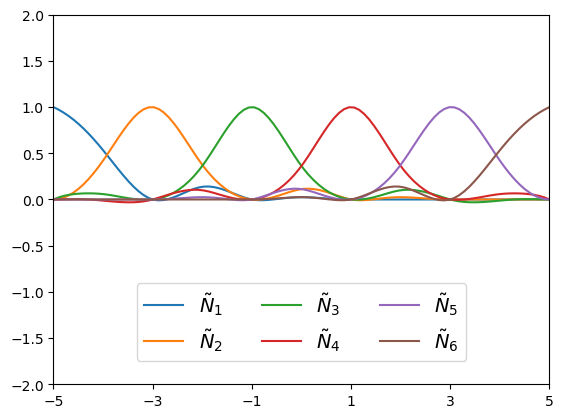}}
\subfigure[$p=2, a=2$ ]{\includegraphics[scale=0.36]{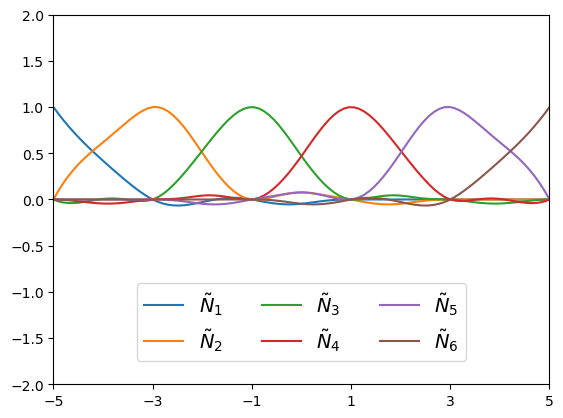}}
\caption{ {CFE shape function in global coordinate system, $\tilde{{N}}_k({x})$ with $x\in [-5,5]$, patch size $s=2$}}
\label{fig:CFEshape-global}
\end{figure}

The superior accuracy and convergence rate of the convolution approximation have been discussed in our previous work \cite{lu2023convolution,li2023convolution,park2023convolution,lu2024convolution}, by comparing with conventional FE approaches. We summarize here the key observations and advantages of the CFE method.
\begin{itemize}
    \item CFE enables arbitrarily high order approximations without changing the mesh, offering flexible $p$-adaptivity.  {Standard FE would have to either insert more nodes into the mesh or increase the element size to include more nodes in the element for high order approximations. This sometimes limits the $p$-adaptivity for standard FE.} 
    \item CFE outperforms FE in terms of accuracy and smoothness of the solution,  {if the same computational cost is considered for both CFE and FE (e.g., \cite{park2023convolution,lu2024convolution}).}
    \item CFE offers additional control parameters, dilation $a$ and patch size $s$, for the length scale control on a fixed mesh, though large  $a$ is usually suggested for better accuracy.
    \item CFE keeps the Kronecker delta property and the partition of unity, facilitating the implementation of Dirichlet boundary conditions,  {compared to other methods that can enable arbitrarily high order approximations, such as meshfree methods (e.g., \cite{nayroles1992generalizing,liu1995reproducing1,li1996moving,liu2010smoothed,chen2017reproducing})}.
\end{itemize}

In the CTD method, only 1D CFE functions are needed, due to the decomposition \eqref{eq:CTD}. The convolution approximation is applied to each of the 1D functions,  $u^{(m)}_x, u^{(m)}_y, u^{(m)}_z$, as below
\begin{equation}
\label{eq:cfeshape1D}
    \begin{aligned}           
    u^{(m)}_x=& \tilde{\boldsymbol{N}}_x \boldsymbol{u}^{(m)}_x {,}\\
    u^{(m)}_y=& \tilde{\boldsymbol{N}}_y\boldsymbol{u}^{(m)}_y {,}\\
    u^{(m)}_z=& \tilde{\boldsymbol{N}}_z \boldsymbol{u}^{(m)}_z {,}\\
    \end{aligned}
\end{equation}
where $\tilde{\boldsymbol{N}}$ is the CFE shape function vector formed by $\tilde{{N}}_k$.  ${u}^{(m)}_x, {u}^{(m)}_y, {u}^{(m)}_z$ are the nodal solutions in each direction.  {Based on the theoretical analysis in \cite{lu2023convolution},} the CTD solution should converge to the CFE solution \eqref{eq:CFEM} with an increased number of modes,  {since they use the same shape functions for approximation}. This has been confirmed by our numerical studies \cite{lu2023convolution}. Moreover, CTD  has the same advantages as CFE in terms of  arbitrarily high order approximations, smoothness, and accuracy, etc. 

\subsection{Solution algorithm}
\label{sect:solualgo}
The CTD solution algorithm for solving the AC equation is presented here. For notation simplification, we first rewrite Eq. \eqref{eq:CTD} as below
\begin{equation}
\label{eq:u-increment}
    u(x,y,z)={\sum_{m=1}^{M-1} u^{(m)}_x u^{(m)}_y u^{(m)}_z}+u_x u_y u_z {,}
\end{equation}
where $u_x, u_y, u_z$ denote the mode functions for the $M$-th mode at a given time step $t^{k+1}$. Without loss of generality, we can assume the $M-1$ modes are already computed and known by previous computations. We present here how to compute the $M$-th mode.

First, taking the variation of $u$, we have 
\begin{equation}
\label{eq:deltau}
    \delta u(x,y,z)= \delta u_x u_y u_z +  u_x \delta u_y u_z + u_x u_y \delta u_z {.}
\end{equation}
Then, we can plug \eqref{eq:u-increment} and \eqref{eq:deltau} into the following stabilized semi-implicit formulation of the AC problem
\begin{equation}
    \int_\Omega  \delta u \ (u^{k+1}-u^{k})(\frac{1}{\Delta t}+\alpha L)\ d \Omega + \int_\Omega \delta u \  L w({u}^{k}) \ d \Omega +\int_\Omega  \nabla\delta u\cdot  L\kappa \nabla{u}^{k+1} \ d \Omega =0  {,}
\end{equation}
with $u^{k+1}=u={\sum_{m=1}^{M-1} u^{(m)}_x u^{(m)}_y u^{(m)}_z}+u_x u_y u_z$ and $u^{k}=\sum_{m=1}^{M_{k}} u^{(m,k)}_x u^{(m,k)}_y u^{(m,k)}_z $, where $u^{k}$ is known by the previous time step $t^k$. This will lead to a nonlinear problem related to $u_x, u_y, u_z$. We can use an alternating fixed point algorithm \cite{ammar2006new,lu2023convolution}  to solve the problem. The overall solution procedure  {using the alternating fixed point algorithm} can be summarized as below:
\begin{enumerate}
    \item Initialize $u_x$, $u_y$, and $u_z$
    \item Fix the $u_y$ and $u_z$ and solve for the $u_x$ 
    \item Fix the $u_x$ and $u_z$ and solve for the $u_y$ 
    \item Fix the $u_x$ and $u_y$ and solve for the $u_z$  
    \item Check convergence of the mode: $|u_x u_y u_z|\rightarrow$ constant
    \item If convergence, go to next step. Otherwise, return to Step 2. 
    \item Increase the number of mode, i.e., $M \leftarrow M+1$, and repeat from Step 1 for the next mode until ${|u_x u_y u_z|}\rightarrow 0$
\end{enumerate}

For a better understanding of the above algorithm, we illustrate here a detailed solution procedure with the final CTD discretized formulation based on a 2D problem, i.e., $u^{k+1}=\sum_{m=1}^{M-1} u^{(m)}_x u^{(m)}_y +u_x u_y $, and assume $u^{k}=\sum_{m=1}^{M_{k}} u^{(m,k)}_x u^{(m,k)}_y $ is already known. The detailed derivation for the discretized formulation  {\eqref{eq:ux}-\eqref{eq:KyQy}} can be found in \ref{apdx:CTDderivation}. 
\begin{enumerate}
    \item Initialize $u_x$ and $u_y$
    \item Fix the $u_y$ and solve for the $u_x$ with the following discretized equation
    \begin{equation}
    \label{eq:ux}
    \bold{K}_x\boldsymbol{u}_x=\bold{Q}_x {,}
    \end{equation}
with 
\begin{equation}
\label{eq:KxQx}
\begin{aligned}
\begin{cases}
    \bold{K}_x=&(1/\Delta t + \alpha L)\bold{M}_{xx}\boldsymbol{u}_y^T\bold{M}_{yy}\boldsymbol{u}_y+L\kappa\bold{K}_{xx}\boldsymbol{u}_y^T\bold{M}_{yy}\boldsymbol{u}_y+L\kappa\bold{M}_{xx}\boldsymbol{u}_y^T\bold{K}_{yy}\boldsymbol{u}_y {,}\\
    \bold{Q}_x=&\sum_{m=1}^{M_k}(1/\Delta t + \alpha L)\bold{M}_{xx}\boldsymbol{u}^{(m,k)}_x\boldsymbol{u}_y^T\bold{M}_{yy}\boldsymbol{u}^{(m,k)}_y\\&-\sum_{m=1}^{M-1}(1/\Delta t + \alpha L)\bold{M}_{xx}\boldsymbol{u}^{(m)}_x\boldsymbol{u}_y^T\bold{M}_{yy}\boldsymbol{u}^{(m)}_y\\
    &-\bold{M}_{xy}\boldsymbol{u}_y\\
    &-\sum_{m=1}^{M-1}L\kappa\bold{K}_{xx}\boldsymbol{u}^{(m)}_x\boldsymbol{u}_y^T\bold{M}_{yy}\boldsymbol{u}^{(m)}_y\\
     &-\sum_{m=1}^{M-1}L\kappa\bold{M}_{xx}\boldsymbol{u}^{(m)}_x\boldsymbol{u}_y^T\bold{K}_{yy}\boldsymbol{u}^{(m)}_y {,}
\end{cases} 
\end{aligned}
\end{equation}
and 
\begin{equation}
\label{eq:KxxMxx}
\begin{aligned}
\begin{cases}
    \bold{K}_{xx}=&\int_{\Omega_x} \Tilde{\boldsymbol{B}}_x^T\Tilde{\boldsymbol{B}}_x\ dx  {,}\\
    \bold{M}_{xx}=&\int_{\Omega_x} \Tilde{\boldsymbol{N}}_x^T\Tilde{\boldsymbol{N}}_x\ dx  {,}\\
    \bold{K}_{yy}=&\int_{\Omega_y} \Tilde{\boldsymbol{B}}_y^T\Tilde{\boldsymbol{B}}_y\ dy  {,}\\ 
    \bold{M}_{yy}=&\int_{\Omega_y} \Tilde{\boldsymbol{N}}_y^T\Tilde{\boldsymbol{N}}_y\ dy {,}\\
    \bold{M}_{xy}=&\int_{\Omega} \Tilde{\boldsymbol{N}}_x^T\Tilde{\boldsymbol{N}}_yLw({u}^{k})\ dx dy {,}
\end{cases} 
\end{aligned}
\end{equation}
where $\Tilde{\boldsymbol{N}}$ is the CFE shape function, $\Tilde{\boldsymbol{B}}$ is the derivative of $\Tilde{\boldsymbol{N}}$, $\Omega:=\Omega_x \times \Omega_y$. 
    \item Fix the $u_x$ and solve for the $u_y$ with the following discretized equation
     \begin{equation}
    \bold{K}_y\boldsymbol{u}_y=\bold{Q}_y {,}
    \end{equation}
with 
\begin{equation}
\label{eq:KyQy}
\begin{aligned}
\begin{cases}
    \bold{K}_y=&(1/\Delta t + \alpha L)\bold{M}_{yy}\boldsymbol{u}_x^T\bold{M}_{xx}\boldsymbol{u}_x+L\kappa\bold{K}_{yy}\boldsymbol{u}_x^T\bold{M}_{xx}\boldsymbol{u}_x+L\kappa\bold{M}_{yy}\boldsymbol{u}_x^T\bold{K}_{xx}\boldsymbol{u}_x {,}\\
    \bold{Q}_y=&\sum_{m=1}^{M_k}(1/\Delta t + \alpha L)\bold{M}_{yy}\boldsymbol{u}^{(m,k)}_y\boldsymbol{u}_x^T\bold{M}_{xx}\boldsymbol{u}^{(m,k)}_x\\&-\sum_{m=1}^{M-1}(1/\Delta t + \alpha L)\bold{M}_{yy}\boldsymbol{u}^{(m)}_y\boldsymbol{u}_x^T\bold{M}_{xx}\boldsymbol{u}^{(m)}_x\\
    &-\bold{M}_{xy}^T\boldsymbol{u}_x\\
    &-\sum_{m=1}^{M-1}L\kappa\bold{K}_{yy}\boldsymbol{u}^{(m)}_y\boldsymbol{u}_x^T\bold{M}_{xx}\boldsymbol{u}^{(m)}_x\\
     &-\sum_{m=1}^{M-1}L\kappa\bold{M}_{yy}\boldsymbol{u}^{(m)}_y\boldsymbol{u}_x^T\bold{K}_{xx}\boldsymbol{u}^{(m)}_x {.}
\end{cases} 
\end{aligned}
\end{equation}
    \item Check convergence of the mode: $|u_x u_y|\rightarrow$ constant,  {e.g., $\max (| (\boldsymbol{u}_x \boldsymbol{u}_y^T)^{(i)}- (\boldsymbol{u}_x \boldsymbol{u}_y^T)^{(i-1)}|)$ $\leq \epsilon_1$, where $i$ is the iteration number, $\epsilon_1$ is a small value $\mathcal{O}(10^{-2})$}  
    \item If convergence, go to next step. Otherwise, return to Step 2
    \item Increase the number of mode, i.e., $M \leftarrow M+1$, and repeat from Step 1 for the next mode  {until ${|u_x u_y|}\rightarrow 0$, e.g., $\max (|\boldsymbol{u}_x \boldsymbol{u}_y^T|)\leq \epsilon_2$, with $\epsilon_2$  a small value $\mathcal{O}(10^{-4})$}
\end{enumerate}
Once the criterion is satisfied, the total number of mode $M$ is determined and the CTD solution for $u^{k+1}$ is obtained. The same procedure is performed for each time step.

As we can see in the solution algorithm, the multidimentional problem is divided into several 1D problems that are cheap to solve. The discretized matrices $\bold{K}_x$ and $\bold{K}_y$ have a dimension much smaller than the $\bold{K}$ of the original discretized full order system \eqref{eq:AC-discrete-semi-stable-K}. This leads to a model order reduction. Solving 1D problems is much cheaper than the 2D/3D problems when high-resolution meshes are used for spatial discretization. Numerical results have confirmed this point and will be presented later.

We remark that the matrix $\bold{M}_{xy}$ in \eqref{eq:KxxMxx} needs a full integration in the domain $\Omega$, which can be easily parallelized and should not be a concern when dealing with high-resolution meshes. An alternative way to perform the integration is to first decompose $w(u^k)$ using singular value decomposition or HOPGD \cite{lu2018multi}, i.e., $w=\sum_{q=1}^{Q} w_x^{(q)}(x)w_y^{(q)}(y)$, and then integrate in a separated form as below
 \begin{equation}
 \label{eq:MxySep}
\bold{M}_{xy}=\int_{\Omega}\Tilde{\boldsymbol{N}}_x^T\Tilde{\boldsymbol{N}}_yLw({u}^{k})\ dx dy =\sum_{q=1}^{Q}L \int_{\Omega_x}\Tilde{\boldsymbol{N}}_x^T w_x^{(q)} dx \int_{\Omega_y}\Tilde{\boldsymbol{N}}_yw_y^{(q)} dy {.}
\end{equation}
In this case, the full multidimensional integration in $\Omega$ is split into products of 1D integration, which can be very efficient even with limited parallel computing resources. 

Another approach that can be used for accelerating the integration is to consider a reduced domain $\Omega_r$ with $\Omega_r:=\{(x,y)\in \Omega\ | \  w(u^k)\neq 0 \}$.  {From the implementation point of view, this can be done by evaluating $w(u^k)$ at the Gauss points of each element. If they are close to zero, we can exclude the element from $\Omega_r$.} Then, the integration becomes
 \begin{equation}
  \label{eq:MxyRD}
\bold{M}_{xy}=\int_{\Omega}\Tilde{\boldsymbol{N}}_x^T\Tilde{\boldsymbol{N}}_yLw({u}^{k})\ dx dy = \int_{\Omega_r}\Tilde{\boldsymbol{N}}_x^T\Tilde{\boldsymbol{N}}_yLw({u}^{k})\ dx dy  {.}
\end{equation}
In 3D cases, the integration can be done in a  way similar to \eqref{eq:MxySep} and \eqref{eq:MxyRD}. 

 {• Adaptivity}

 An adaptive algorithm can be employed to further enhance the computational efficiency of CTD. As shown in the CTD algorithm, the iterative solution procedure is performed until a number of modes is found for certain accuracy. Hence the number of modes $M$ can be critical for the efficiency of the CTD solution.  If the total number of modes $M$ is relatively small,  the CTD algorithm  should be very efficient. Otherwise, the efficiency might be reduced. The following adaptive strategy can used for an  {better} efficiency of the proposed method.
\begin{enumerate}
    \item Start the CTD algorithm
    \item Check if the total number of modes $M$ reaches a critical value $M_c$. If $M\leq M_c$, continue the CTD algorithm. Otherwise, switch to FE or CFE for the current time step
    \item Return to Step 1 for the next time step
\end{enumerate}
The above adaptive algorithm leads to a hybrid CTD/FE (or CTD/CFE) computational strategy and can limit the number of modes for the CTD solution at each time step. This method should enable a  {better} computational efficiency and robustness  {than the pure CTD method when the number of modes becomes very large. This method} is called adaptive CTD (ACTD) in the next section.

We remark that the critical number of modes $M_c$ can also be  adaptively selected for different time steps. This will be investigated more in our future work.

\section{Numerical results}
\subsection{2D AC problem}
We start  {by} applying the proposed CTD method for solving 2D AC problems. \figurename~\ref{fig:domain} illustrates the computational domain and the randomly generated initial condition with $u^0\in [-0.5, 0.5]$. The coefficients in the AC equation \eqref{eq:AC} and the bulk free energy \eqref{eq:energyfunction} are chosen as below in \tablename~\ref{table:coef}. For both cases, the stabilizing coefficient is chosen as $\alpha=50$ for the semi-implicit time integration. We can verify that the criterion \eqref{eq:AC-discrete-semi-stable-step} is satisfied with any solution $u\in [-1,1]$. Therefore, the time step size can be freely chosen without violating the energy law. In this work, the step size is fixed as $\Delta t =0.01$ for $t\in [0,10]$, which results in 1000 time steps. A uniform mesh with element size $h=0.008$ is used first to study the accuracy of the CTD method in terms of solving the AC equation.

\begin{figure}[htbp]
\centering
{\includegraphics[scale=0.5]{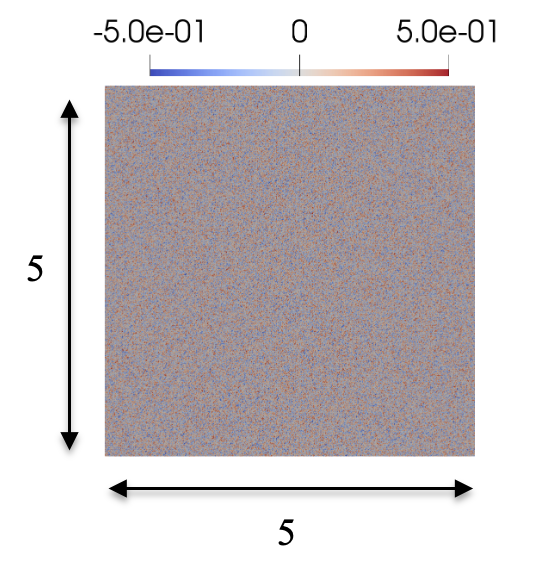}}
\caption{Computational domain and initial condition}
\label{fig:domain}
\end{figure}

\begin{table}[htbp]
\caption{Coefficients in the AC equation and the bulk free energy}
\centering
\begin{tabular}{|c|c|c|c|}
\hline
 Case $\#$ & $L$  & $\kappa$ & $a_0$ \\ \hline
Case 1 & 5 &  1  & 10 \\ \hline
Case 2 & 5  & 0.5 & 10  \\ \hline
\end{tabular}\\
\label{table:coef}
\end{table}

 \figurename~\ref{fig:u2D} illustrates the CTD solutions for the Case 1 problem. For comparison purposes, we also show the FE and CFE solutions under the same problem setup.  The controlling parameters ($s=1, p=1, a=12$) are used for both CTD and CFE for the convolution approximation. They should yield comparable results to linear FE but with better smoothness \cite{lu2023convolution}. The CFE solutions should be taken as the reference to investigate the accuracy of the CTD solutions, although FE provides qualitatively similar results in the current example. Therefore, by taking the difference between the CTD and CFE solutions, we study the point-wise full field accuracy for the CTD solutions. As shown in \figurename~\ref{fig:du2D}, most of the differences between CTD and the reference solutions concentrate at the transition interface between the two phases $-1$ and 1. The CTD solutions remain very accurate over the entire time range. 

\begin{figure}[htbp]
\centering
\subfigure[FE, $t=0.01$]{\includegraphics[scale=0.11]{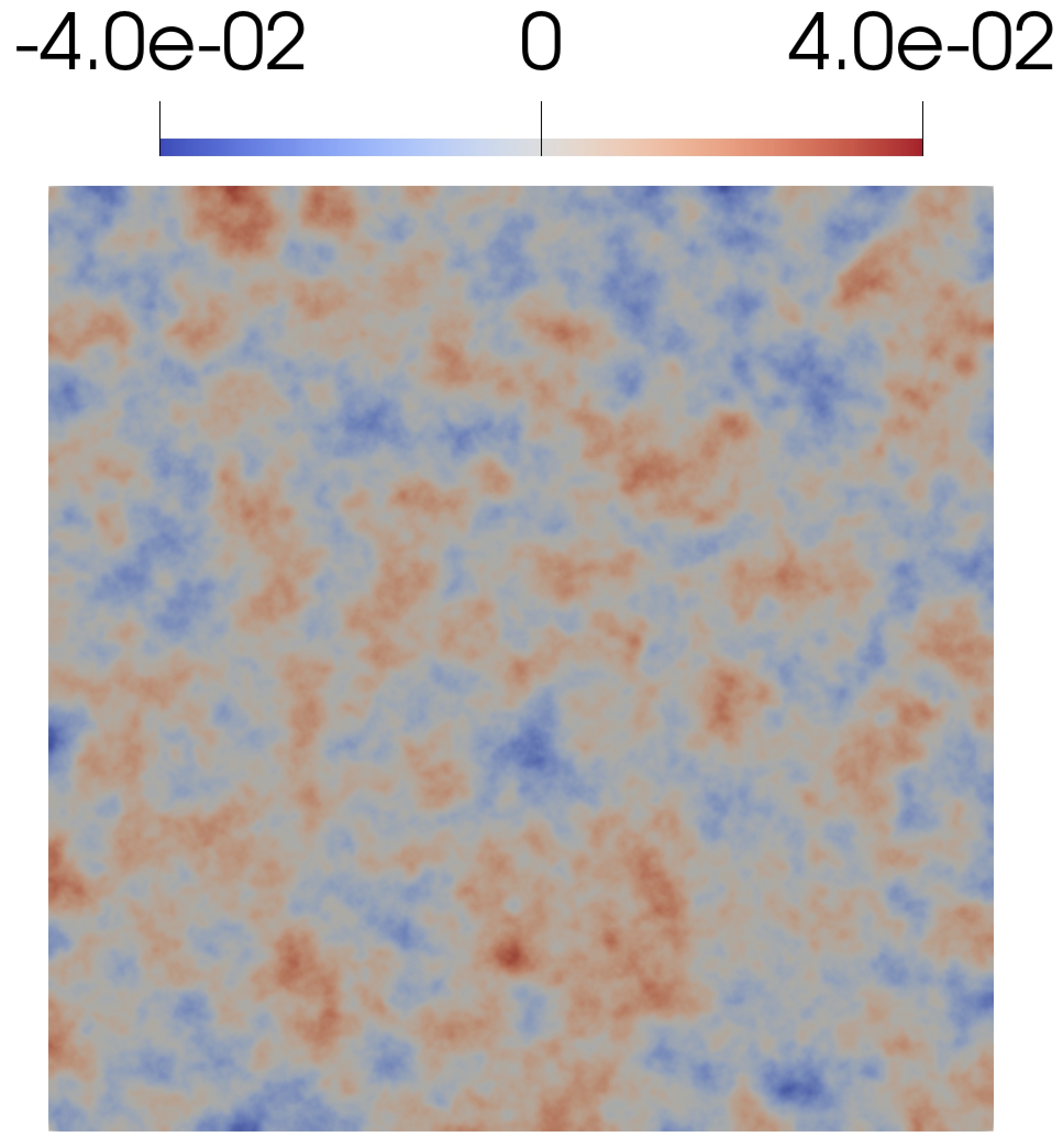}}
\subfigure[FE, $t=0.2$]{\includegraphics[scale=0.11]{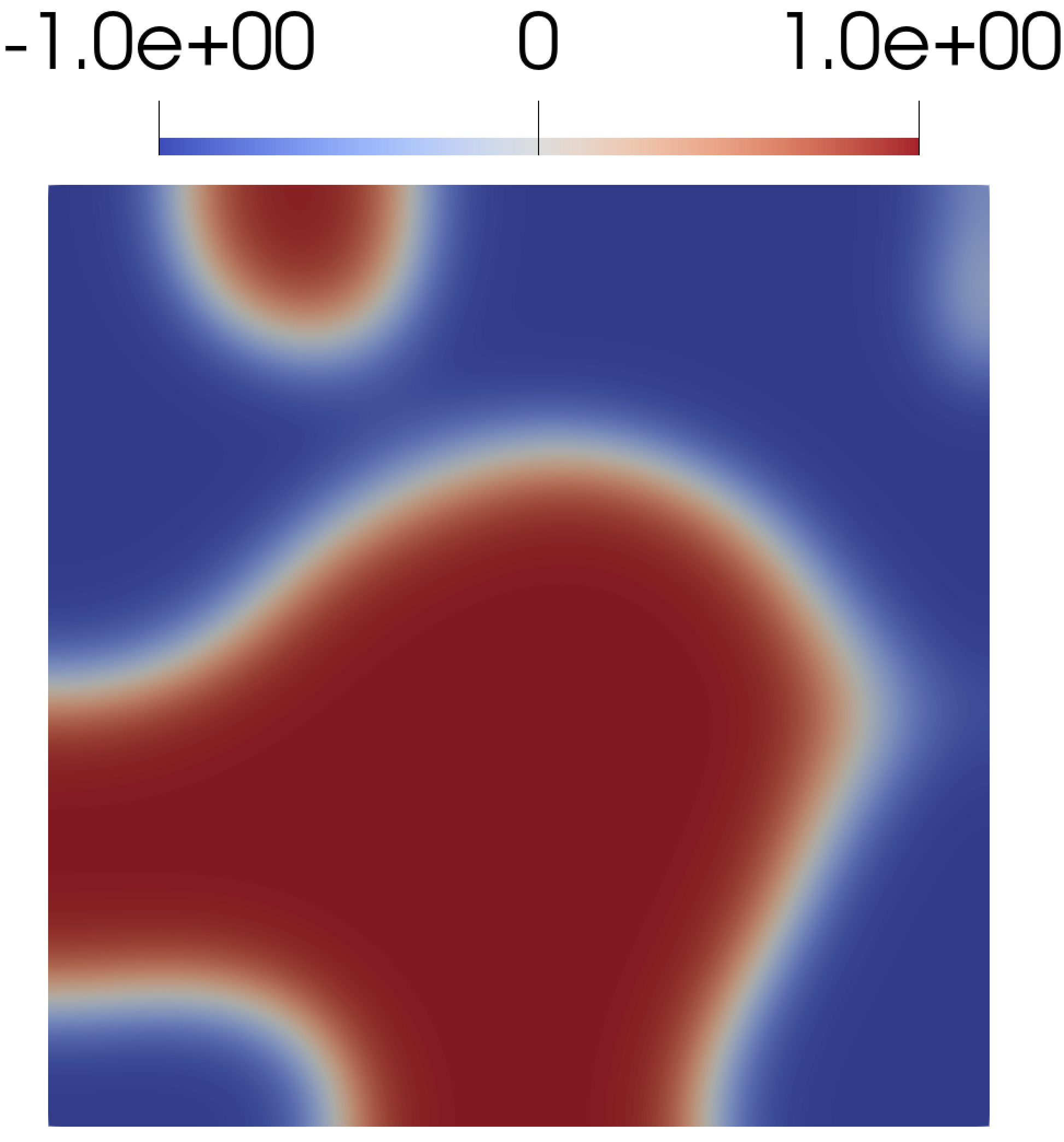}}
\subfigure[FE, $t=2$ ]{\includegraphics[scale=0.11]{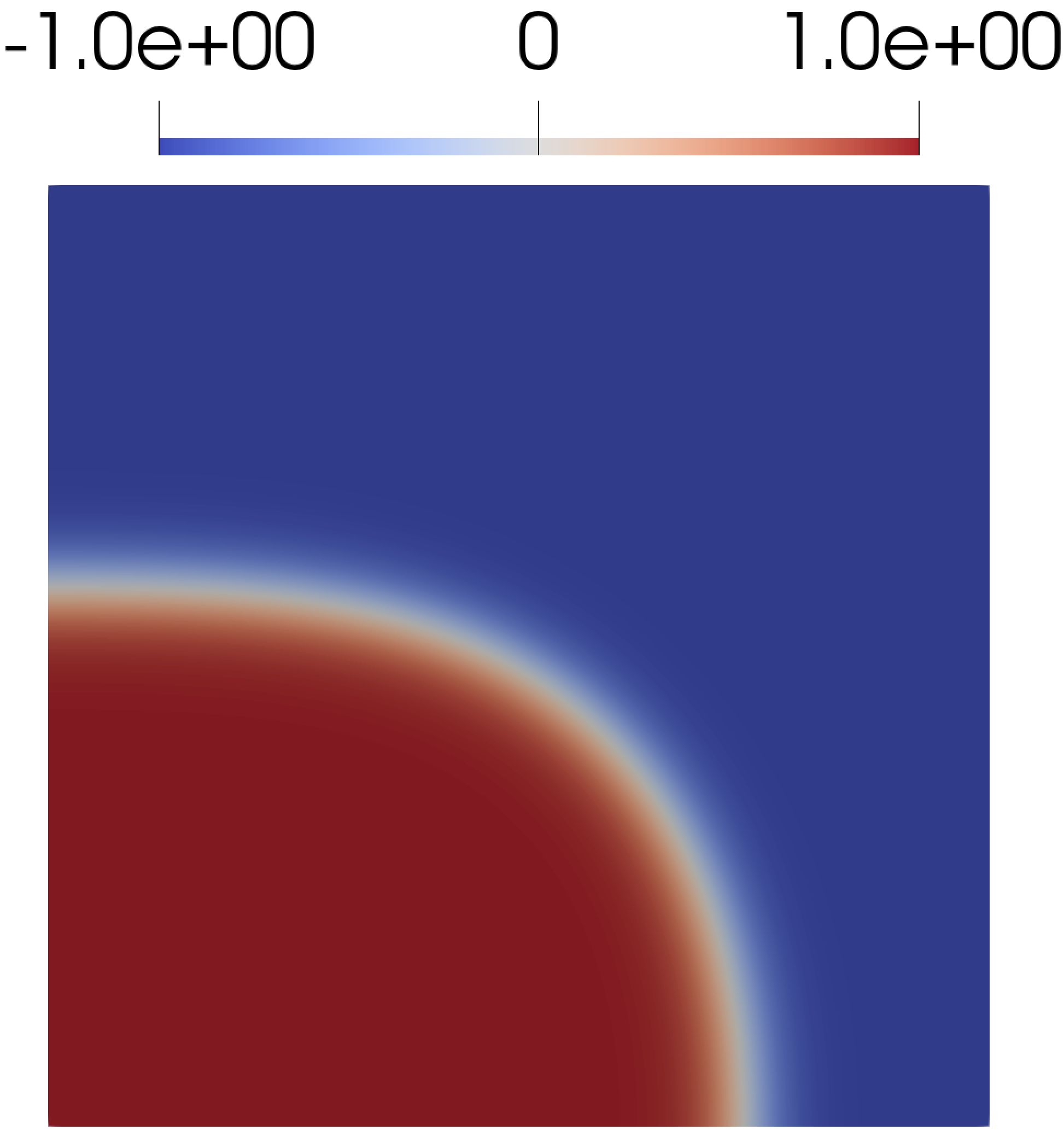}}
\subfigure[FE, $t=10$ ]{\includegraphics[scale=0.11]{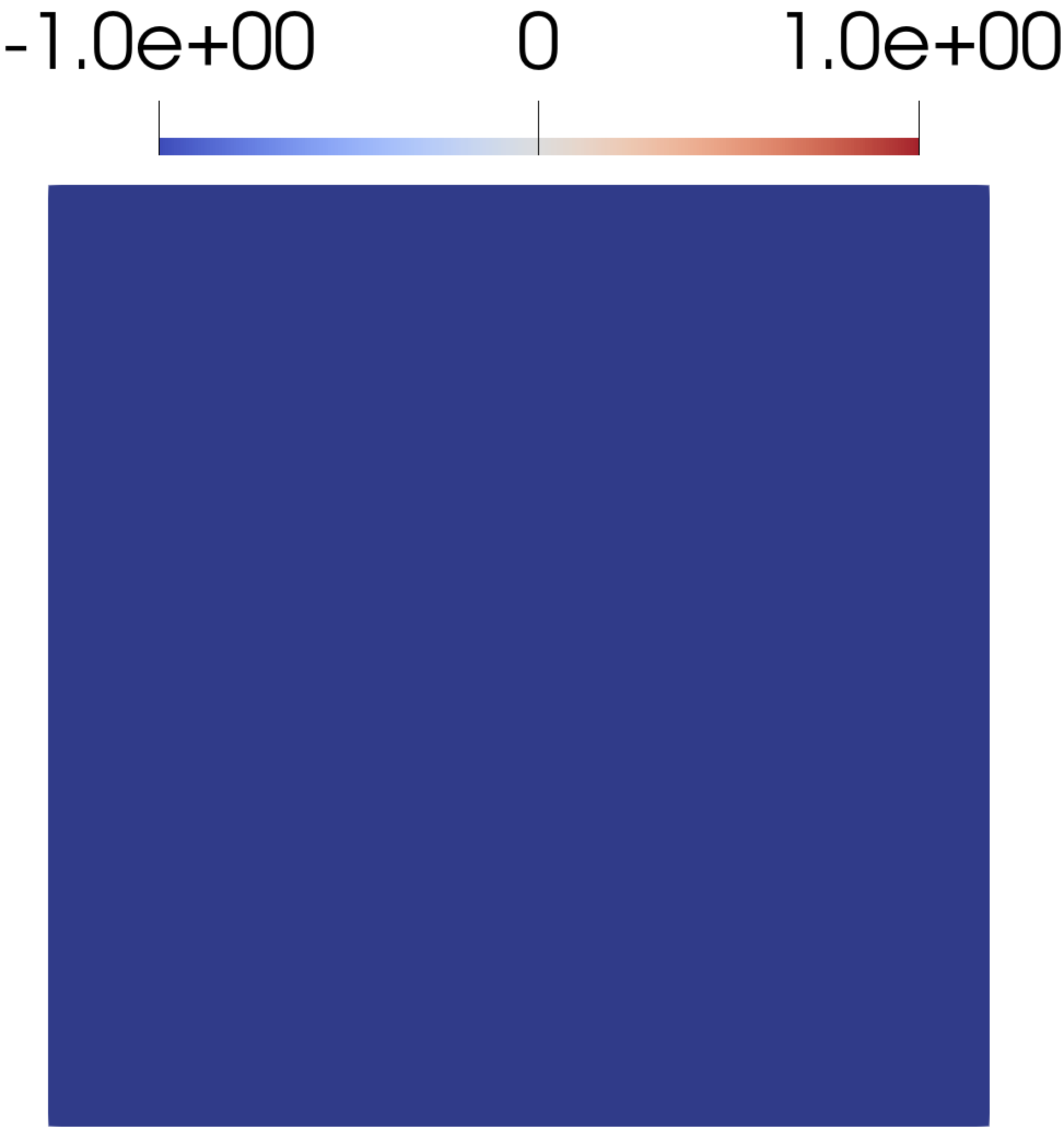}}\\
\subfigure[CFE, $t=0.01$]{\includegraphics[scale=0.11]{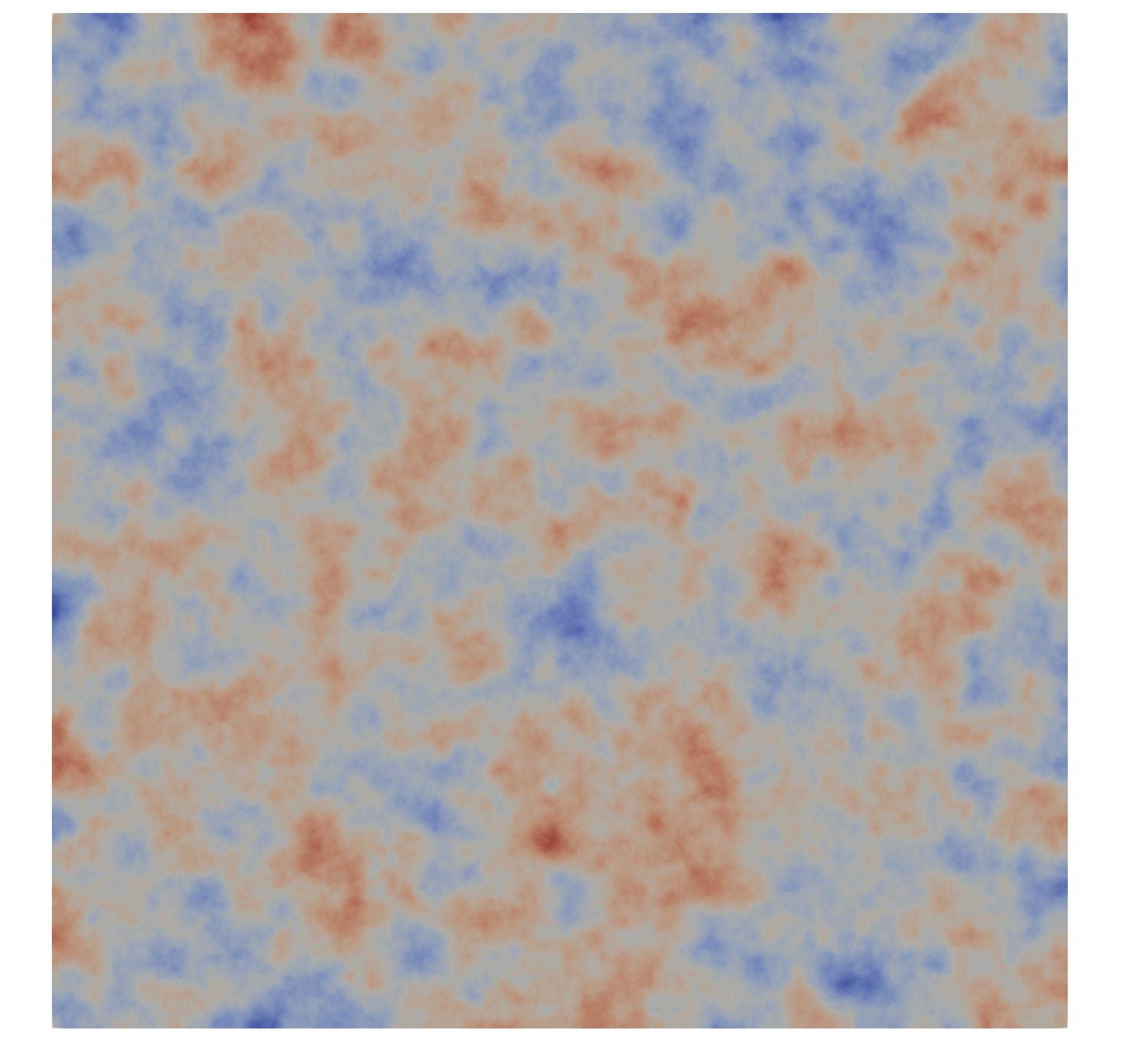}}
\subfigure[CFE, $t=0.2$]{\includegraphics[scale=0.11]{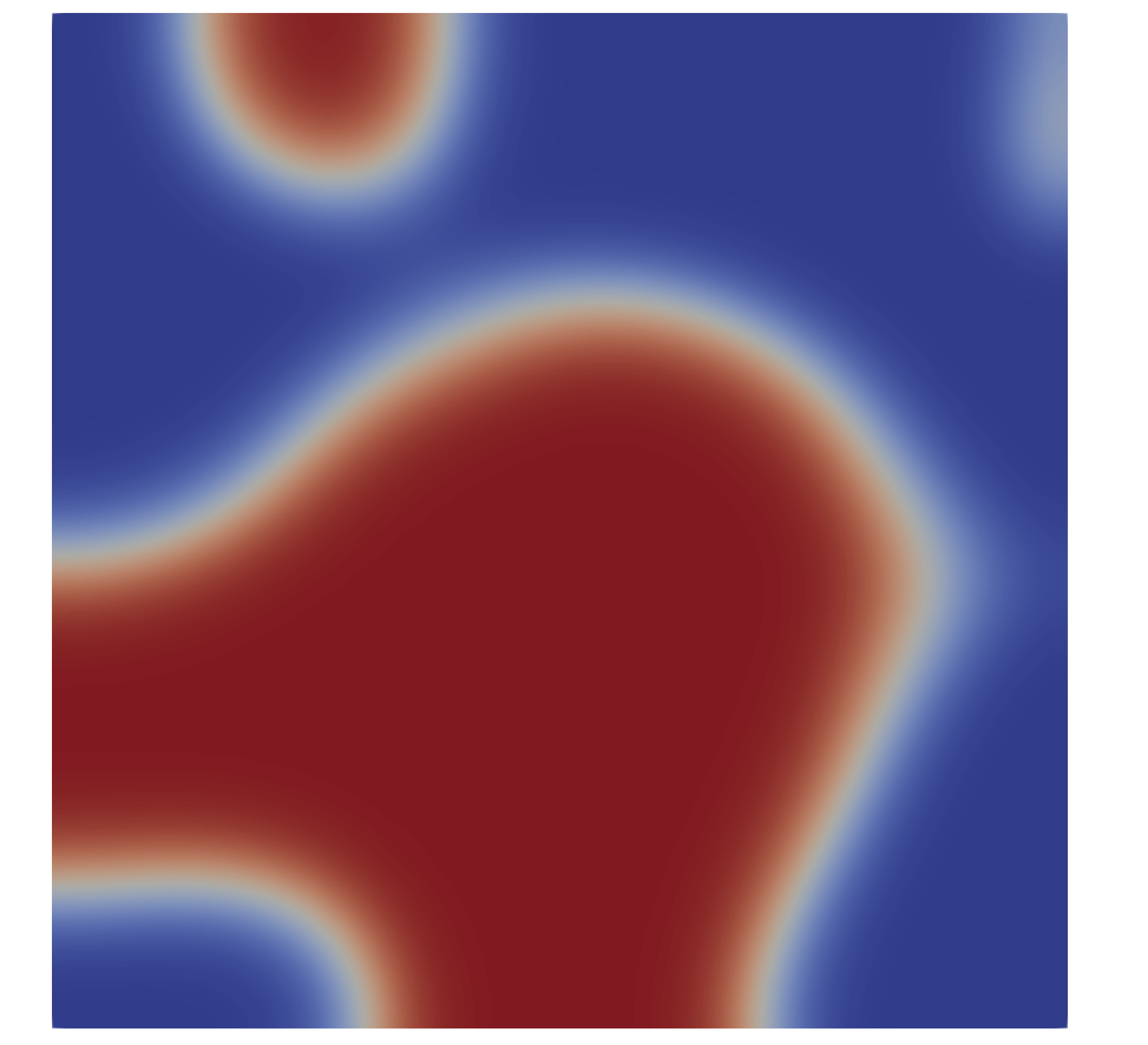}}
\subfigure[CFE, $t=2$ ]{\includegraphics[scale=0.11]{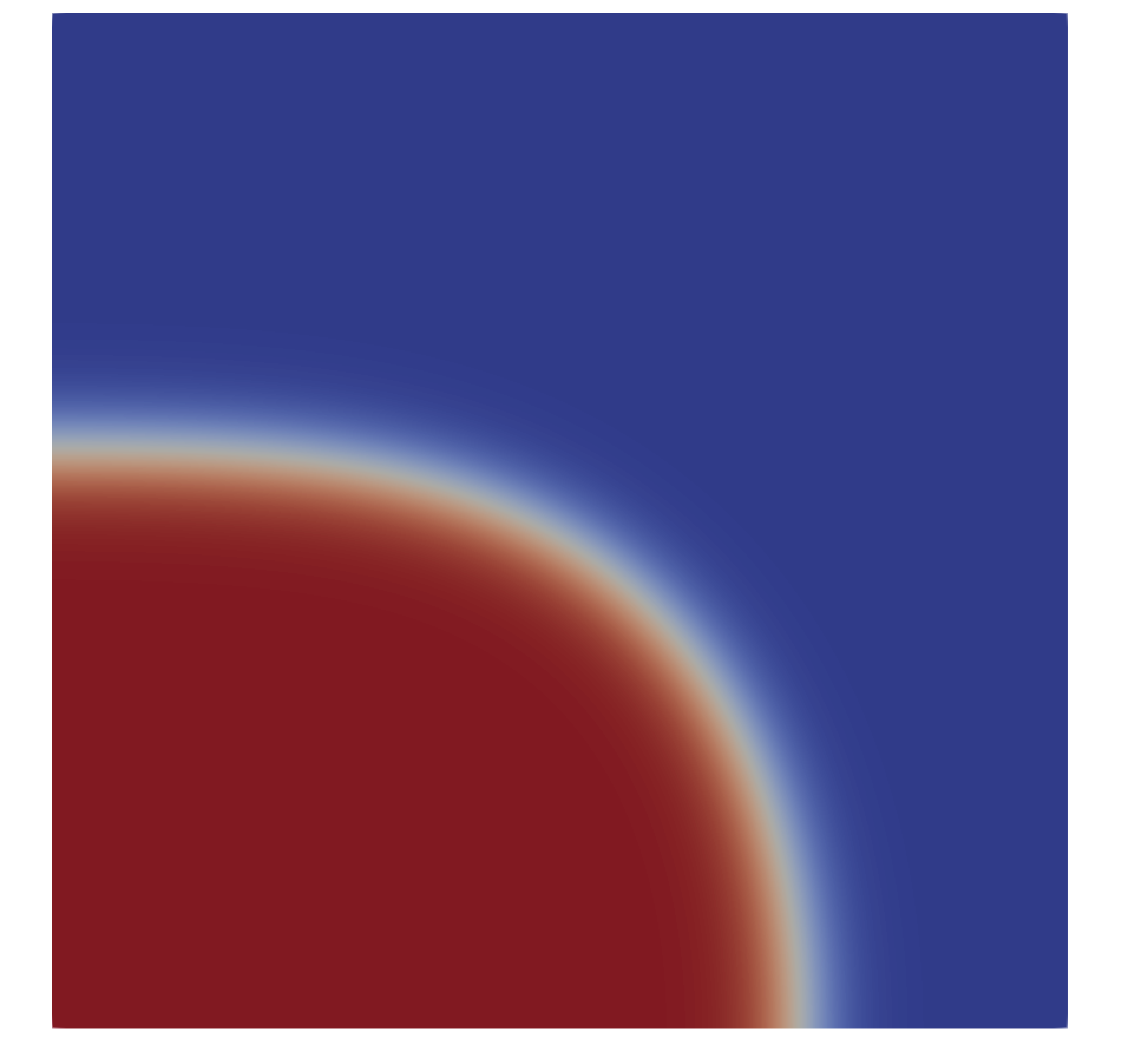}}
\subfigure[CFE, $t=10$ ]{\includegraphics[scale=0.11]{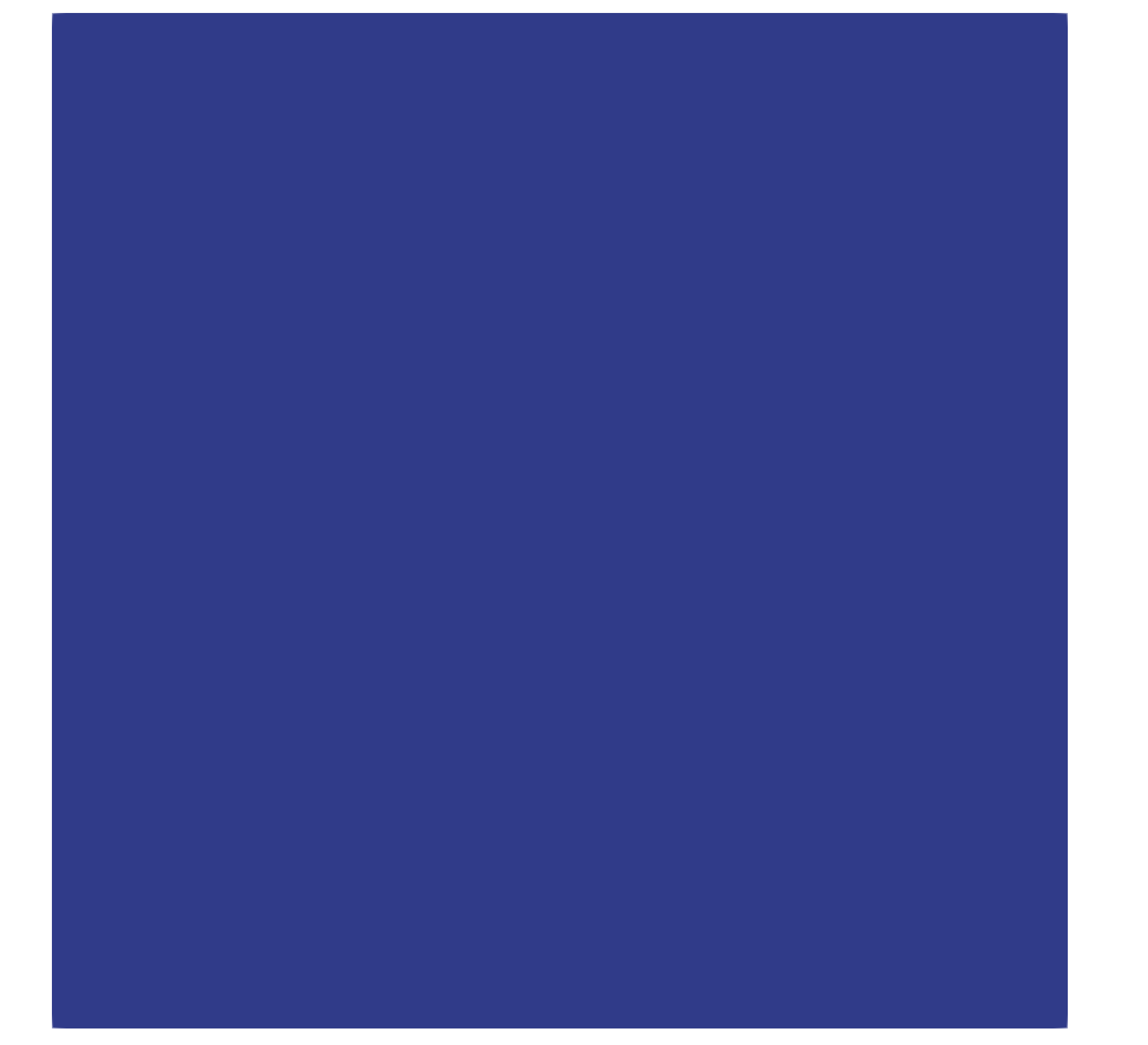}}\\
\subfigure[CTD, $t=0.01$]{\includegraphics[scale=0.11]{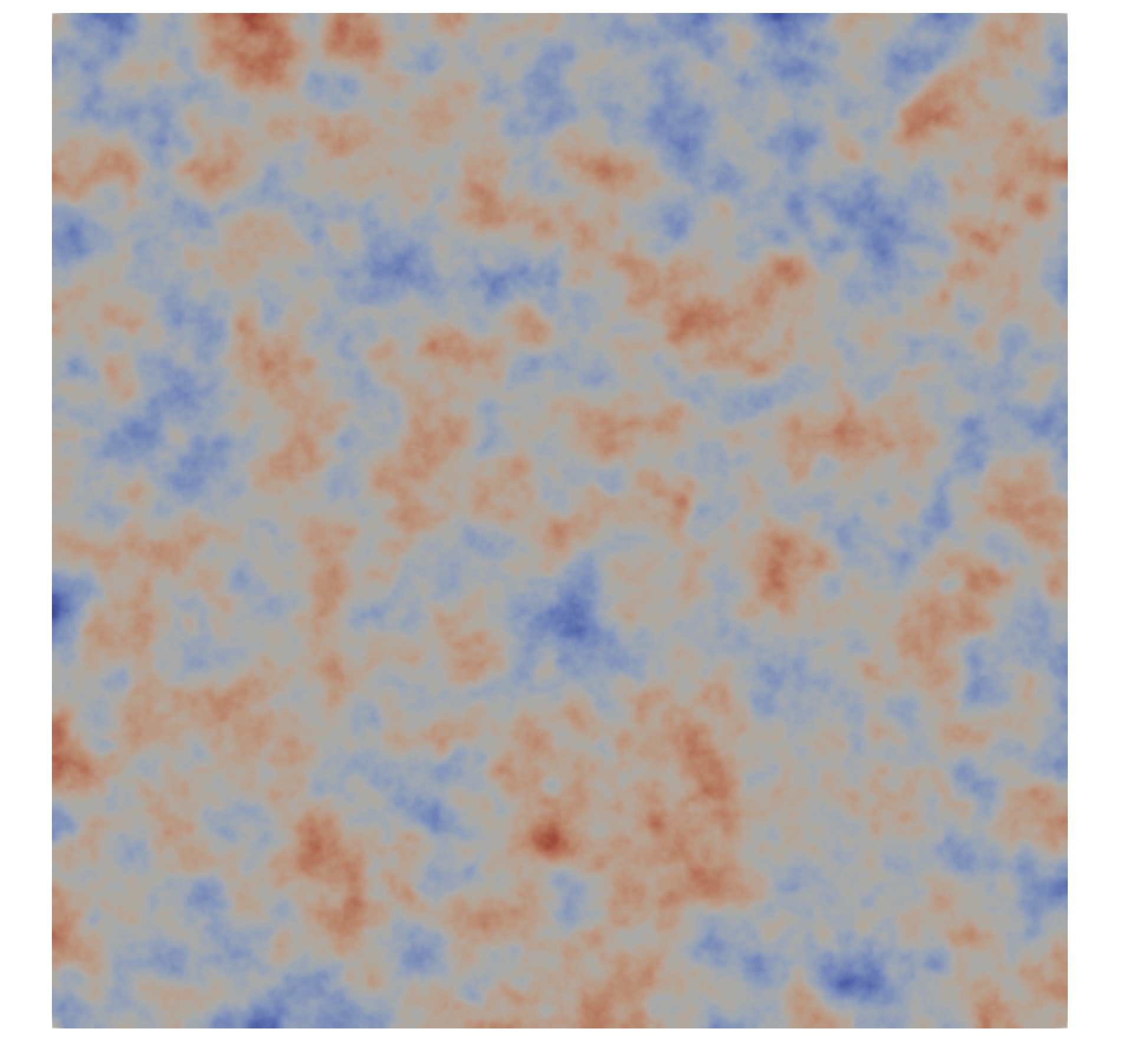}}
\subfigure[CTD, $t=0.2$]{\includegraphics[scale=0.11]{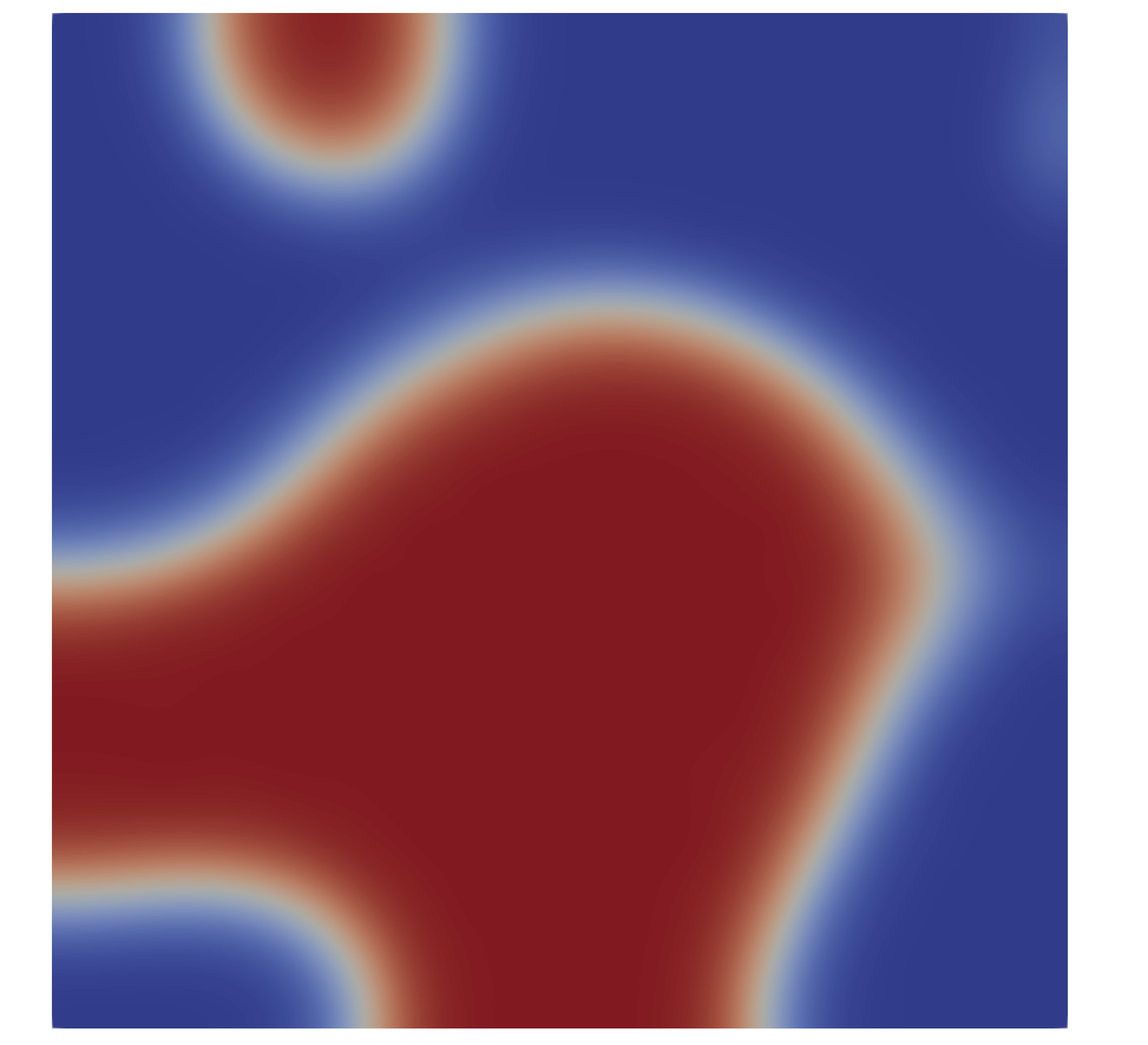}}
\subfigure[CTD, $t=2$ ]{\includegraphics[scale=0.11]{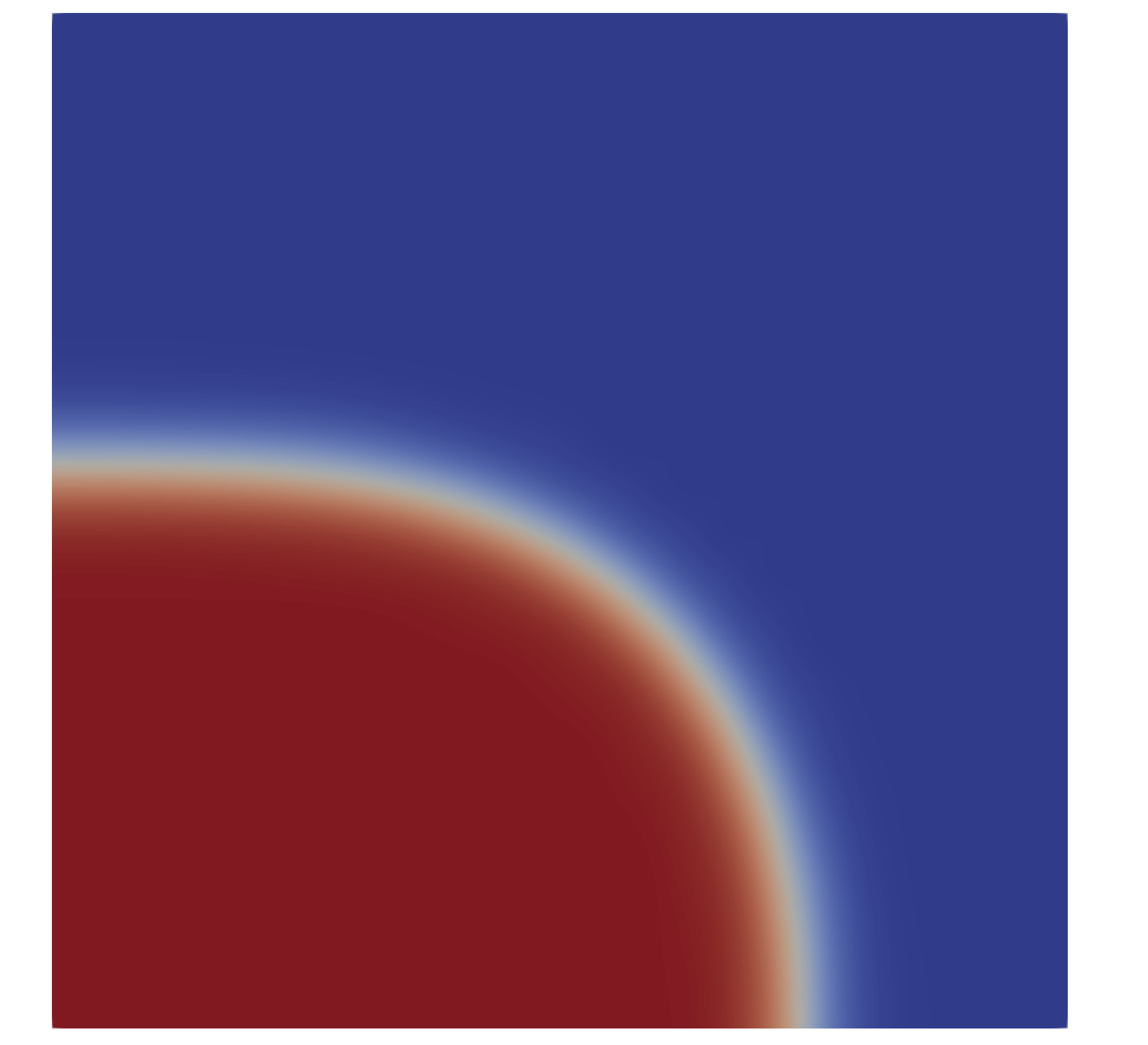}}
\subfigure[CTD, $t=10$ ]{\includegraphics[scale=0.11]{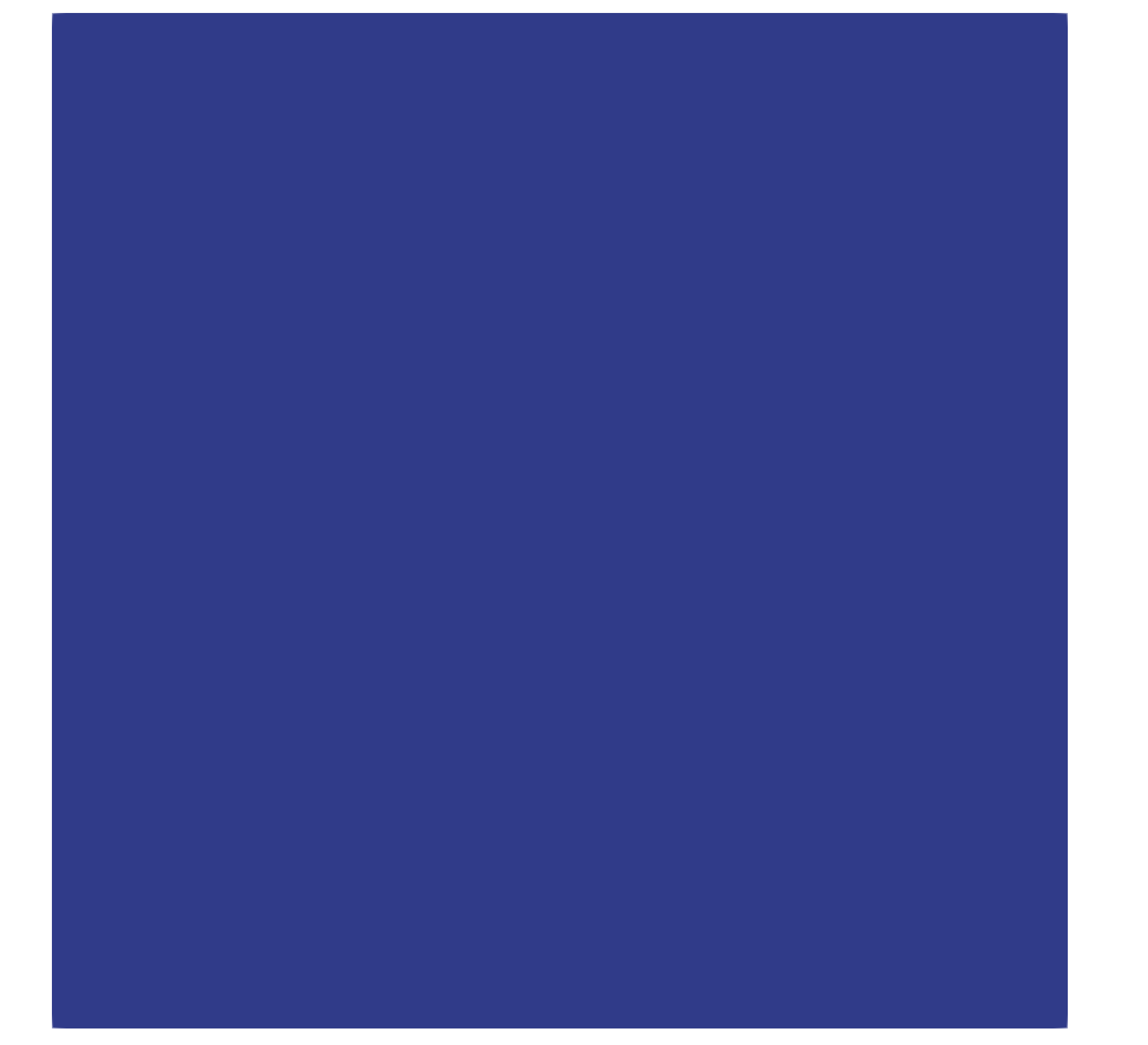}}
\caption{Case 1 - Solution snapshots for FE, CFE, and CTD at different time steps}
\label{fig:u2D}
\end{figure}

\begin{figure}[htbp]
\centering
\subfigure[$t=0.01$]{\includegraphics[scale=0.11]{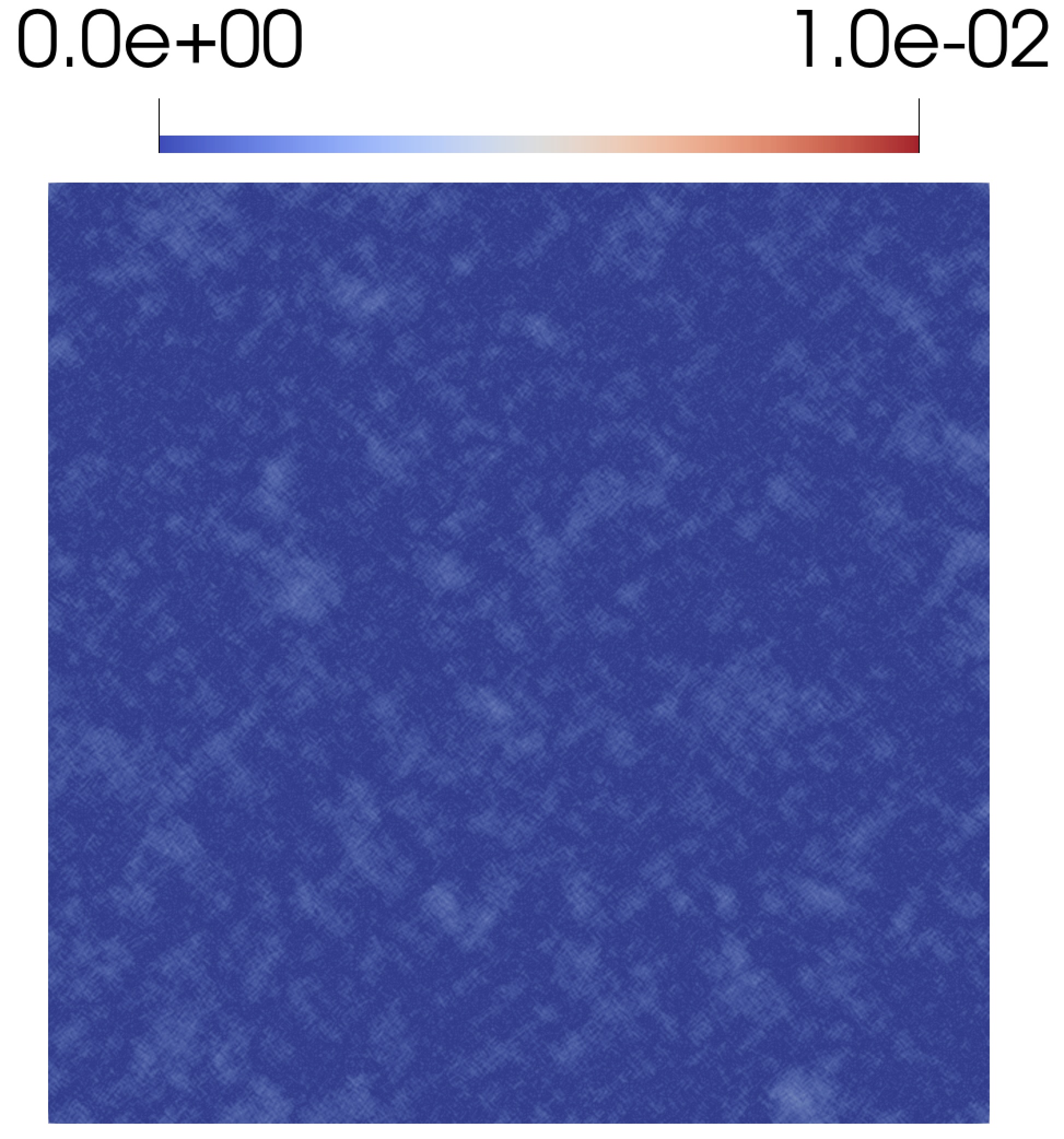}}
\subfigure[$t=0.2$]{\includegraphics[scale=0.11]{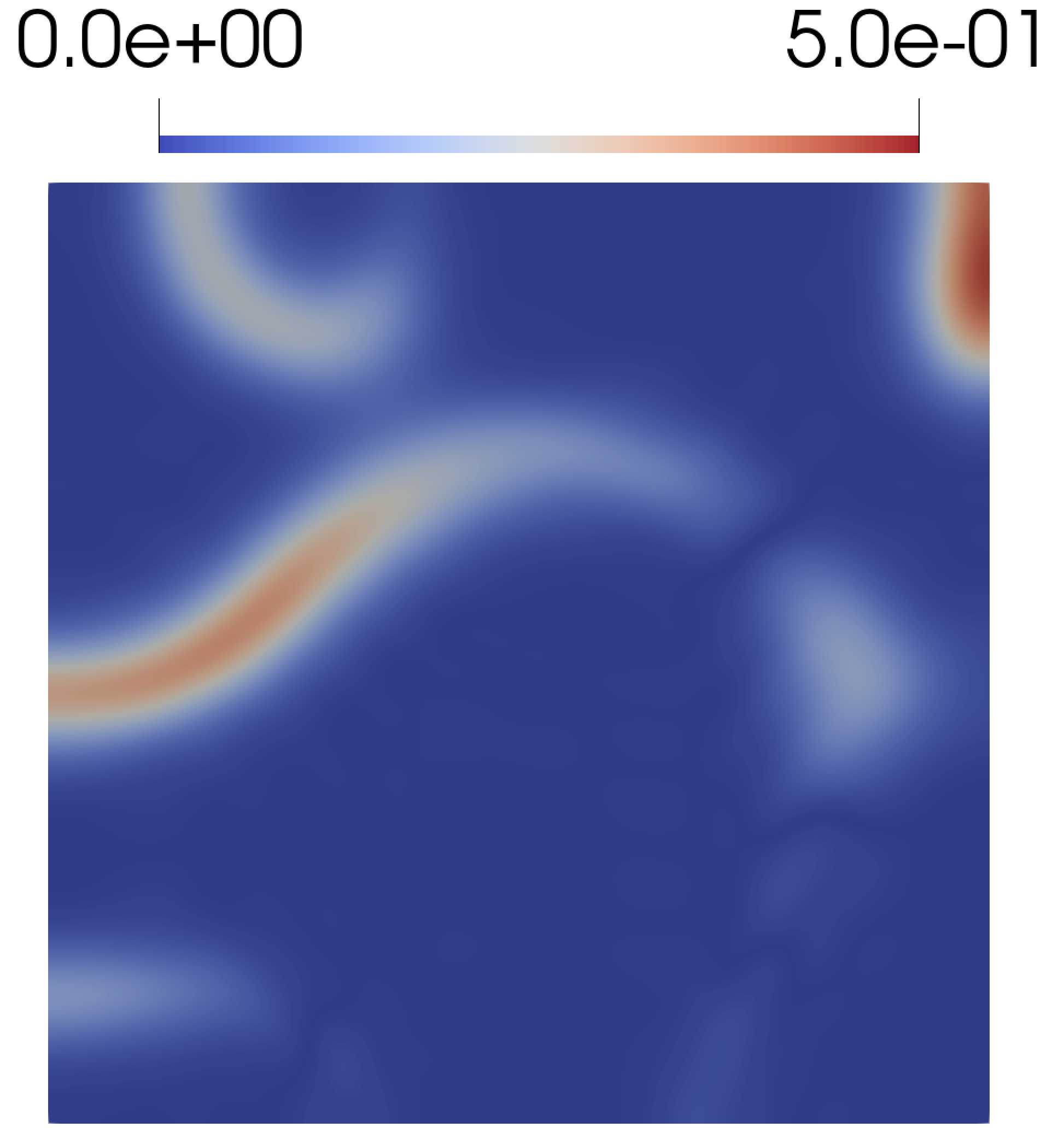}}
\subfigure[$t=2$ ]{\includegraphics[scale=0.11]{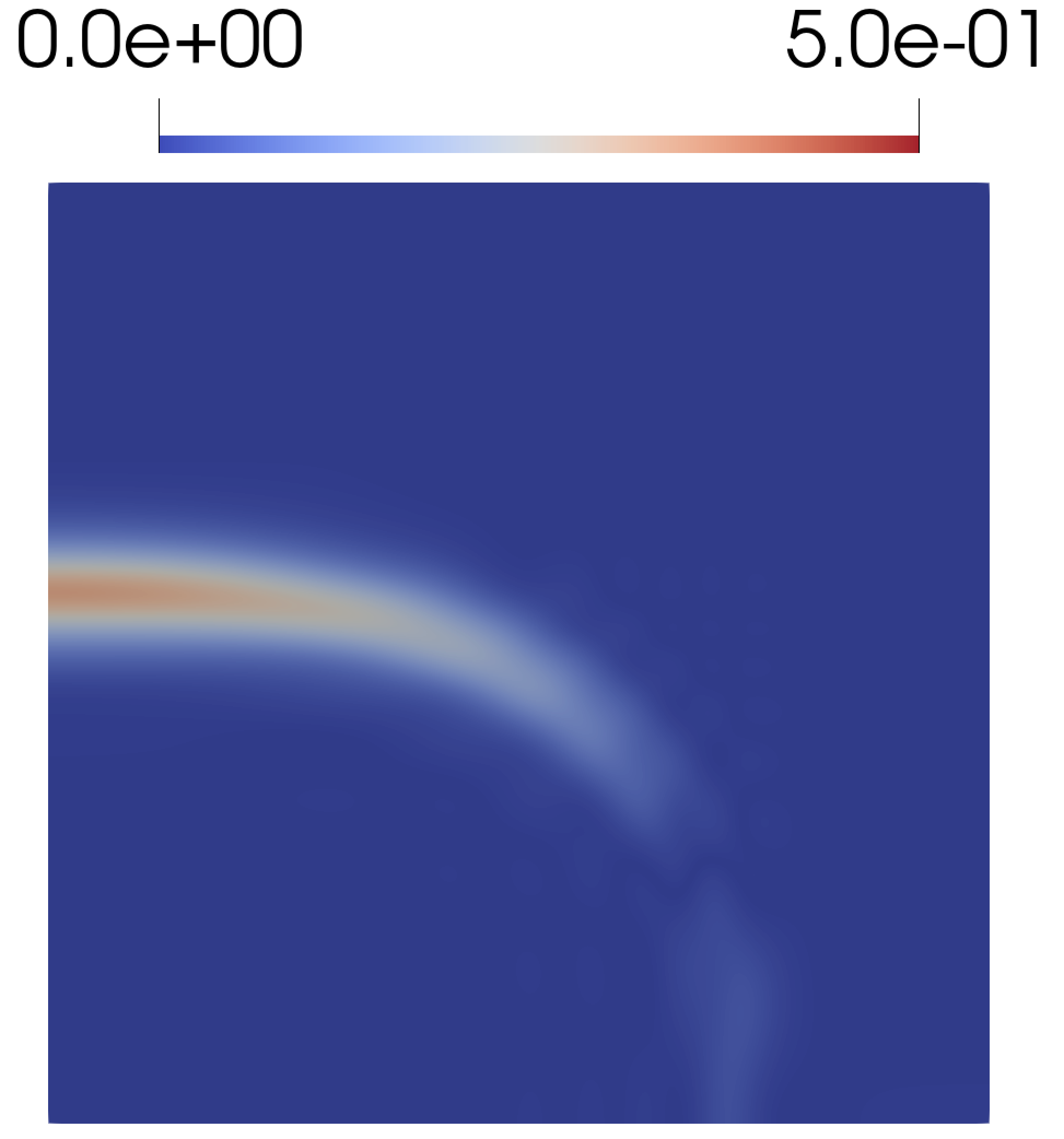}}
\subfigure[$t=10$ ]{\includegraphics[scale=0.11]{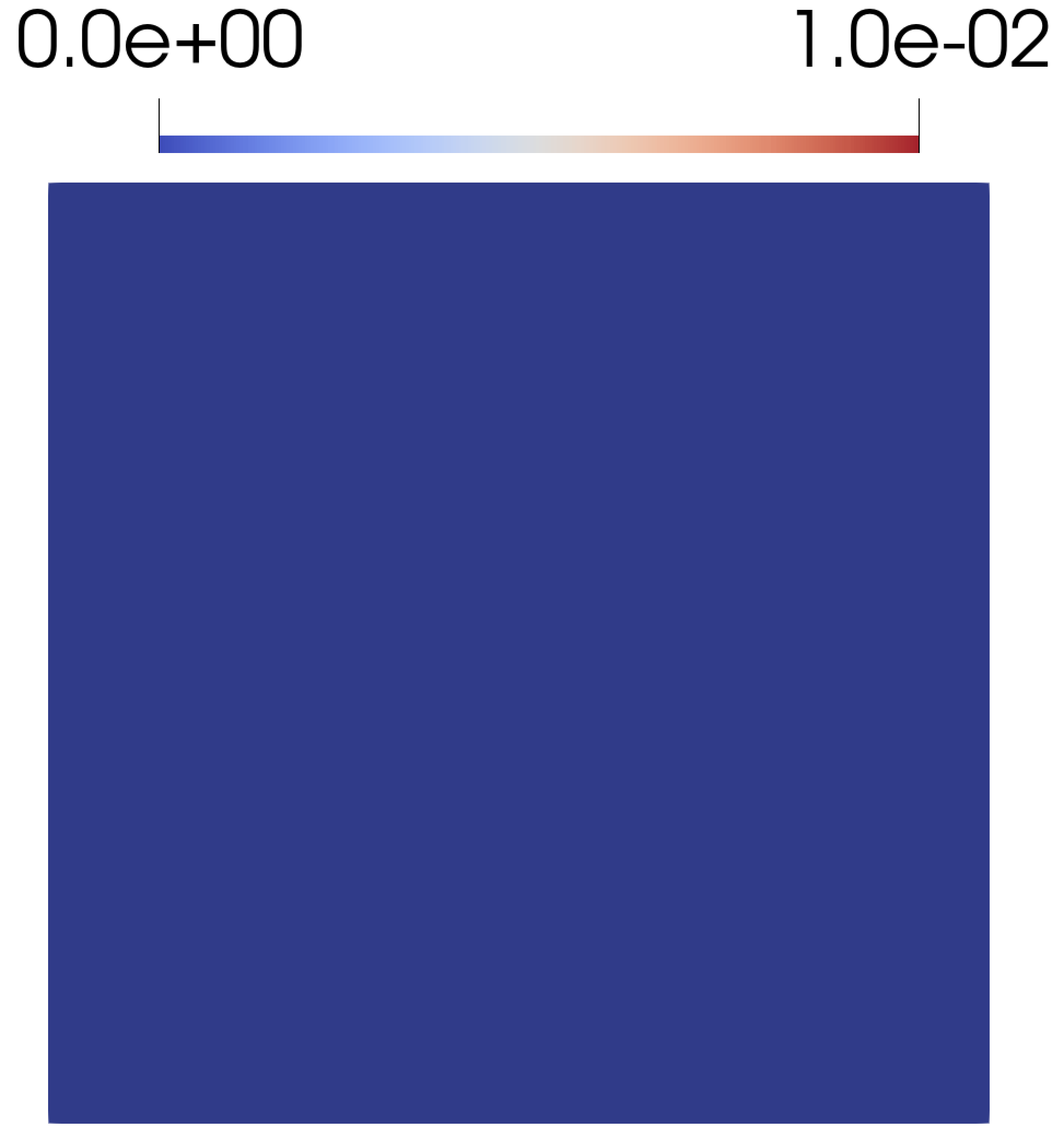}}
\caption{Case 1 - Difference between CTD and CFE, $|u^\text{CTD}-u^\text{CFE}|$}
\label{fig:du2D}
\end{figure}

Similarly, we study the accuracy of the CTD solutions for Case 2. The convolution parameters remain the same as previously for the CTD and CFE solutions. \figurename~\ref{fig:u2D_k2} illustrates the solutions at different time steps. Due to the smaller  gradient energy parameter $\kappa$, the solutions have a thinner transition interface between -1 and 1, leading to a phase evolution completely different from the previous one. The CTD method can capture well this change of solutions. Again, we illustrate the point-wise full field difference between CTD and CFE solutions, as shown in \figurename~\ref{fig:du2D_k2}. The accuracy of the CTD solutions is confirmed.

\begin{figure}[htbp]
\centering
\subfigure[FE, $t=0.01$]{\includegraphics[scale=0.11]{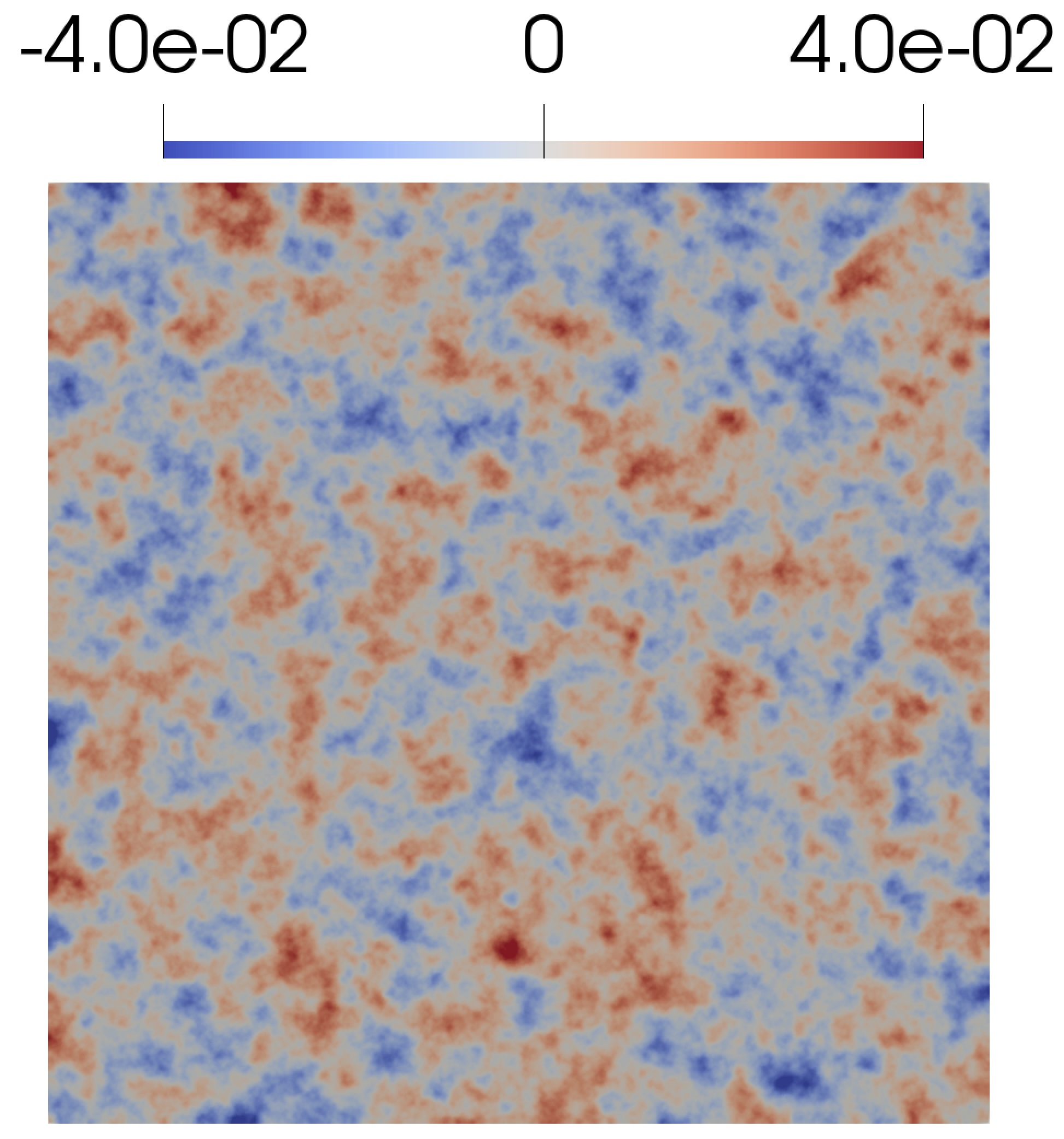}}
\subfigure[FE, $t=0.2$]{\includegraphics[scale=0.11]{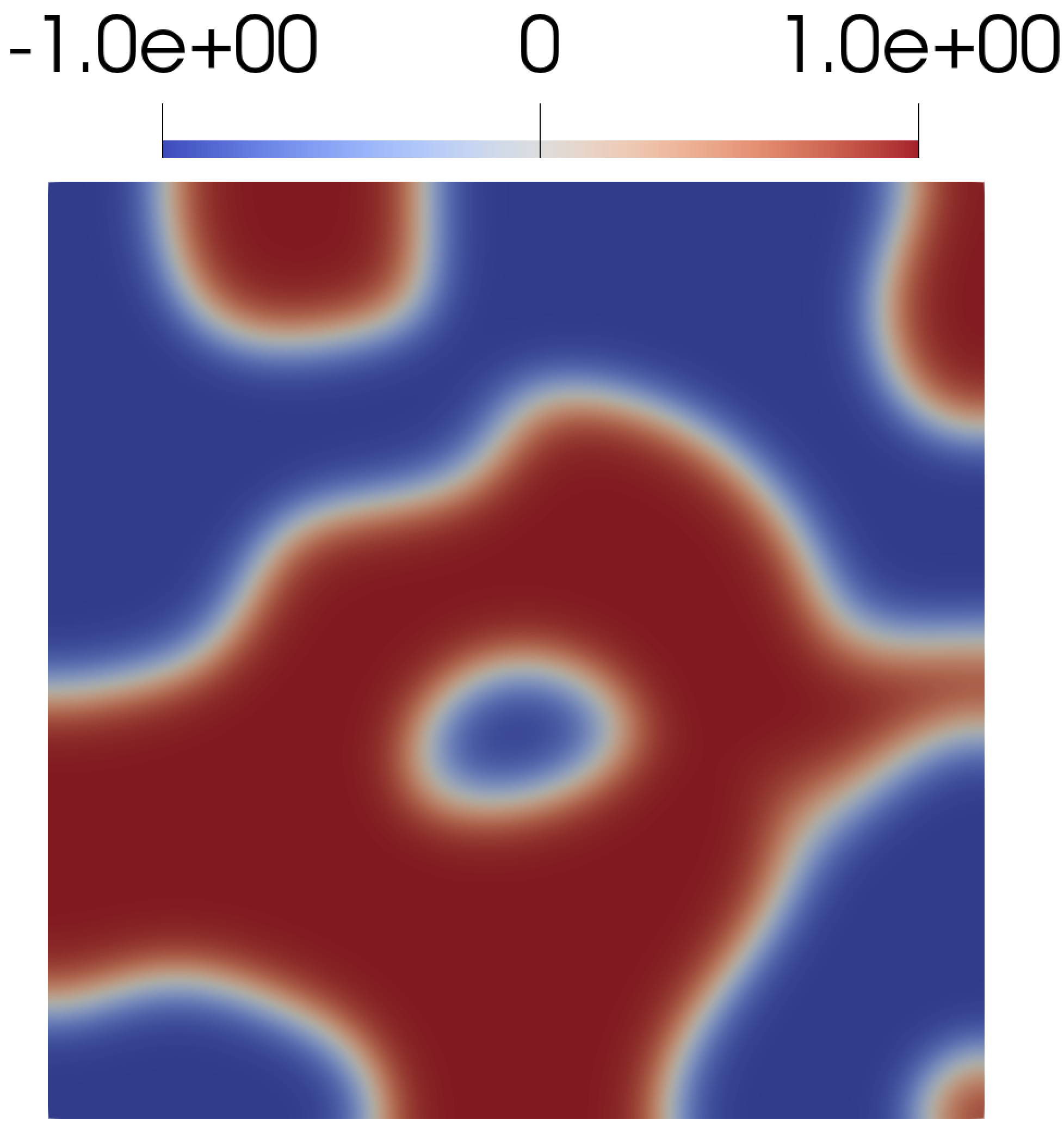}}
\subfigure[FE, $t=2$ ]{\includegraphics[scale=0.11]{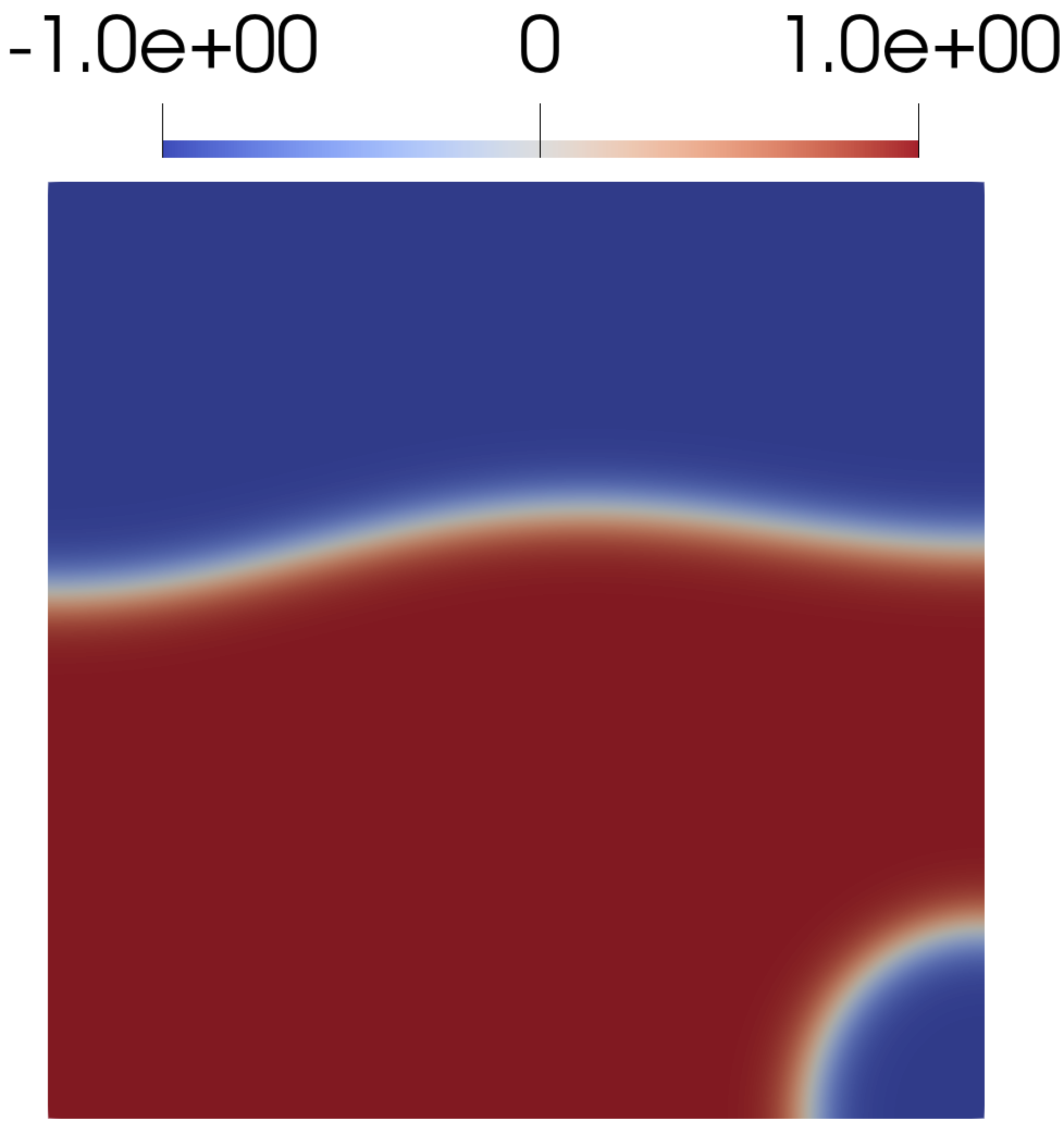}}
\subfigure[FE, $t=10$ ]{\includegraphics[scale=0.11]{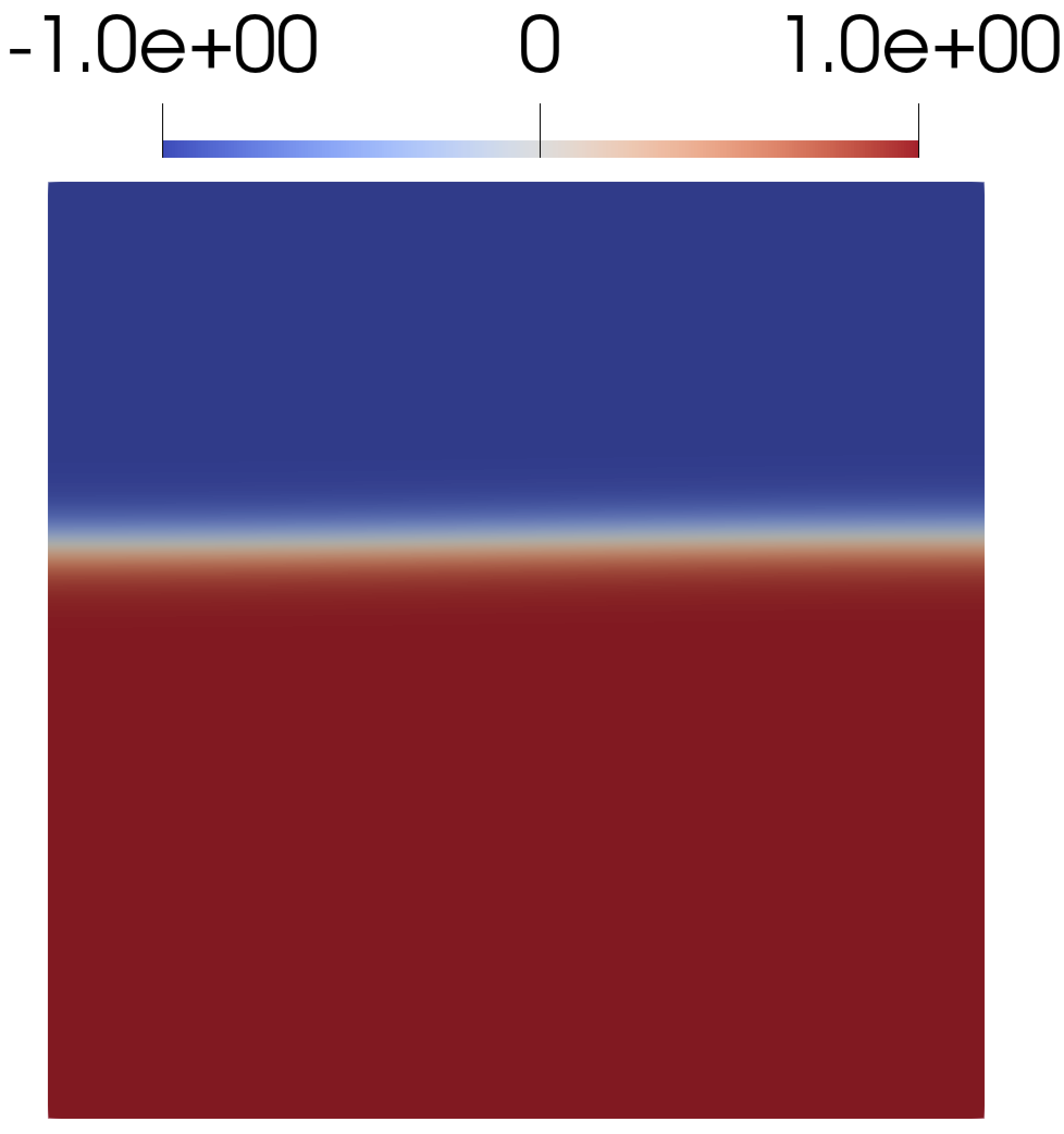}}\\
\subfigure[CFE, $t=0.01$]{\includegraphics[scale=0.11]{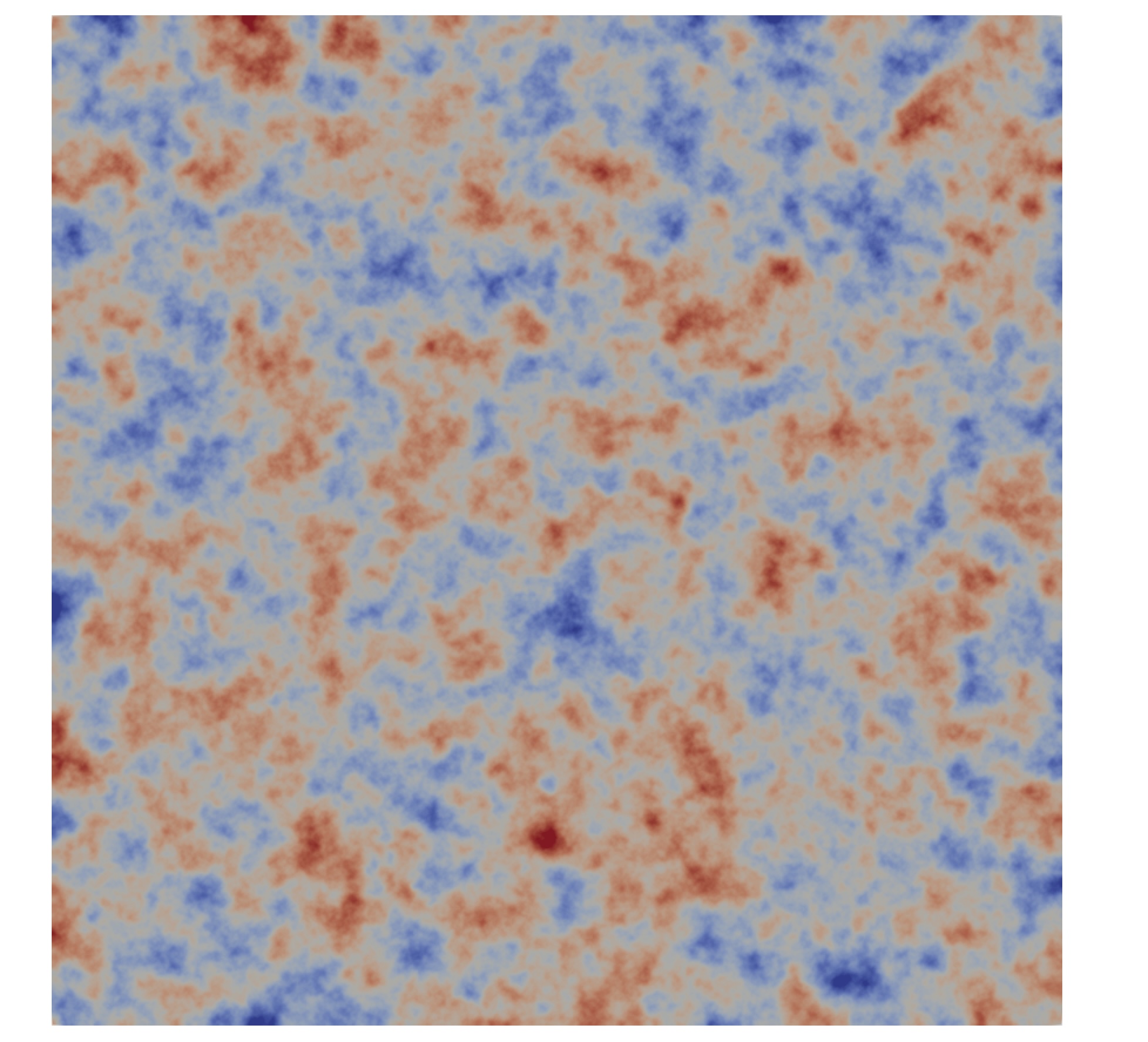}}
\subfigure[CFE, $t=0.2$]{\includegraphics[scale=0.11]{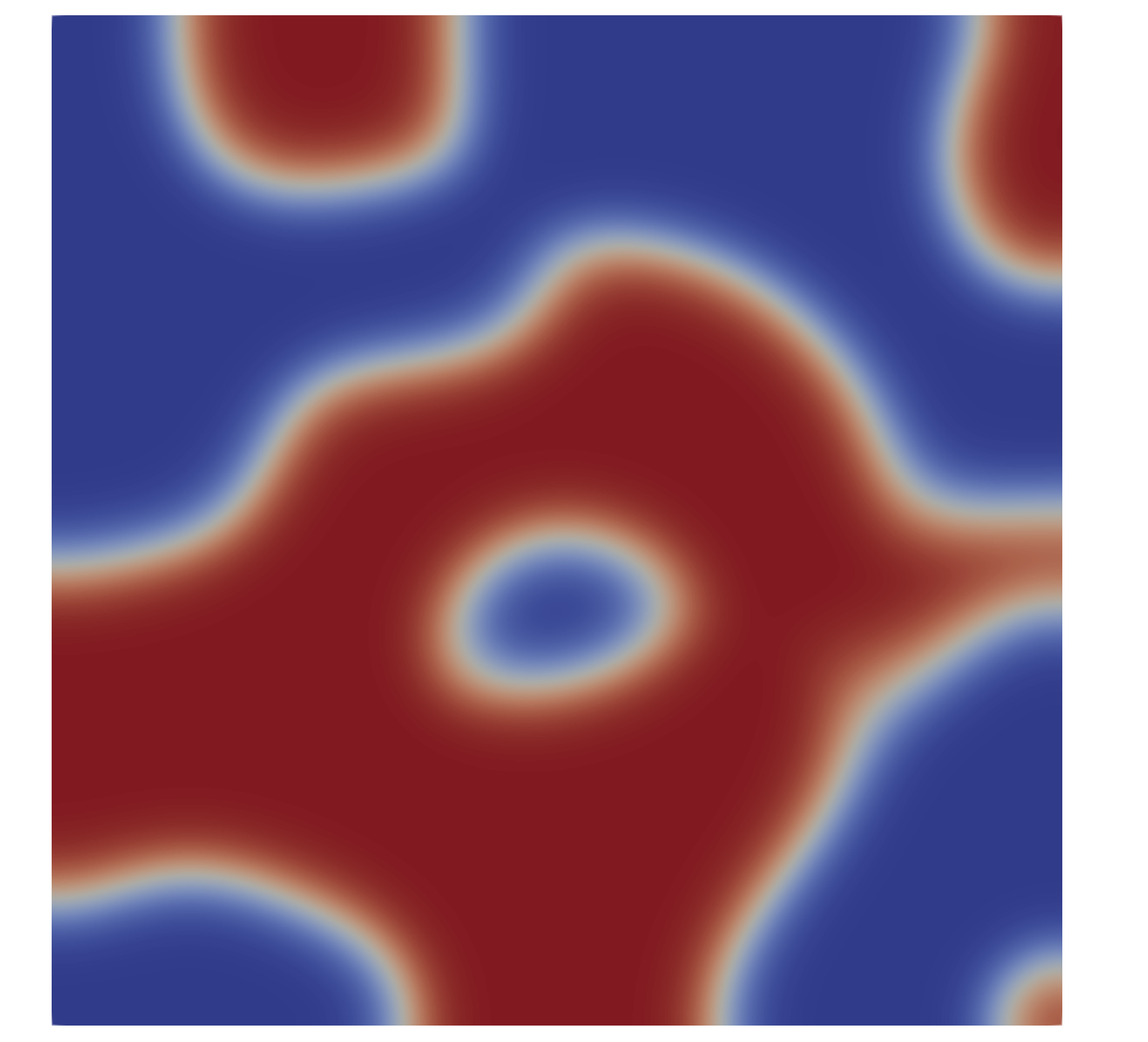}}
\subfigure[CFE, $t=2$ ]{\includegraphics[scale=0.11]{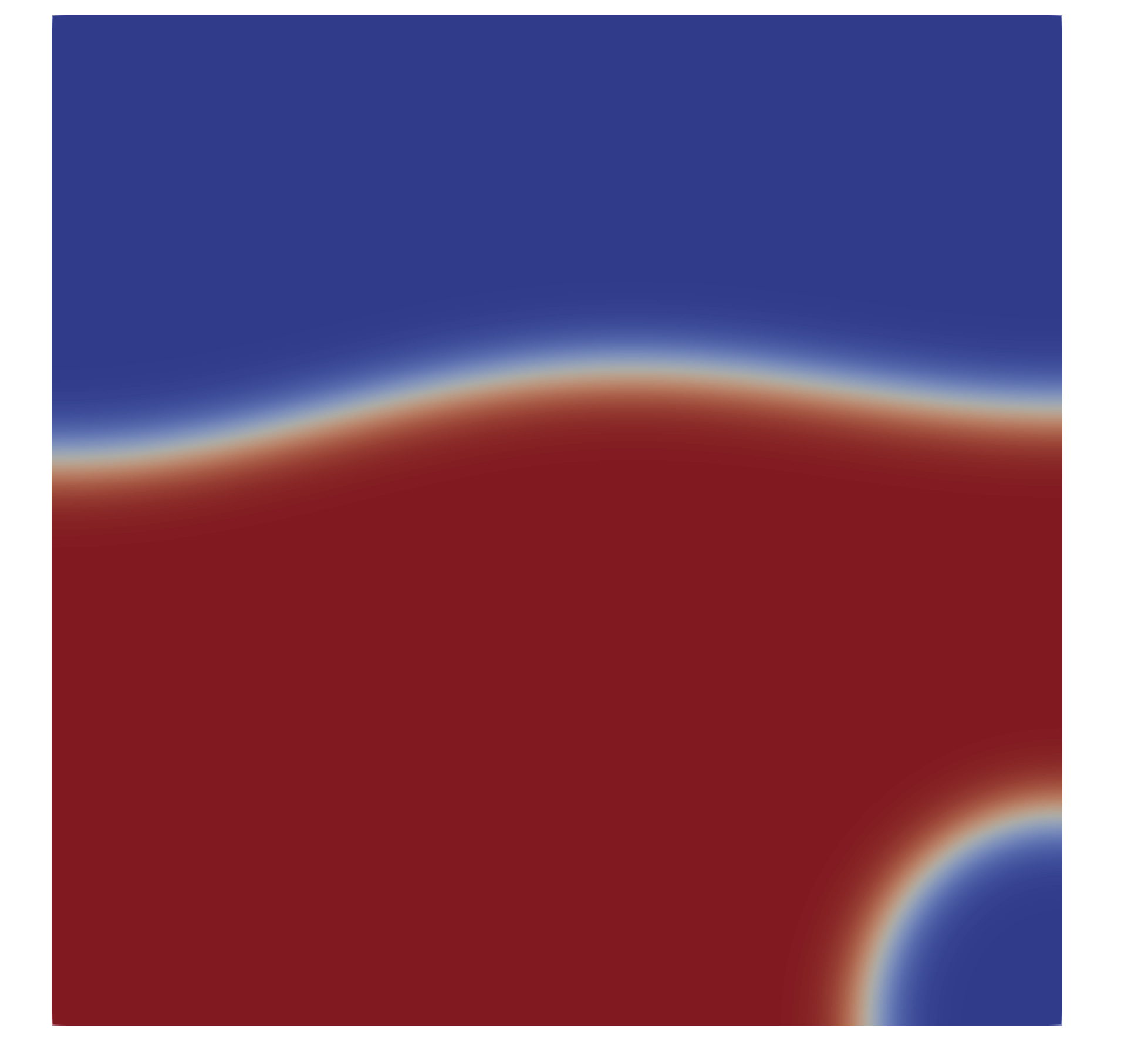}}
\subfigure[CFE, $t=10$ ]{\includegraphics[scale=0.11]{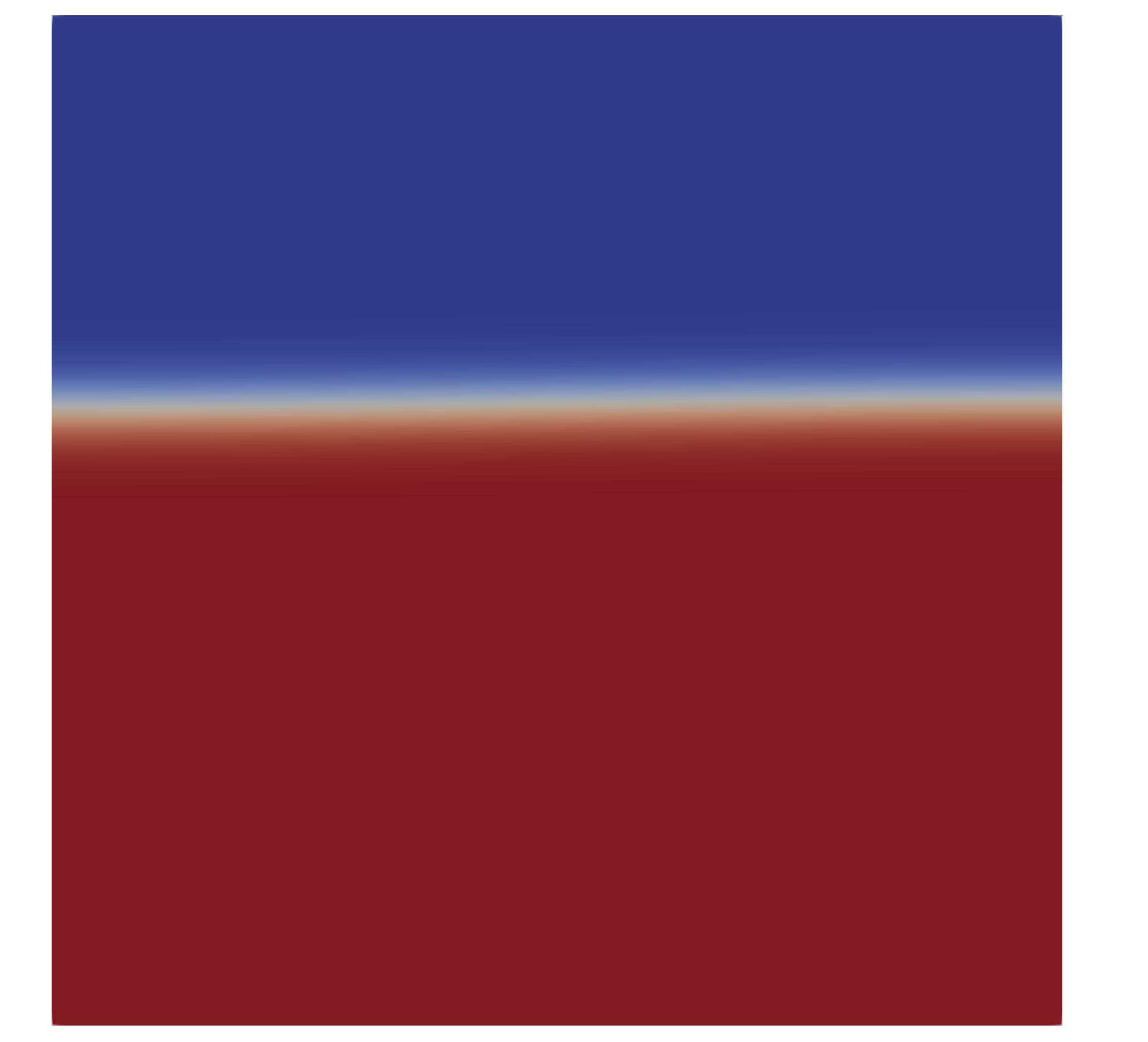}}\\
\subfigure[CTD, $t=0.01$]{\includegraphics[scale=0.11]{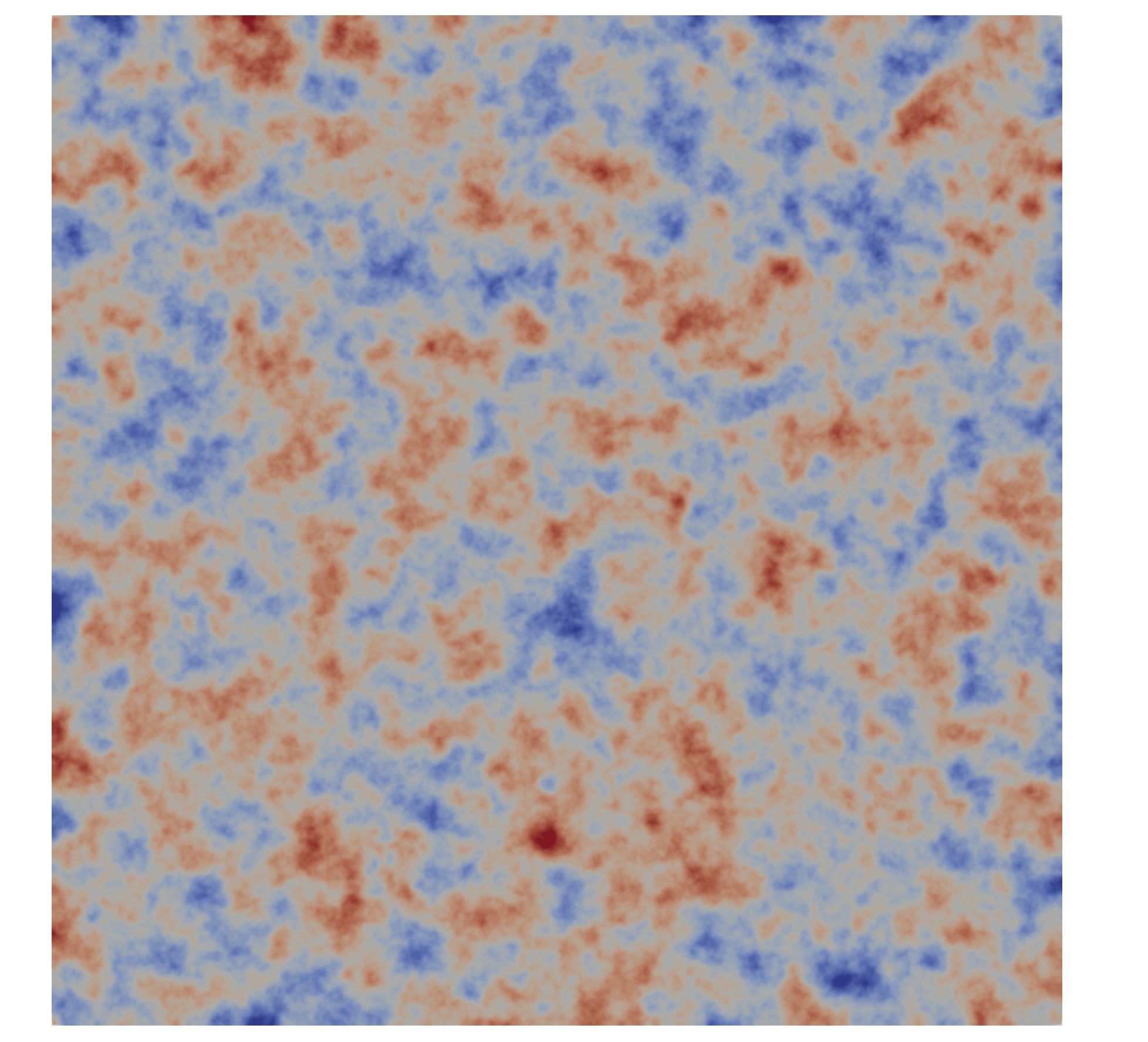}}
\subfigure[CTD, $t=0.2$]{\includegraphics[scale=0.11]{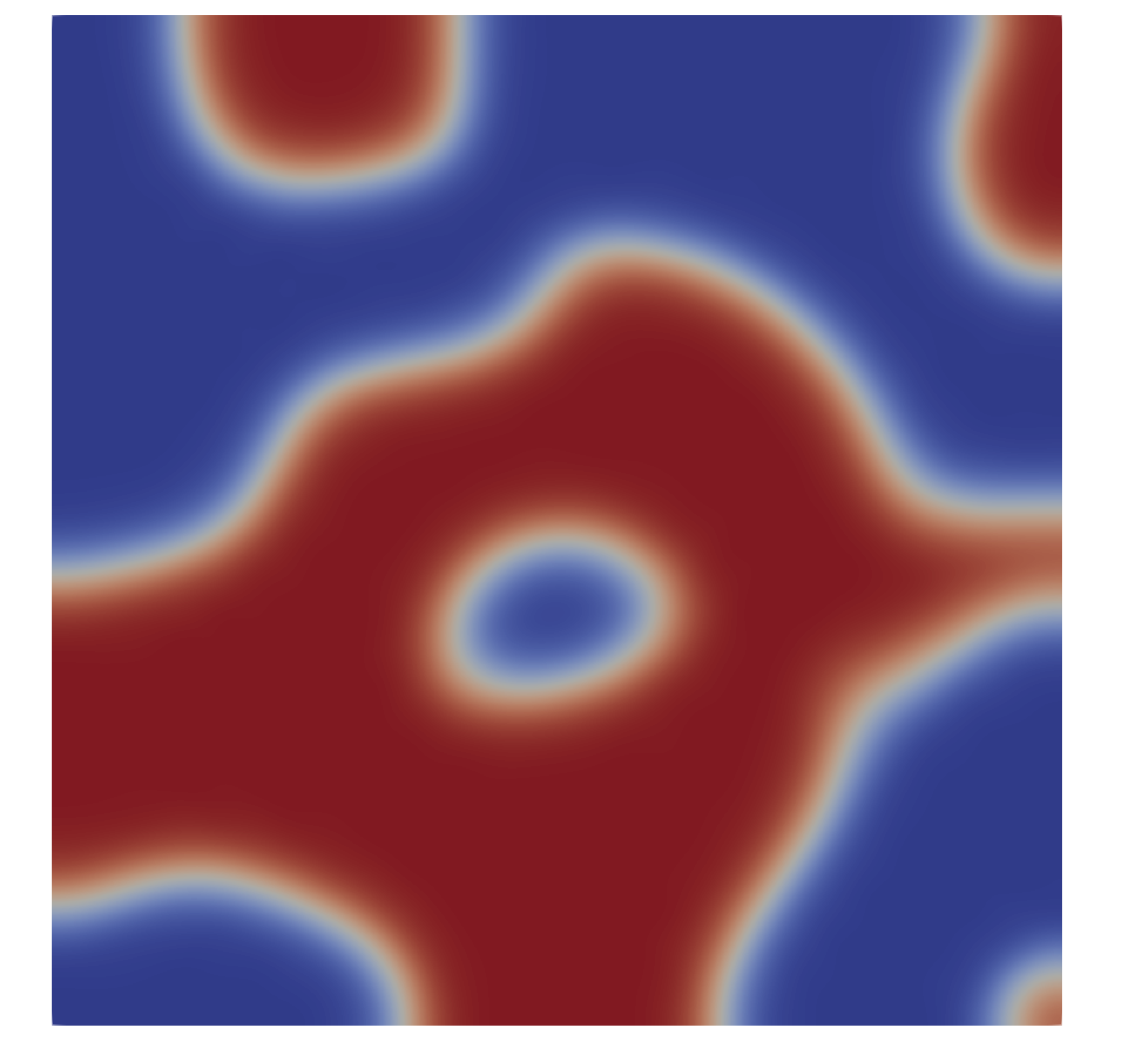}}
\subfigure[CTD, $t=2$ ]{\includegraphics[scale=0.11]{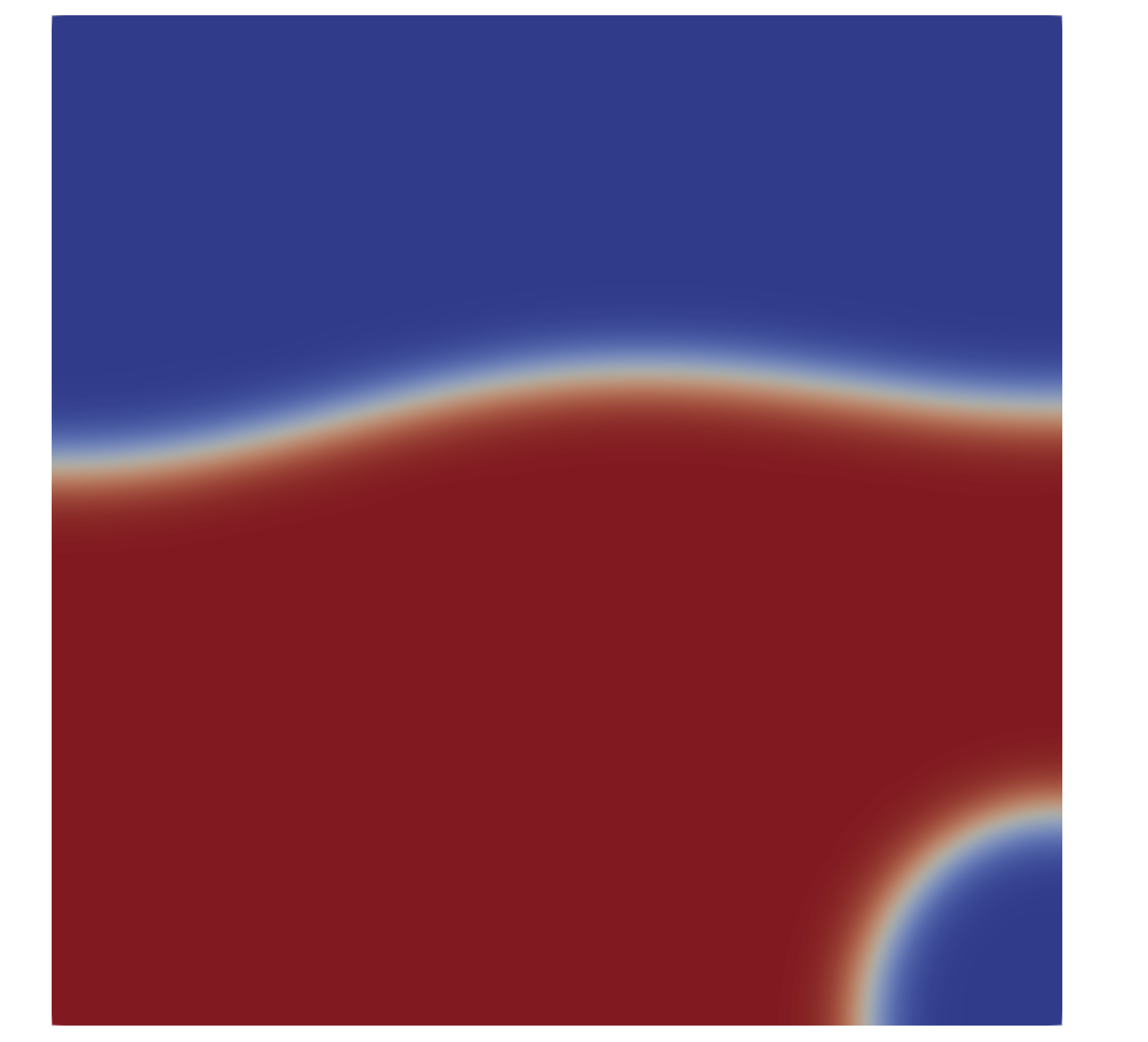}}
\subfigure[CTD, $t=10$ ]{\includegraphics[scale=0.11]{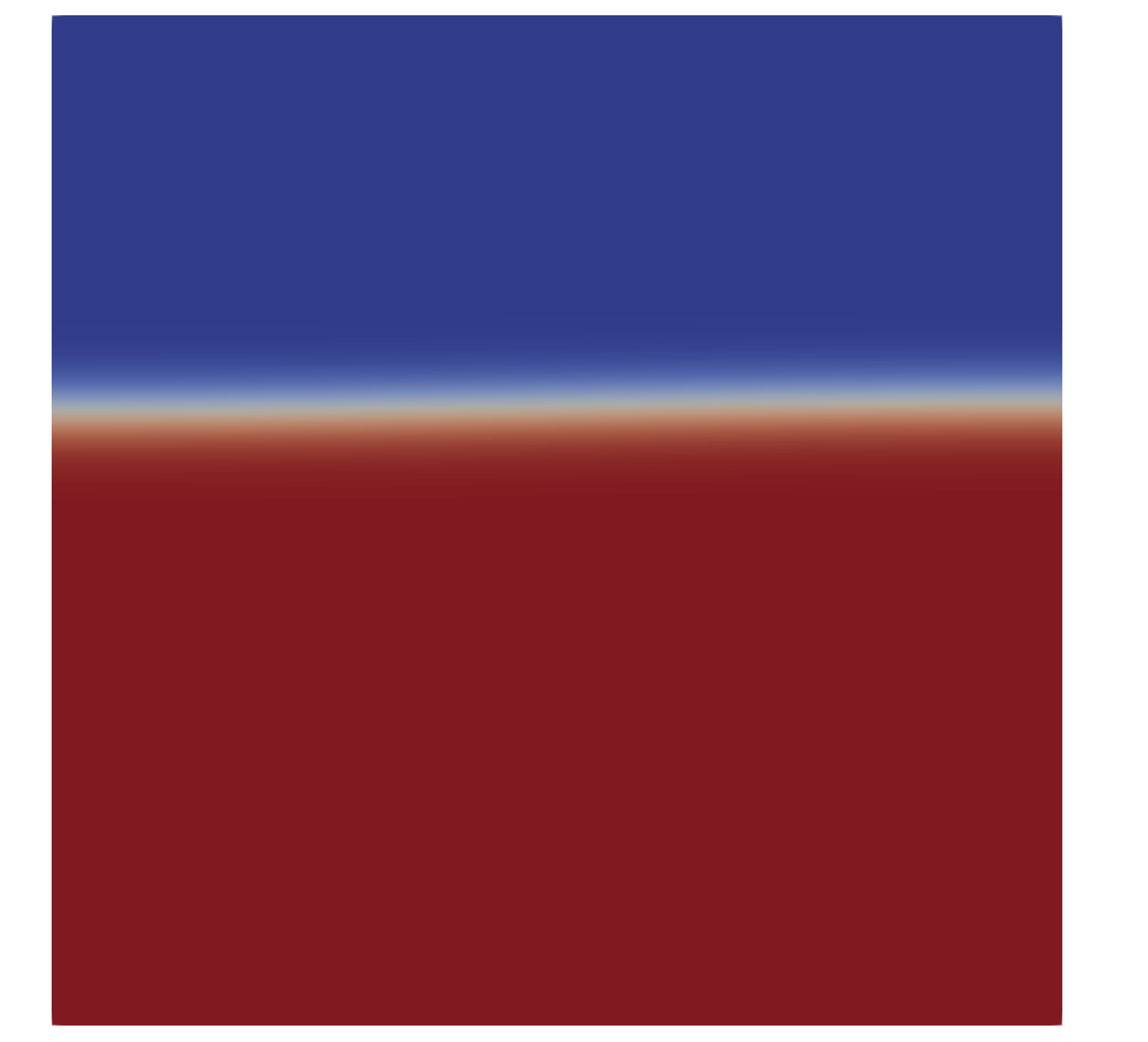}}
\caption{Case 2 - Solution snapshots of FE, CFE, and CTD at different time steps}
\label{fig:u2D_k2}
\end{figure}

\begin{figure}[htbp]
\centering
\subfigure[$t=0.01$]{\includegraphics[scale=0.11]{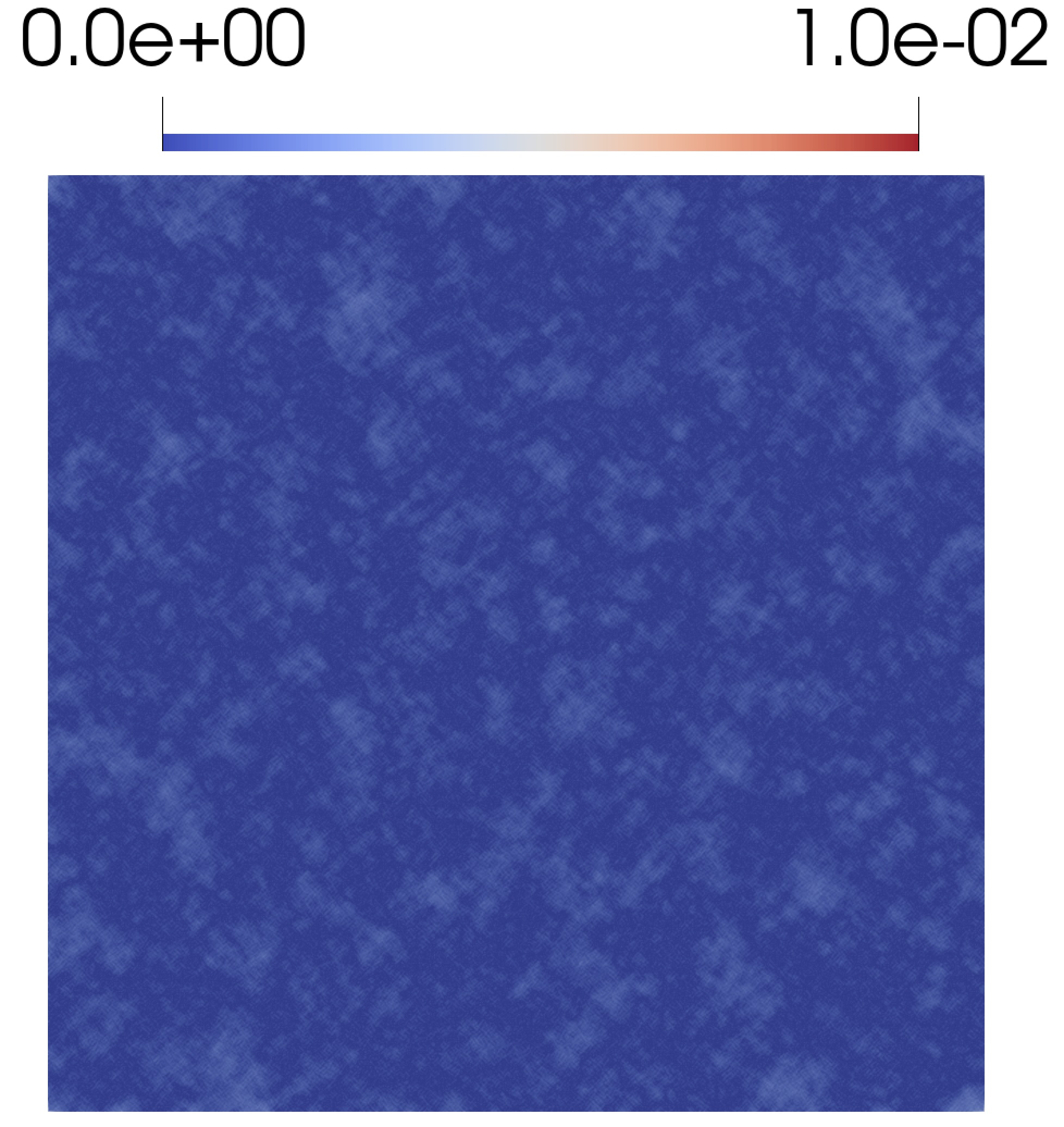}}
\subfigure[$t=0.2$]{\includegraphics[scale=0.11]{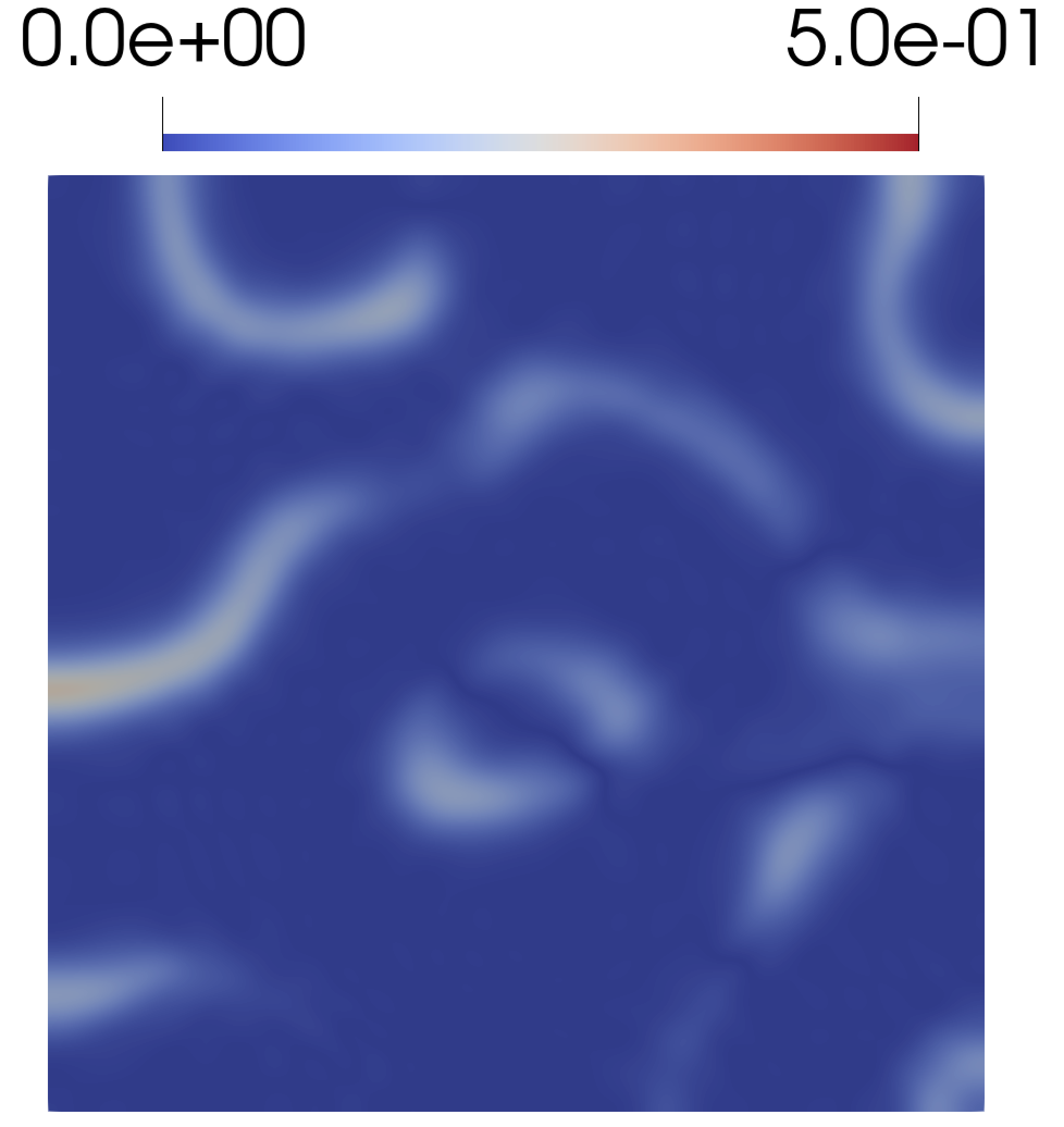}}
\subfigure[$t=2$ ]{\includegraphics[scale=0.11]{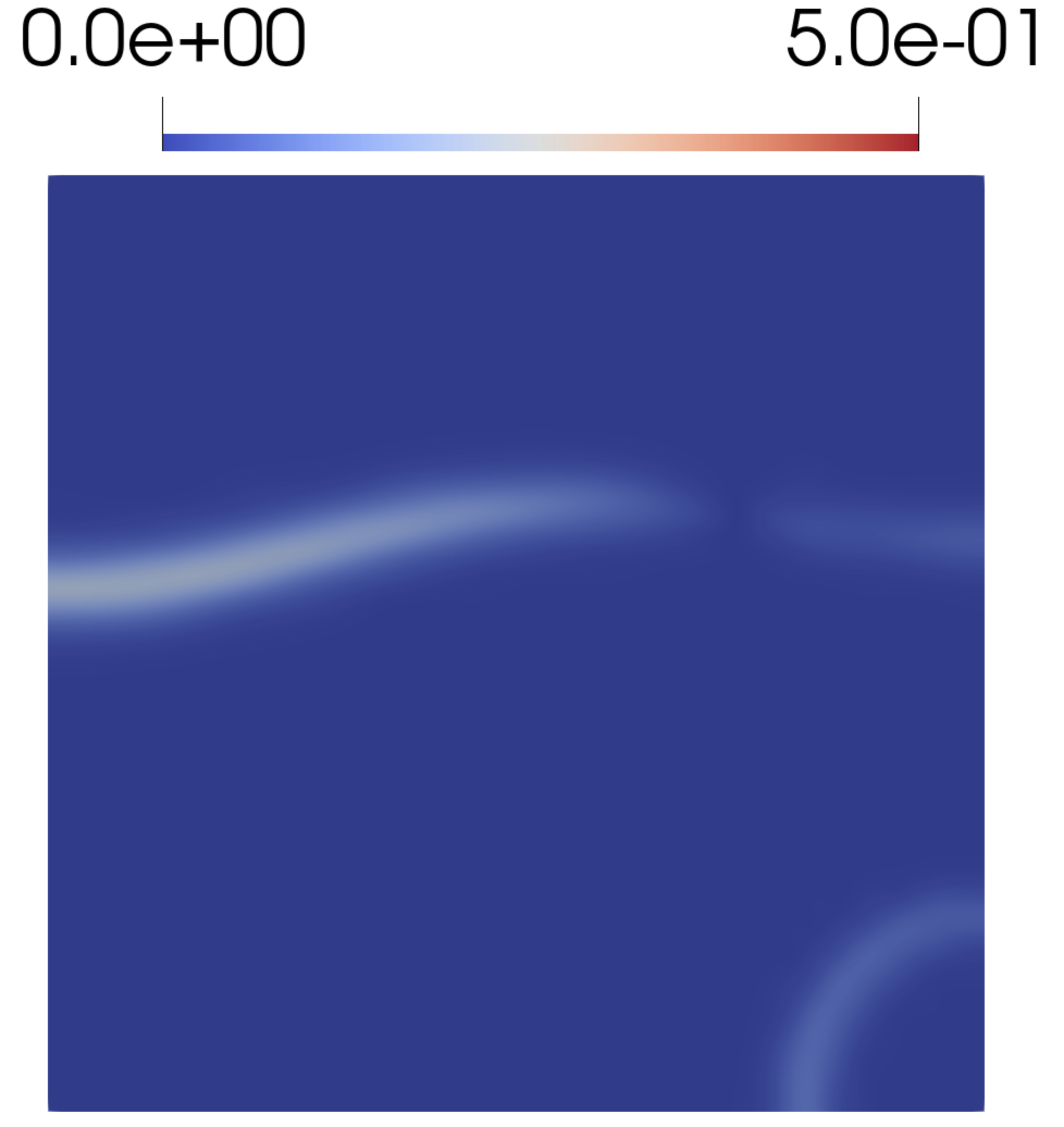}}
\subfigure[$t=10$ ]{\includegraphics[scale=0.11]{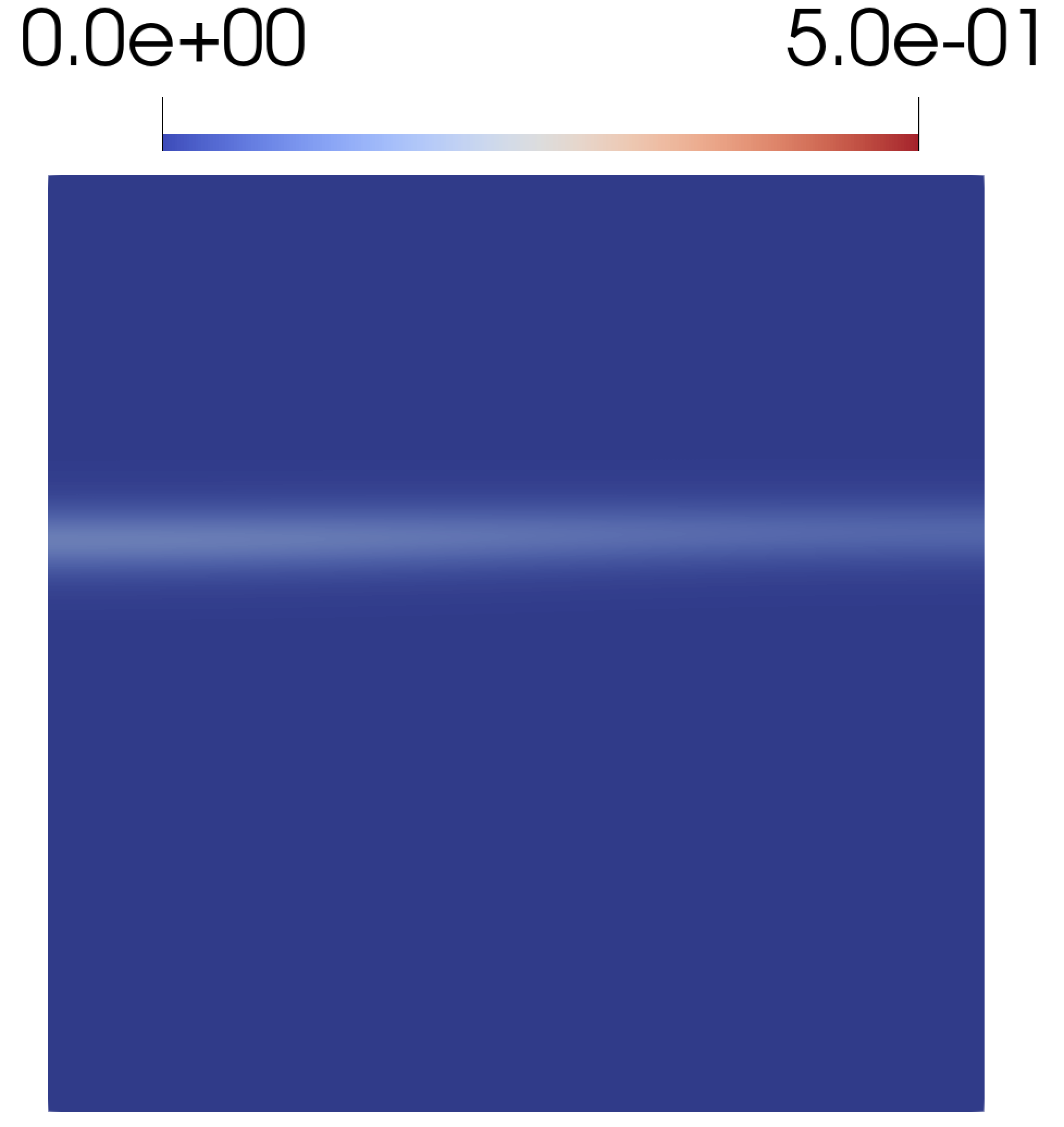}}
\caption{Case 2 - Difference between CTD and CFE, $|u^\text{CTD}-u^\text{CFE}|$}
\label{fig:du2D_k2}
\end{figure}

We remark that the accuracy of the CTD solutions can be further improved by increasing the number of modes for each time step.  {The CFE solution can be used as a reference to verify the accuracy of the CTD solution \cite{lu2023convolution}, since the same shape functions are used in CFE and CTD.} For illustration purposes, the evolution of the CTD solution (at $t=0.2$) with respect to an increasing number of modes is displayed in \figurename~\ref{fig:ctd_m}. We can see that the CTD converges to the CFE solution  {in} \figurename~\ref{fig:u2D_k2}(f) with the increased  {number of} modes $M$. In this example, 21 CTD modes seem sufficient to accurately represent the solution at $t=0.2$. The required number of modes $M$ can be varying for different time steps and is determined automatically by the solution algorithm, as mentioned in Section \ref{sect:solualgo}.  {Furthermore, we can observe that the error decreases relatively slowly after 15 modes in \figurename~\ref{fig:ctd_m}. This is due to the use of the fixed point algorithm, whose convergence rate is typically linear. This could be improved by adopting another solution algorithm (e.g., Newton-Raphson method) that has a higher order convergence rate.  Our experience is that the fixed point algorithm is able to provide satisfactory accuracy at a relatively low cost.}

\begin{figure}[htbp]
\centering
\subfigure[$M=1$]{\includegraphics[scale=0.11]{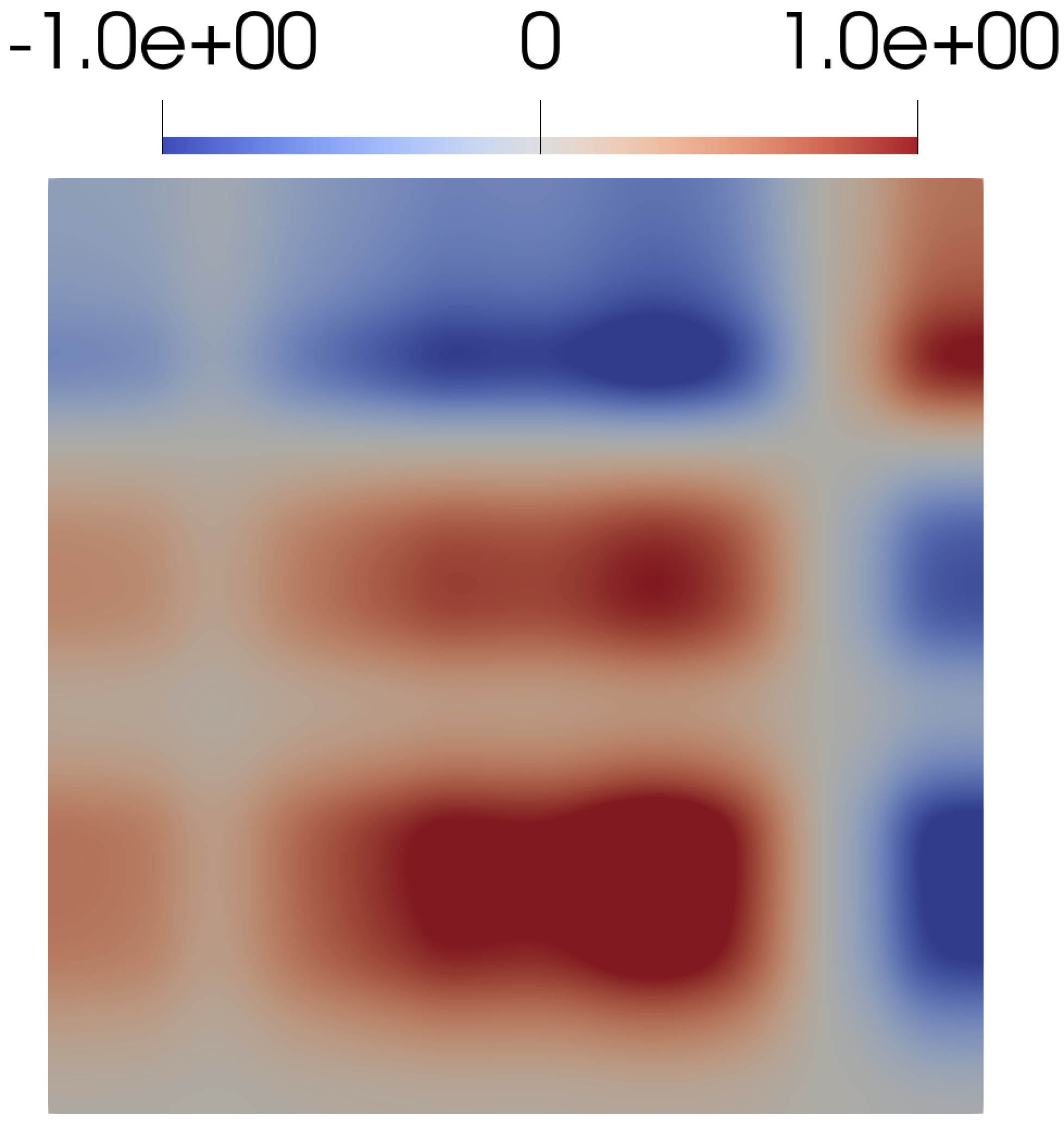}}
\subfigure[$M=5$]{\includegraphics[scale=0.11]{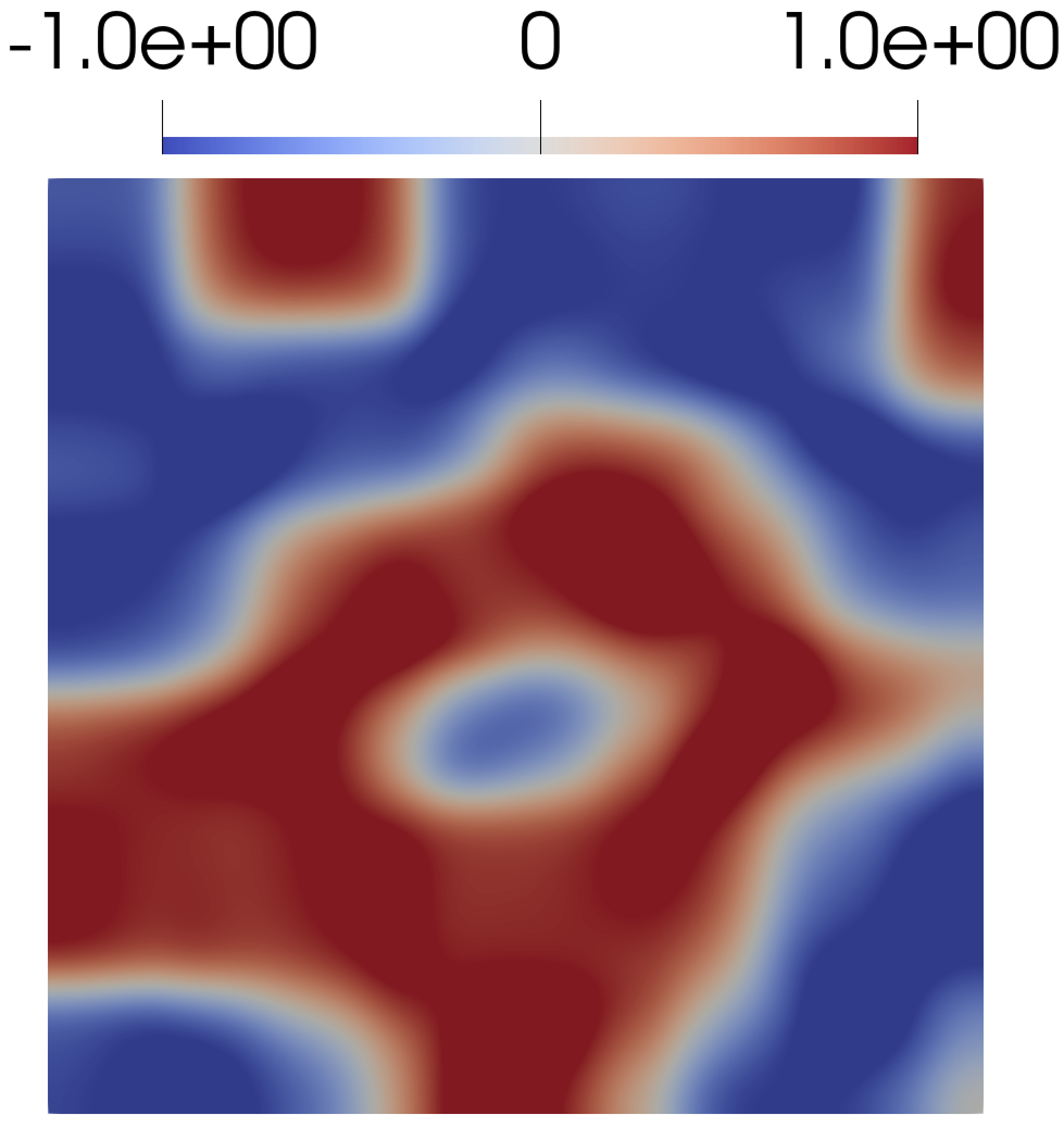}}
\subfigure[$M=10$ ]{\includegraphics[scale=0.11]{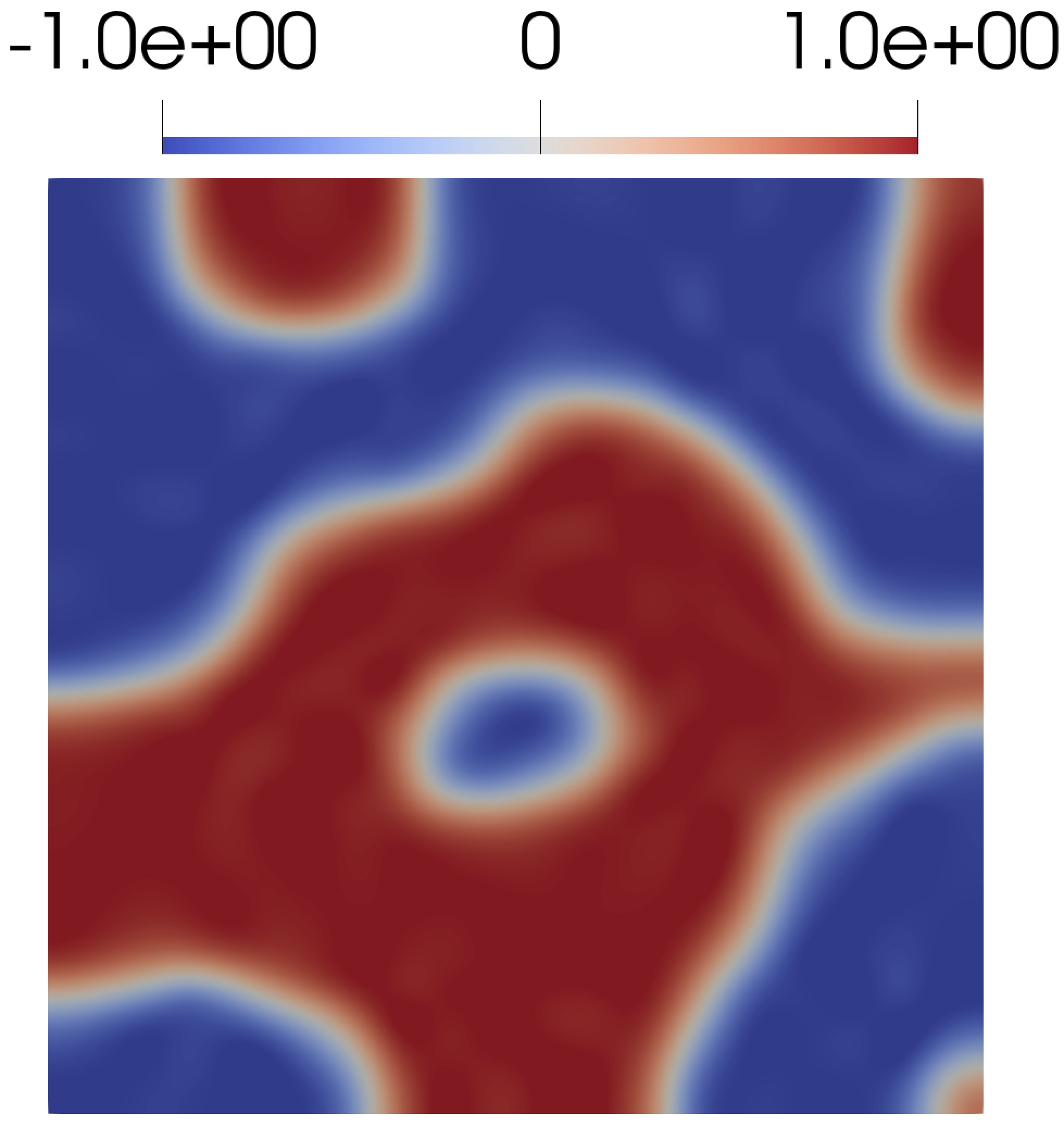}}
\subfigure[$M=21$ ]{\includegraphics[scale=0.11]{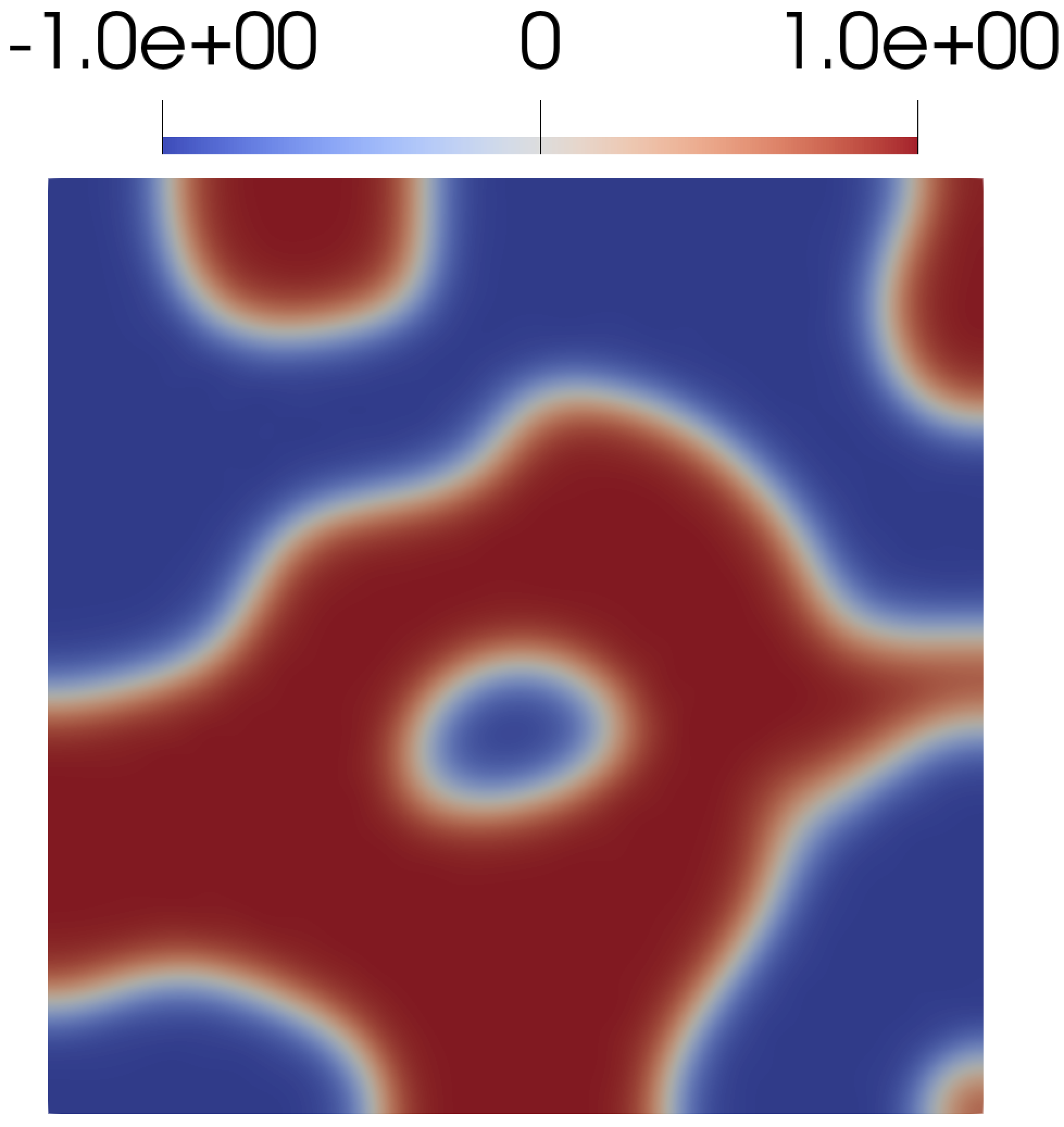}}
\subfigure[ {$\text{Error}=\frac{\|\boldsymbol{u}^\text{CTD}-\boldsymbol{u}^\text{CFE}\|_2}{\|\boldsymbol{u}^\text{CFE}\|_2}$} ]{\includegraphics[scale=0.15]{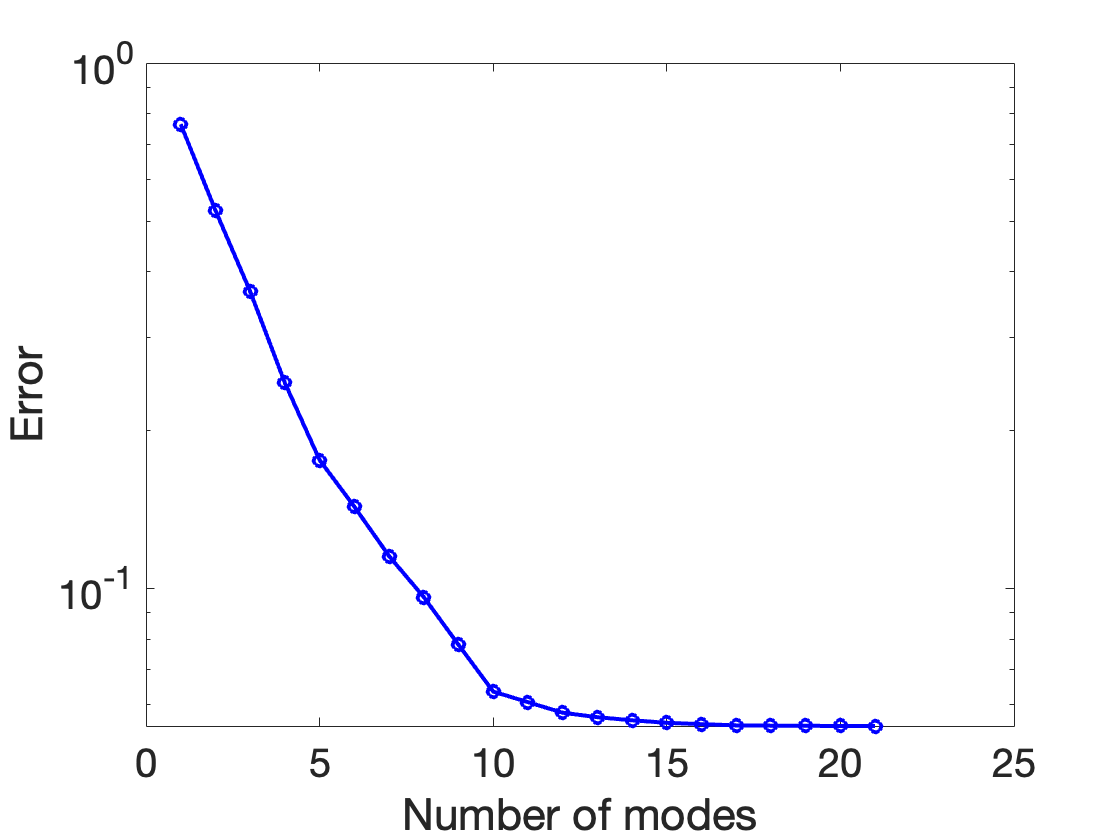}}
\caption{Case 2, $t=0.2$ - Evolution of the CTD solution with an increasing number of modes $M$}
\label{fig:ctd_m}
\end{figure}

Now, let us look at the computational speed of the CTD method, especially when dealing with high-resolution meshes. Three meshes are used here: $h=0.008, 0.004, 0.002$, which results in 391876, 1565001, 6255001 DoFs, respectively.  {The mesh information is summarized in \tablename~\ref{table:mesh}.} For the speed comparison, we use FE as the reference, as it is relatively more commonly used than CFE. We exclude the computational cost related to computing the coefficients of element matrices and the assembly process, since this can be easily minimized by leveraging parallel computing.  {To give an idea about this matrix cost, \tablename~\ref{table:matrixcost} summarizes the average cost per time step for matrix coefficients and assembly without the use of parallel computing.} To make FE more efficient with large size problems, an iterative solver, i.e., the preconditioned conjugate gradients (PCG) method pre-implemented in MATLAB \cite{MATLAB}  {with the default setting and a stopping criterion $10^{-5}$}, is used for solving the discrete problems. All the computations are performed in MATLAB on a workstation with the processor Intel Xeon W9-3495X.

\begin{table}[htbp]
\caption{ {The three meshes used for testing the computational speed}}
\centering
\begin{tabular}{|c|c|c|c|}
\hline
 Element size  & Mesh size & Total elements   & Total DoFs \\ \hline
$h=0.008$ & $625 \times 625$ &  390625  & 391876 \\ \hline
$h=0.004$ & $1250 \times 1250$  & 1562500 & 1565001  \\ \hline
$h=0.002$ & $2500 \times 2500$ &  6250000  & 6255001 \\ \hline
\end{tabular}\\
\label{table:mesh}
\end{table}

\begin{table}[htbp]
\caption{ {Average cost per time step for matrix coefficients and assembly}}
\centering
\begin{tabular}{|c|c|c|c|}
\hline
 Element size  & FE & CTD & Ratio  \\ \hline
$h=0.008$ & 49s&  4.3s & 11.3 \\ \hline
$h=0.004$ & 172s  & 17s & 10.1 \\ \hline
\end{tabular}\\
\label{table:matrixcost}
\end{table}

\figurename~\ref{fig:cpucost2D} summarizes the computational costs for the different cases. As shown in \figurename~\ref{fig:cpucost2D}(a), CTD is much faster than FE on a high-resolution mesh. In the present examples, a speedup 100$\times$ was found for a mesh with a few millions DoFs. This speedup can be even higher when the mesh becomes larger and larger, as the cost of traditional FE grows  {much faster than that of CTD}. This demonstrates the promising capability of the CTD method for high-resolution AC problems. Focusing on a mesh with $h=0.002$ (6255001 DoFs), we can look at the cost per time step for FE and CTD, as shown in \figurename~\ref{fig:cpucost2D}(b). In this example, CTD outperformed FE, except for the first time step. This is due to the relatively large number of modes required to represent the solution at $t=0.01$. We can see that the computational efficiency of CTD is strongly correlated with the number of modes $M$ for each time step, by comparing \figurename~\ref{fig:cpucost2D}(b) and \figurename~\ref{fig:cpucost2D}(c). Fortunately, the required number of modes remains very small for most of the time steps, which ensures the efficiency of method.

\begin{figure}[htbp]
\centering
\subfigure[Cost vs Mesh DoFs]{\includegraphics[scale=0.135]{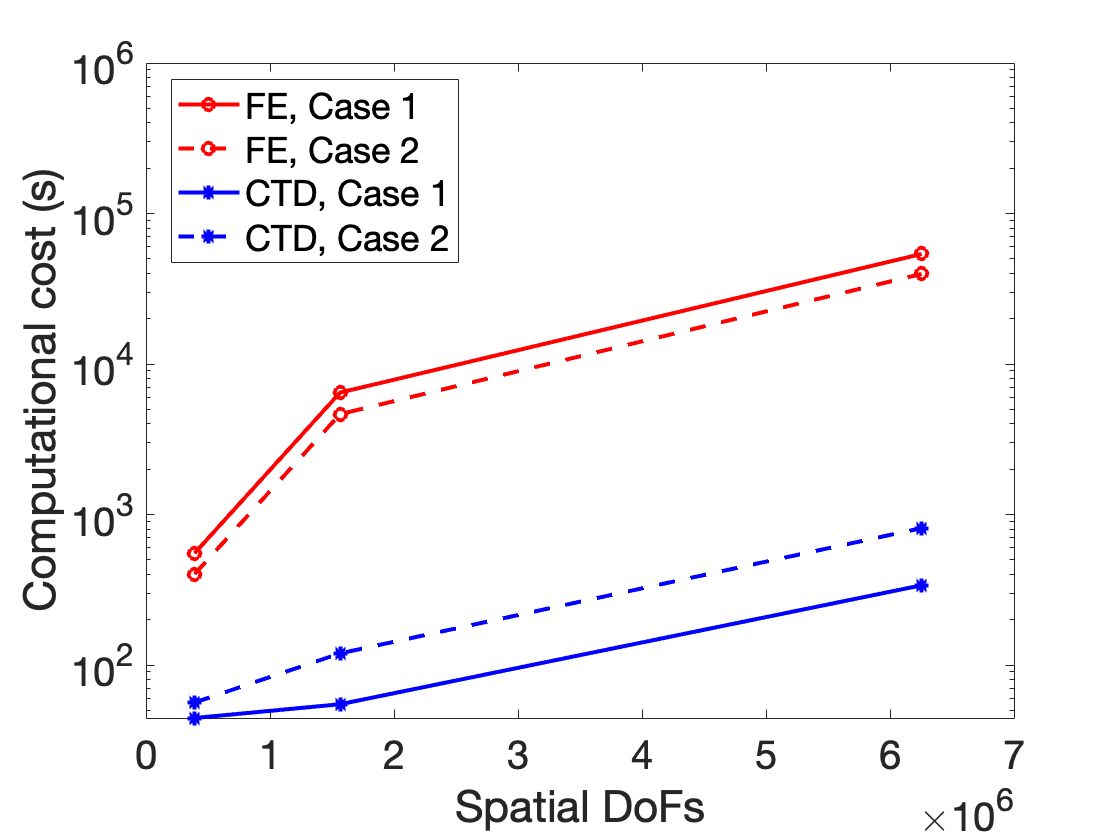}}
\subfigure[Cost for each time step]{\includegraphics[scale=0.135]{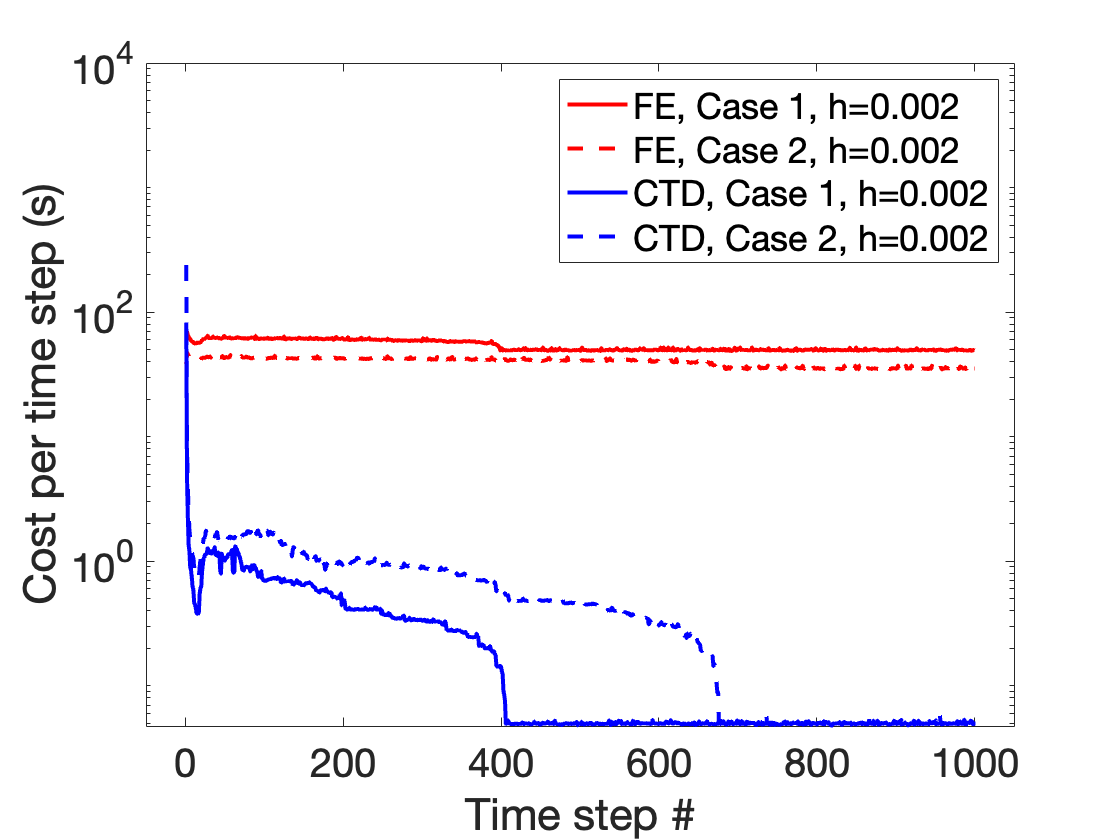}}
\subfigure[Number of modes $M$]{\includegraphics[scale=0.135]{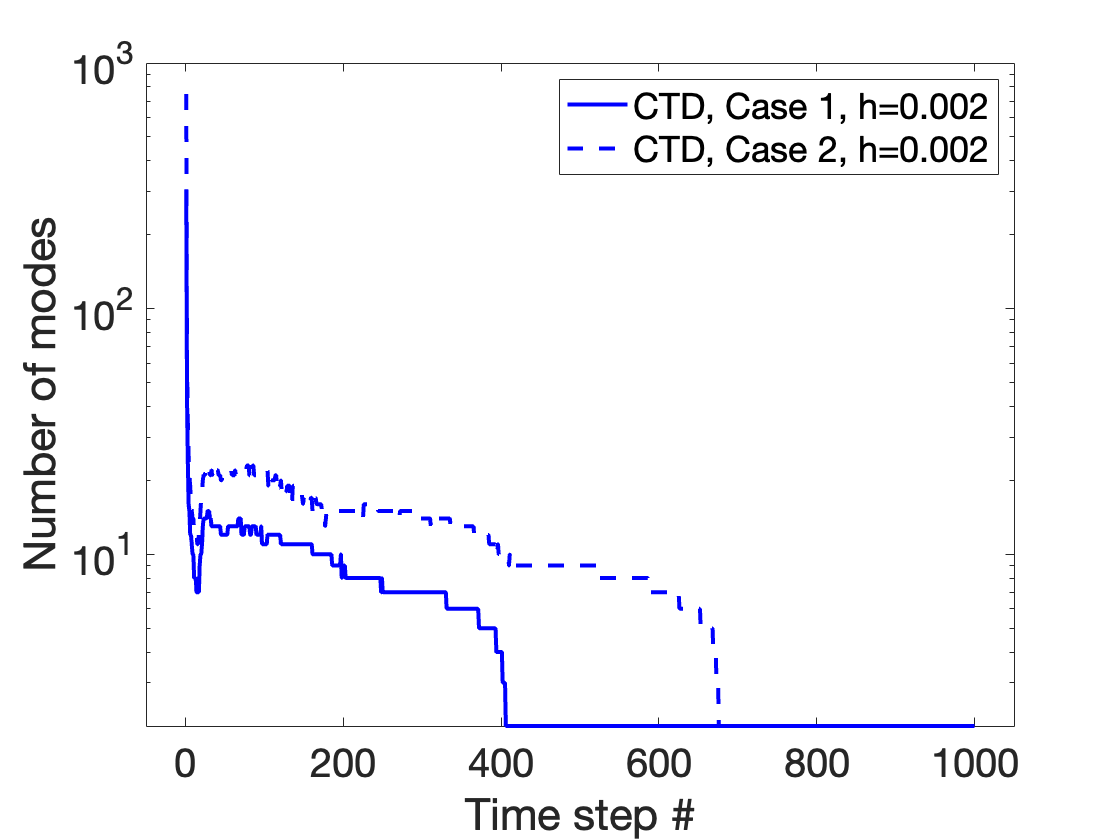}}
\caption{Computational costs of FE and CTD for Case 1 and Case 2}
\label{fig:cpucost2D}
\end{figure}

The efficiency of CTD can still be improved by adopting the proposed adaptive algorithm,  {the so-called ACTD method in Section \ref{sect:solualgo}}. As shown in \figurename~\ref{fig:cpucost2D}(b), the first time step seems the most time-consuming step for CTD, due to the large number of modes. The ACTD method can reduce this computational cost, based on a hybrid CTD/FE computational strategy. As an example, we compare the CTD and ACTD methods using the Case 2 problem with $h=0.002$. The critical (maximum) mode for ACTD is set to be $M_c=100$. We focus on the first 50 time steps, since the remaining steps should be the same for both methods. \figurename~\ref{fig:cpucost2Dadaptive} illustrates the computational costs for CTD and ACTD. As shown in the figure, the computational time is greatly reduced with ACTD for the first time step without increasing the time for  the other steps. The original CTD requires more than 700 modes for the first time step, whereas ACTD does not go beyond 100 modes. This is achieved automatically by slightly modifying the CTD solution algorithm and can be very powerful and robust in general cases. 

\begin{figure}[htbp]
\centering
\subfigure[Cost for each time step]{\includegraphics[scale=0.135]{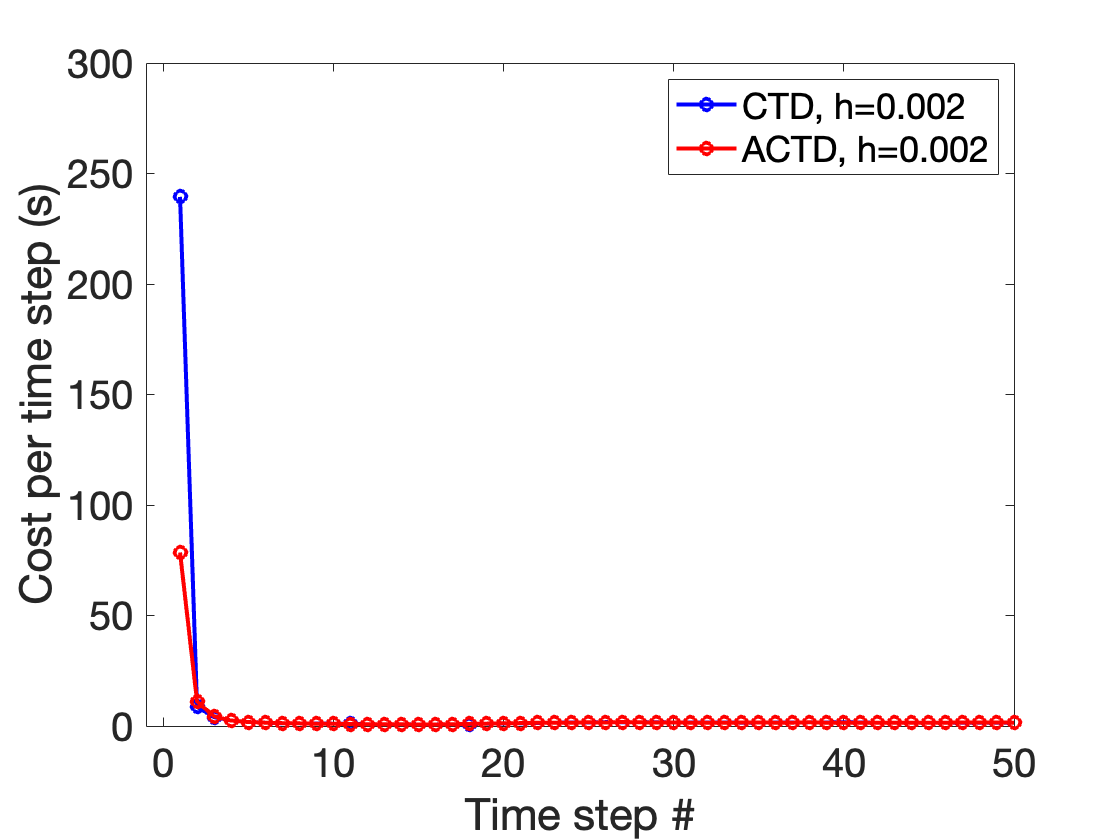}}\quad\quad\quad
\subfigure[Number of modes $M$]{\includegraphics[scale=0.135]{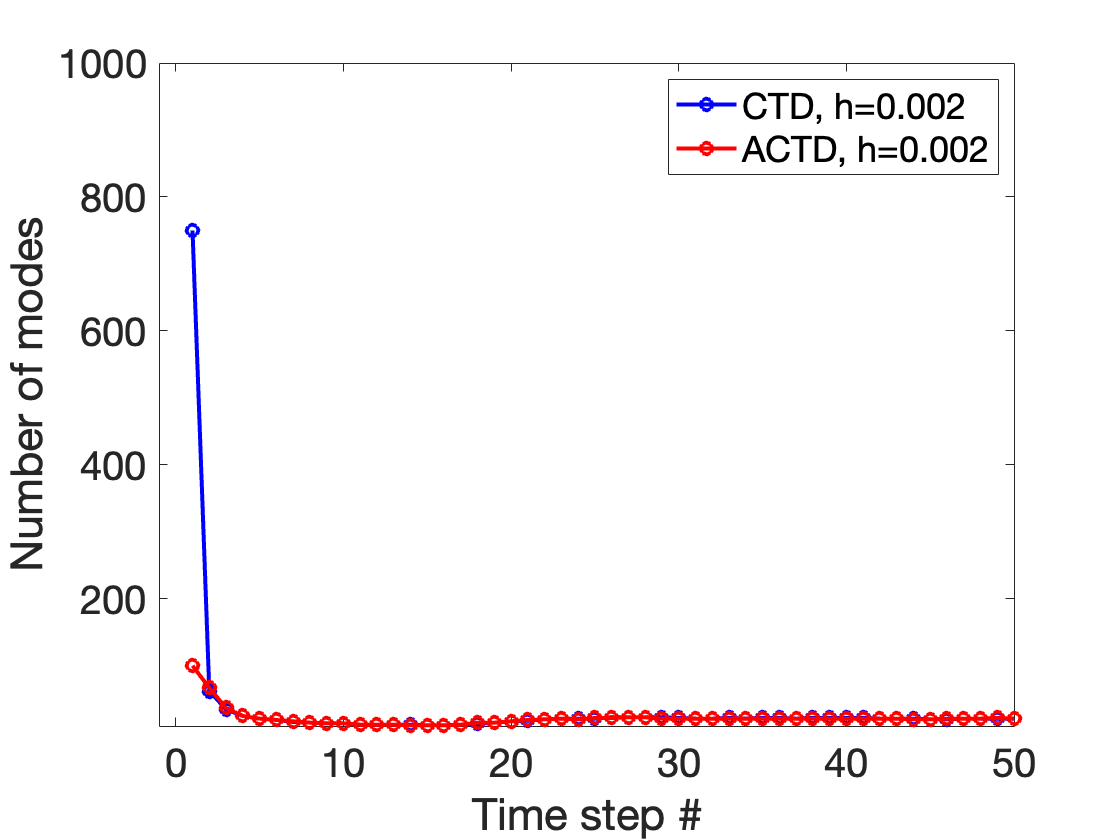}}
\caption{Computational costs of CTD and ACTD for  Case 2}
\label{fig:cpucost2Dadaptive}
\end{figure}

\subsection{3D AC problem}
Next, we test the performance of the method in 3D cases. We use a similar problem setup to  the 2D problems with the coefficients in \tablename~\ref{table:coef} (Case 1), i.e., $L=5, \kappa=1, \alpha_0=10$. The temporal discretization remains the same, i.e., $\Delta t =0.01$ for $t\in [0,10]$. The random initial condition and the computational domain are displayed in \figurename~\ref{fig:domain3D}. We focus on the comparison between ACTD and FE. Here, we ignore the fact that CFE solutions should be taken as the reference for CTD solutions, as our experience with the previous 2D example already demonstrated the similarity between FE and CFE solutions for the given problem. Hence, we can use the FE solutions to qualitatively validate the CTD solutions. The ACTD method used here is designed with a maximum number of mode $M_c=200$.

\begin{figure}[htbp]
\centering
{\includegraphics[scale=0.2]{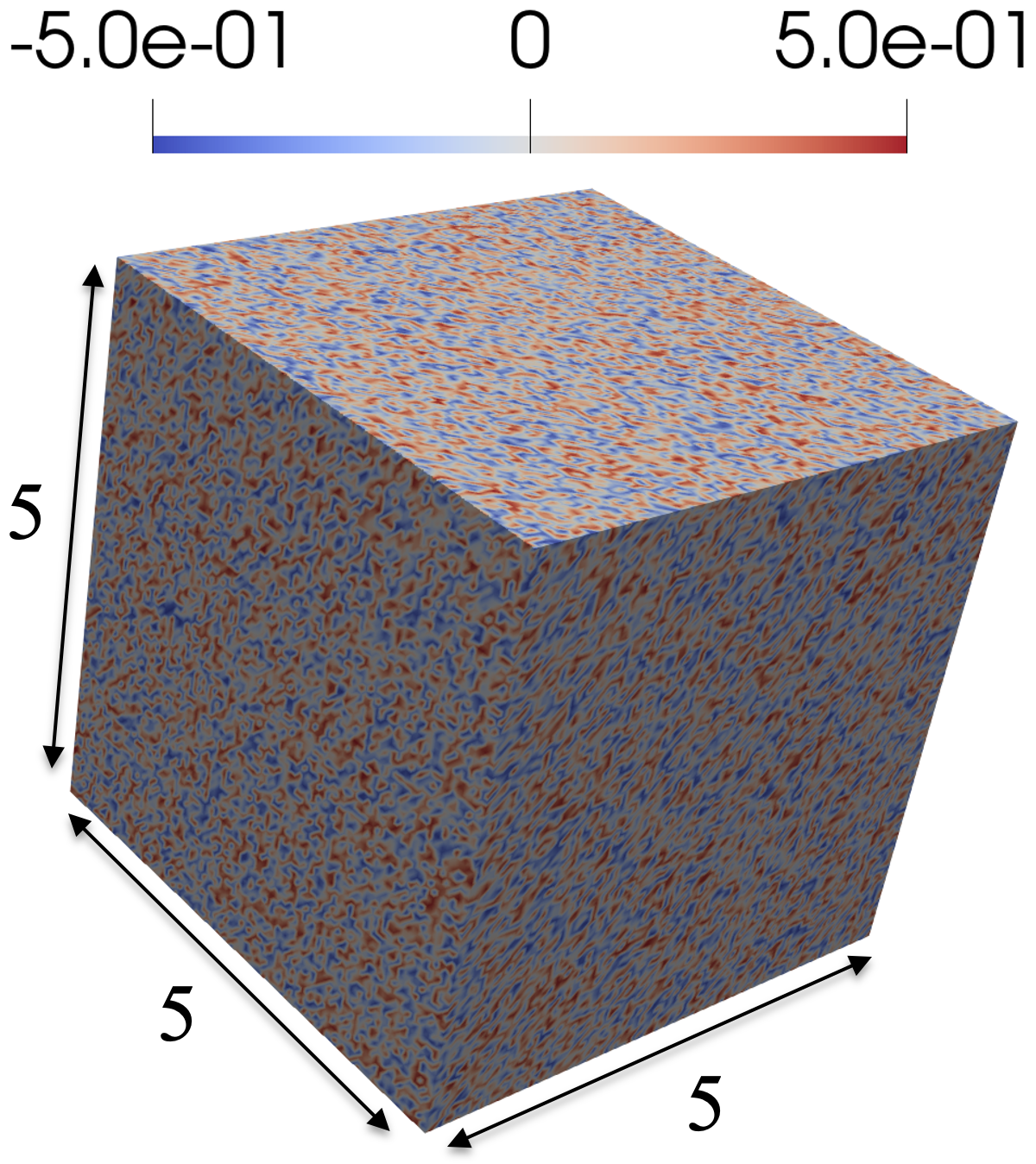}}
\caption{Computational domain and initial condition for the 3D problem}
\label{fig:domain3D}
\end{figure}

\figurename~\ref{fig:u3d} illustrates some solution snapshots of the FE and ACTD methods for the 3D problem. The element size of the 3D uniform mesh is $h=h_0=0.05$. We can see that the ACTD solutions  {agree well with} the FE solutions. This confirms the  {capability} of the method in 3D. To compare the solution efficiency, we further refined the mesh and compared three meshes with $h=h
_0, h_0/1.5, h_0/2$, which corresponds to 1030301, 3442951, 8120601 DoFs. \figurename~\ref{fig:cpucost3D}(a) illustrates the computational cost of each method with the increasing DoFs. As we can see,  ACTD is much more efficient than FE for high resolution problems. The speedup keeps increasing when the resolution increases. \figurename~\ref{fig:cpucost3D}(b) displays the cost per time step for the mesh $h=h_0/2$ in a logarithmic scale. The ACTD method is faster than FE for most of time steps. Again, the first few steps seem less efficient due to the relatively large number of modes, as shown in \figurename~\ref{fig:cpucost3D}(c). ACTD is capable of limiting the computations up to the prescribed number of modes  $M_c=200$, and finally makes the cost of the first few steps equivalent to that of FE. This seems critical as the first few steps might require more than 1000 modes to represent a complex 3D solution, which might reduce the advantages of CTD. This example has confirmed the performance of the proposed method in 3D.

\begin{figure}[htbp]
\centering
\subfigure[FE, $t=0.01$]{\includegraphics[scale=0.11]{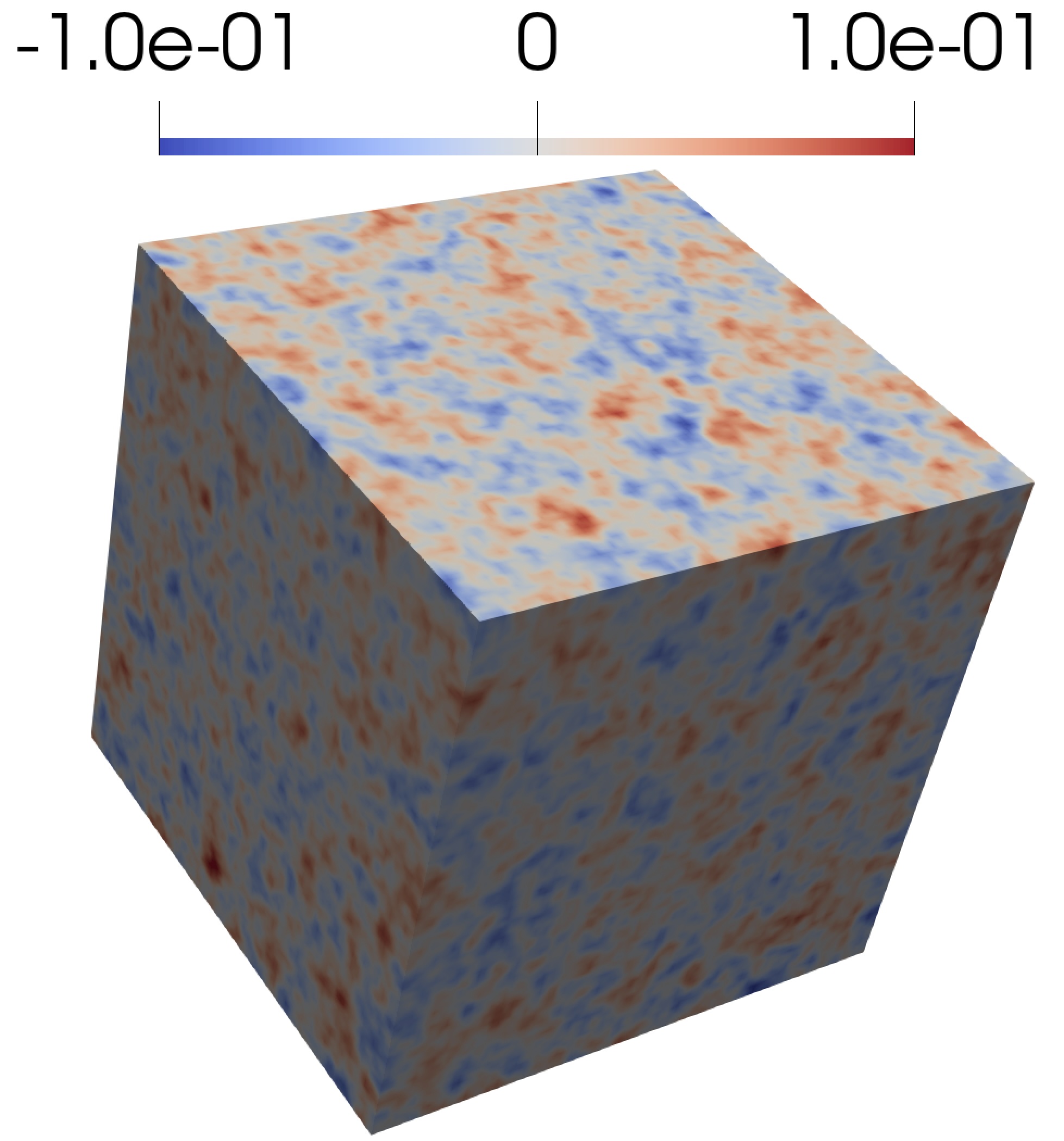}}
\subfigure[FE, $t=0.2$]{\includegraphics[scale=0.11]{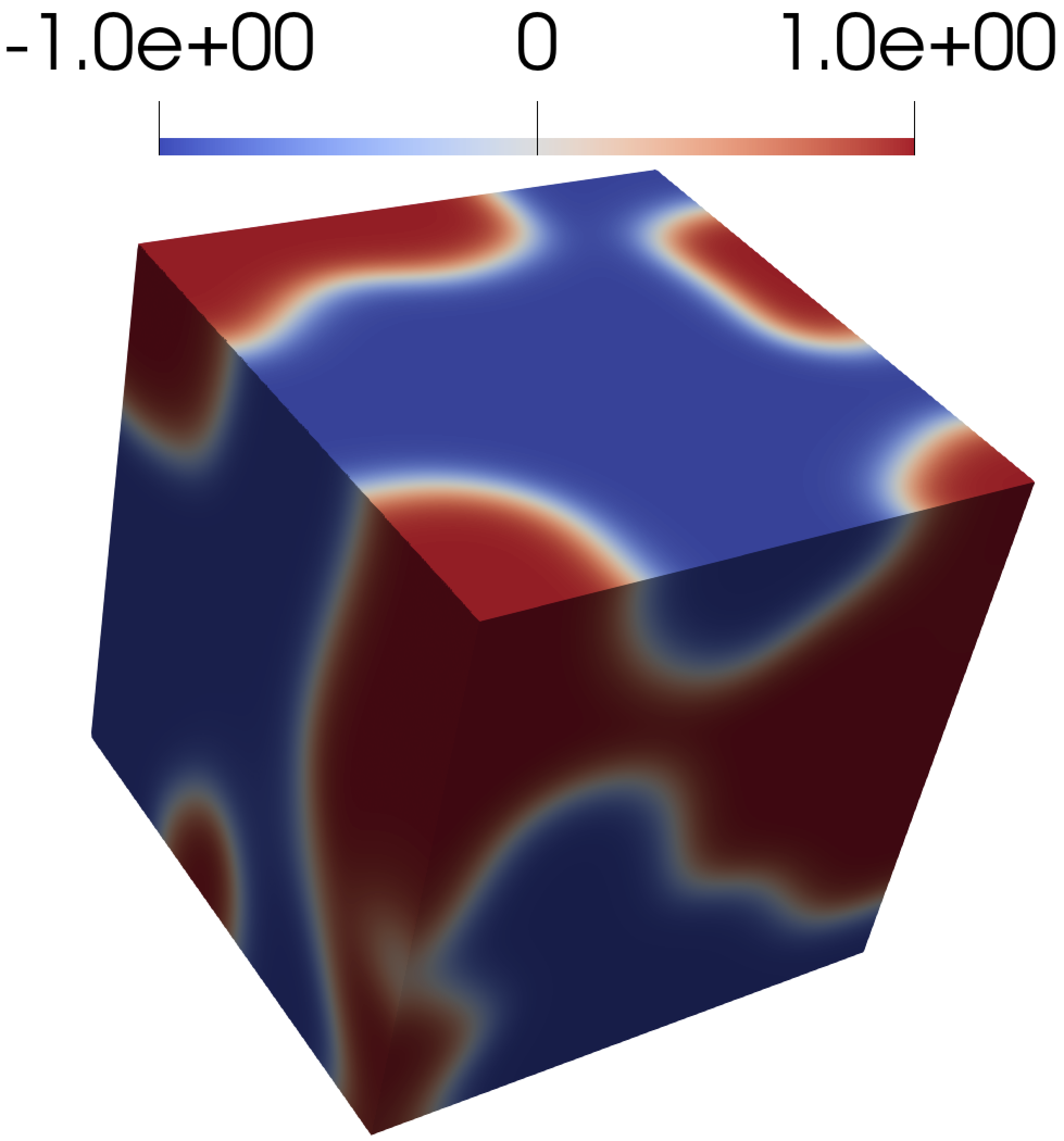}}
\subfigure[FE, $t=2$ ]{\includegraphics[scale=0.11]{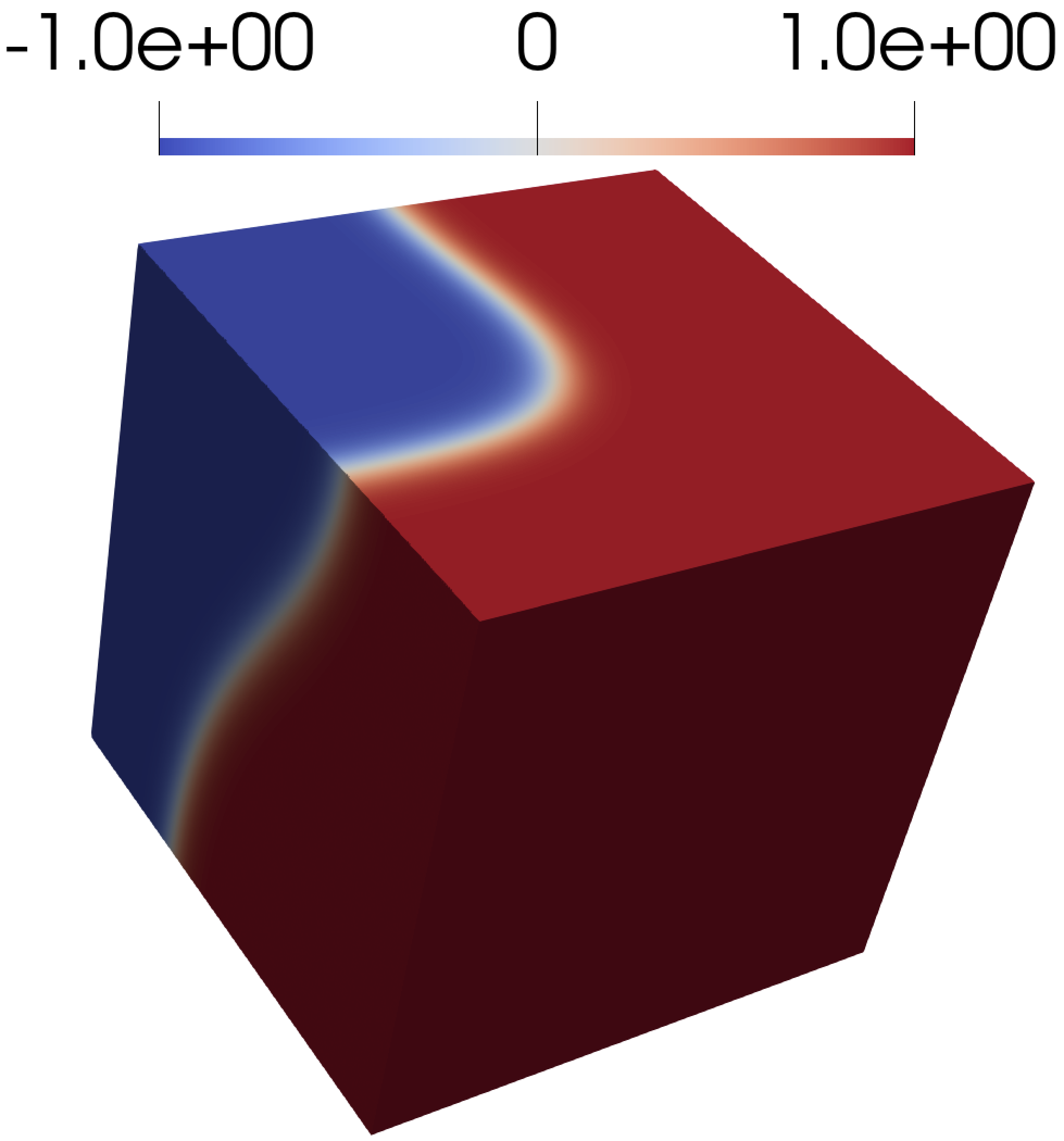}}
\subfigure[FE, $t=10$ ]{\includegraphics[scale=0.11]{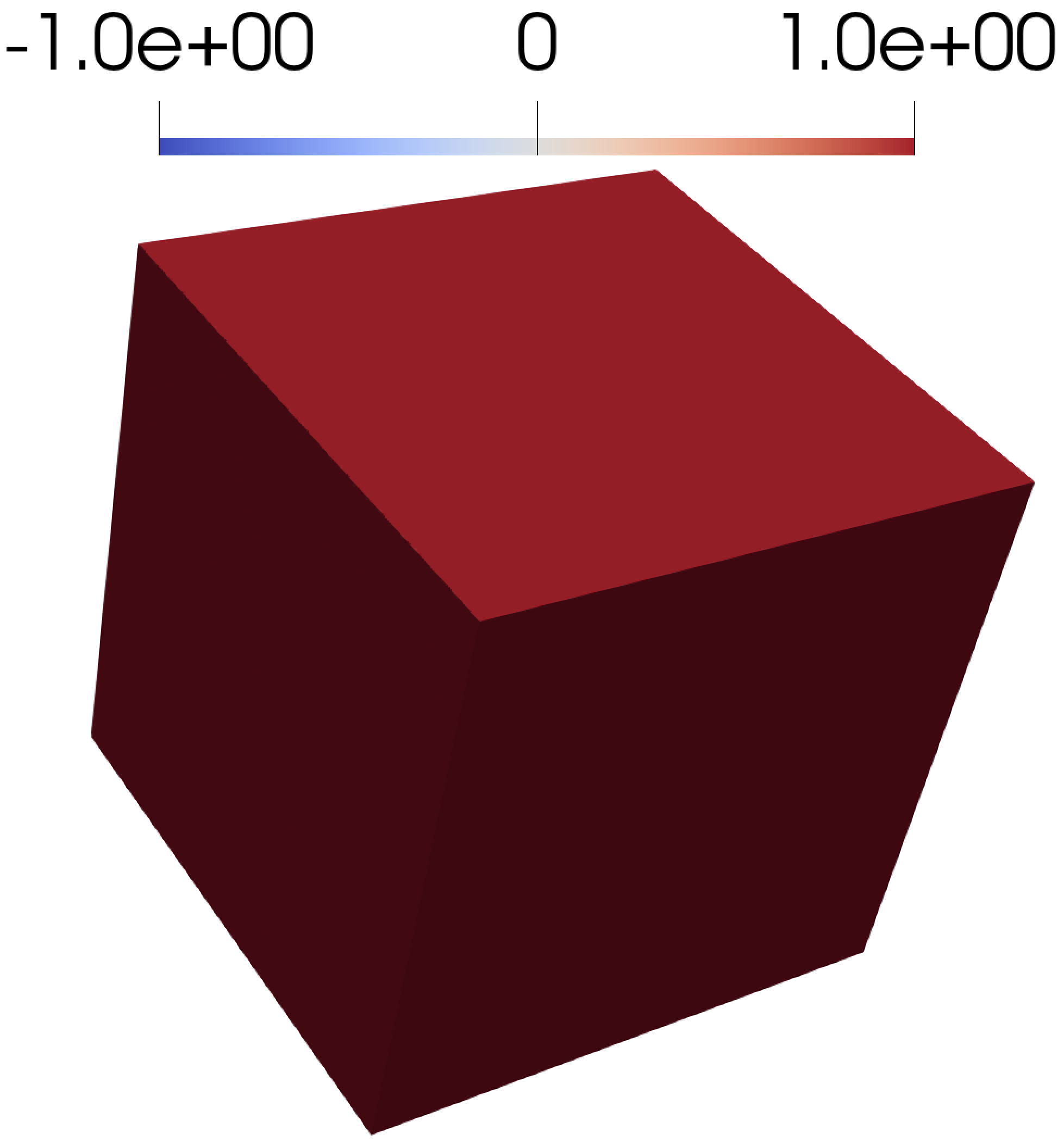}}\\
\subfigure[ACTD, $t=0.01$]{\includegraphics[scale=0.11]{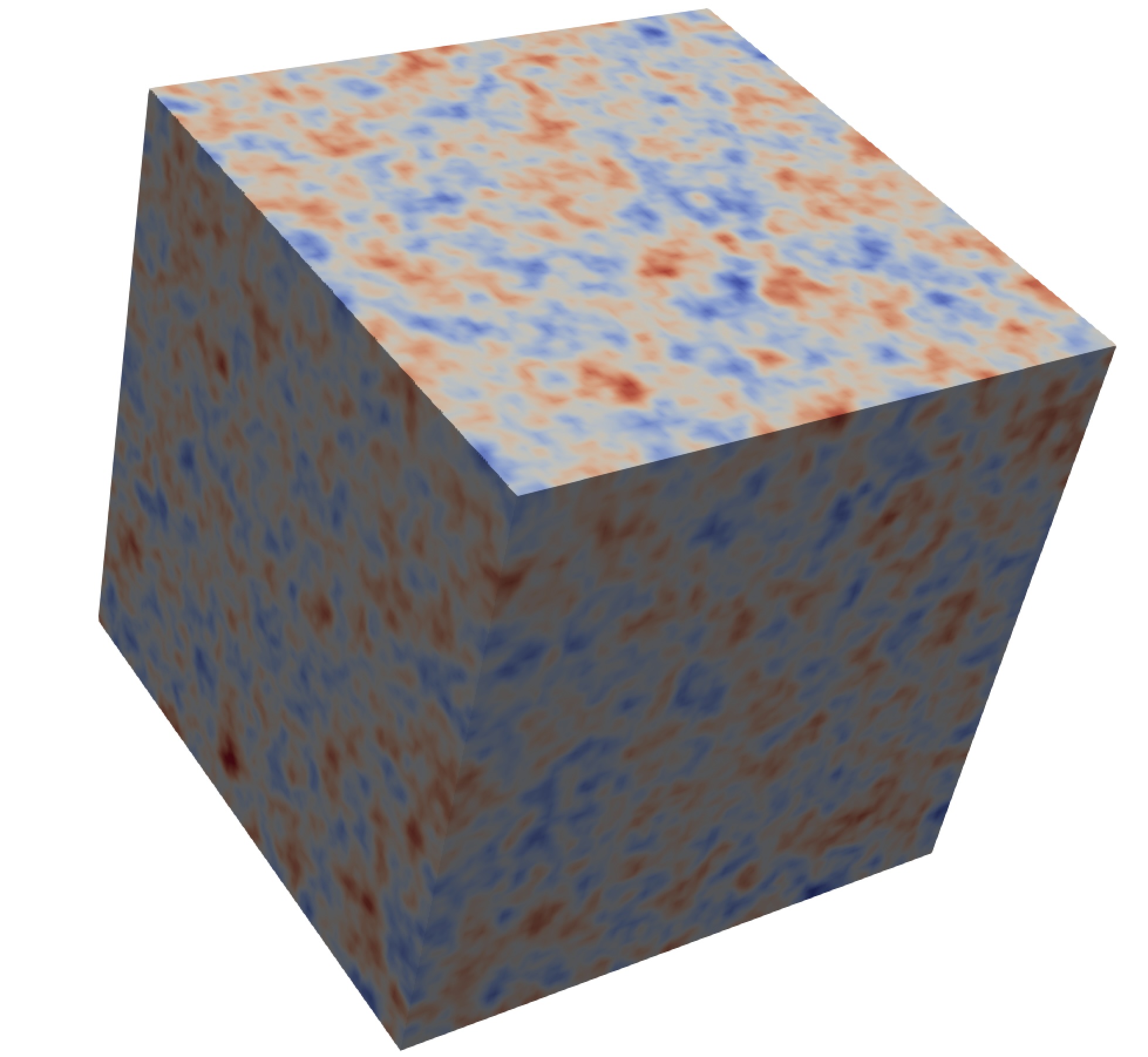}}
\subfigure[ACTD, $t=0.2$]{\includegraphics[scale=0.11]{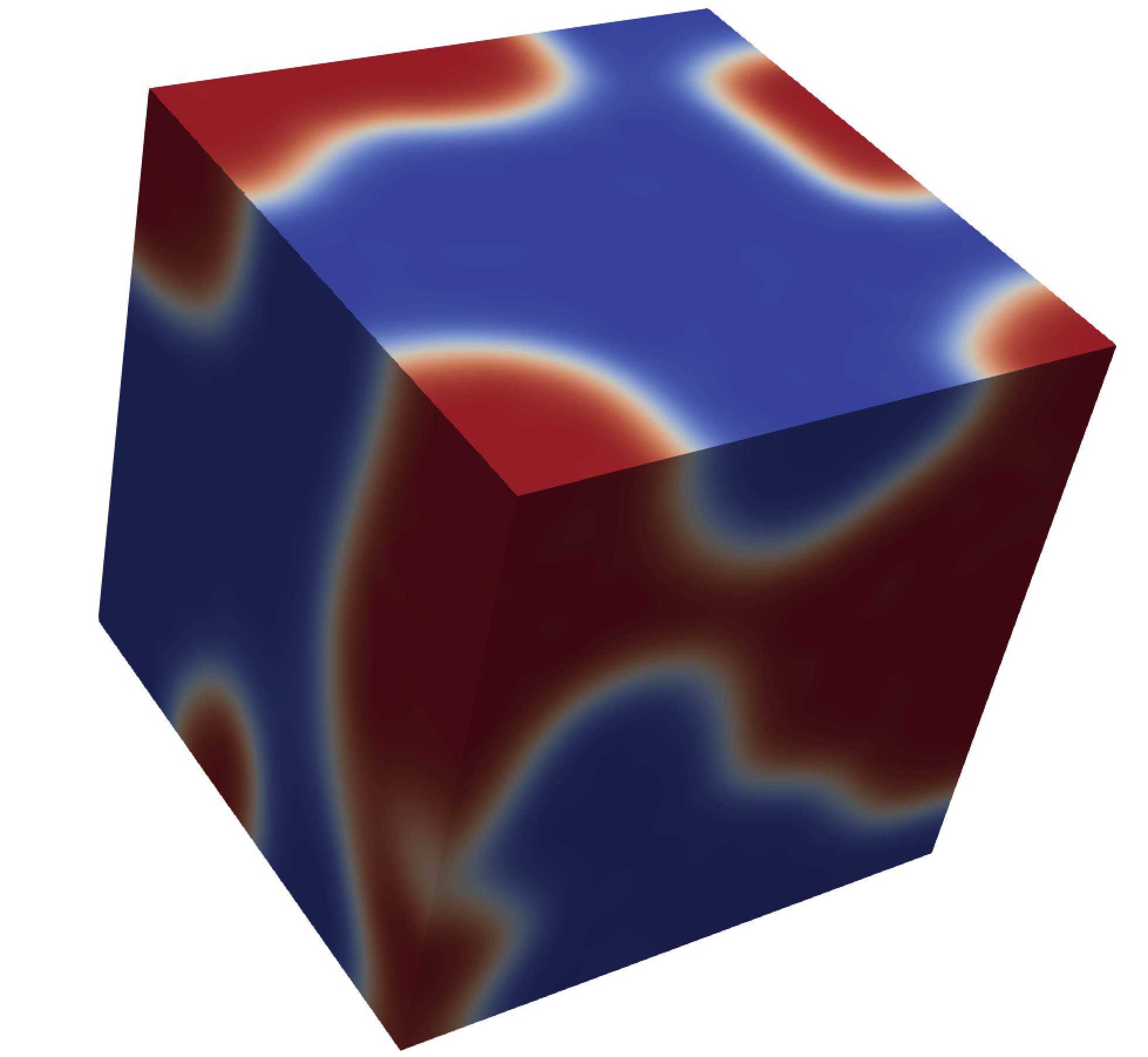}}
\subfigure[ACTD, $t=2$ ]{\includegraphics[scale=0.11]{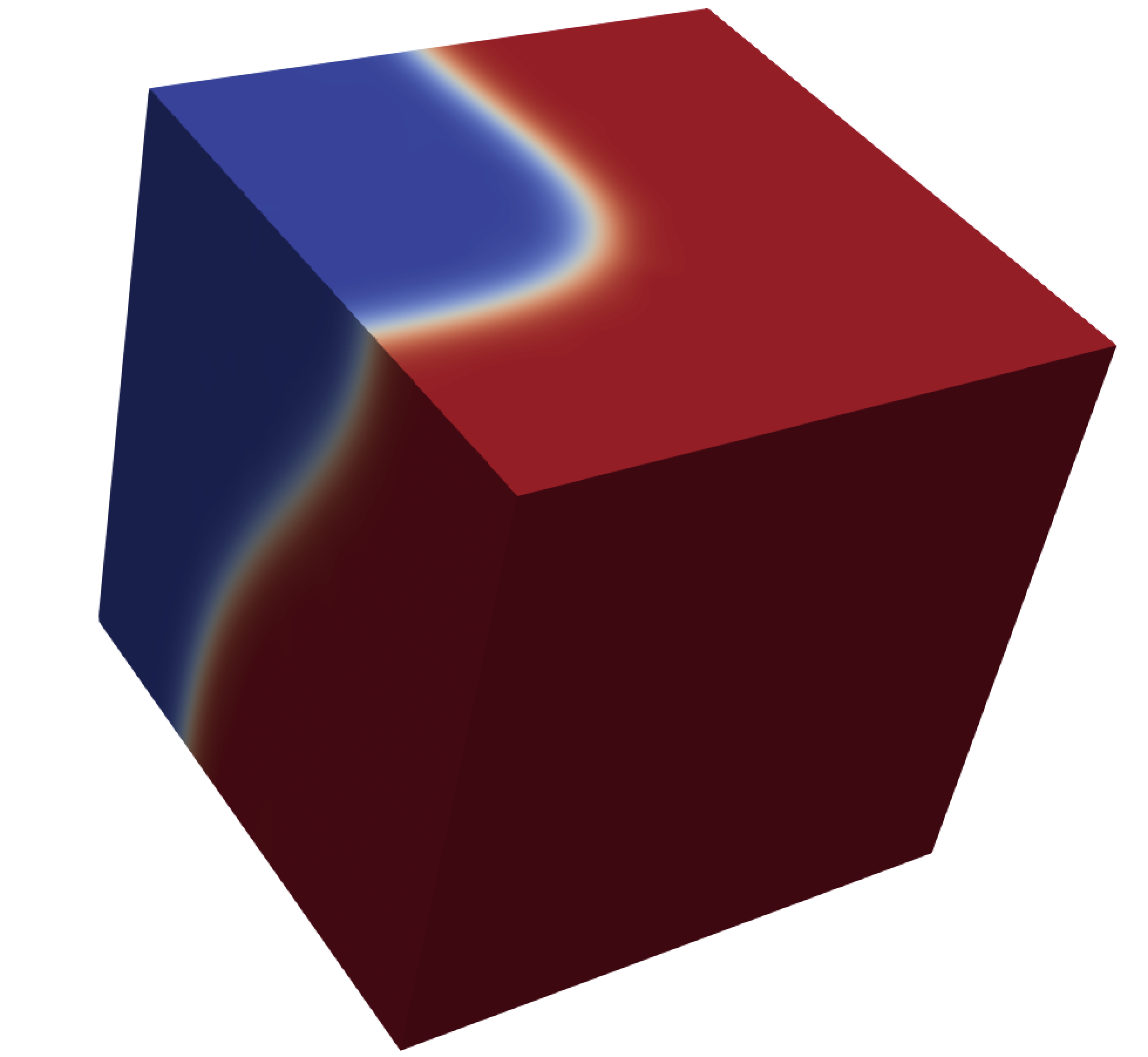}}
\subfigure[ACTD, $t=10$ ]{\includegraphics[scale=0.11]{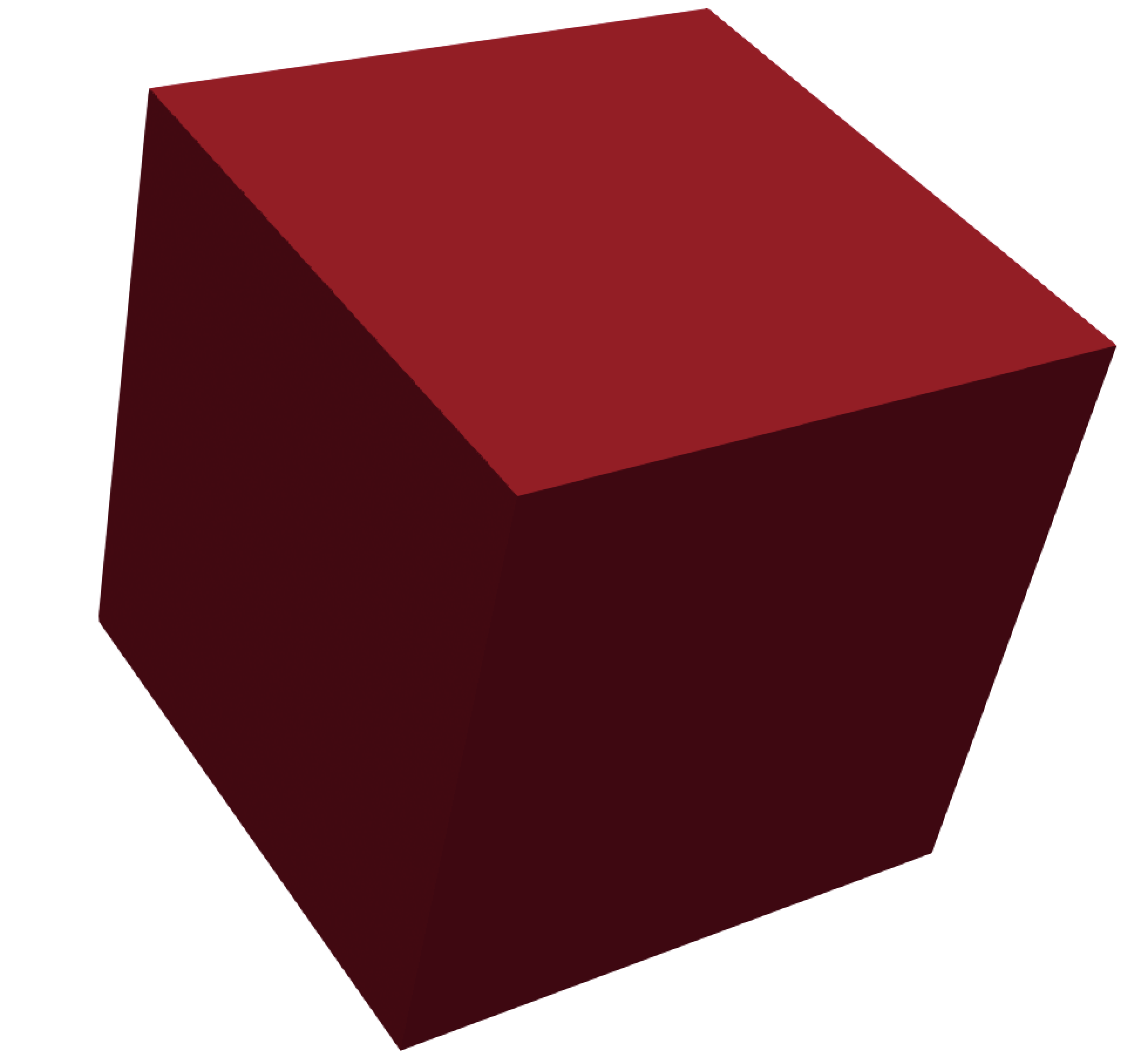}}
\caption{Solution snapshots of 3D FE and ACTD at different time steps}
\label{fig:u3d}
\end{figure}

\begin{figure}[htbp]
\centering
\subfigure[Cost vs Mesh DoFs]{\includegraphics[scale=0.135]{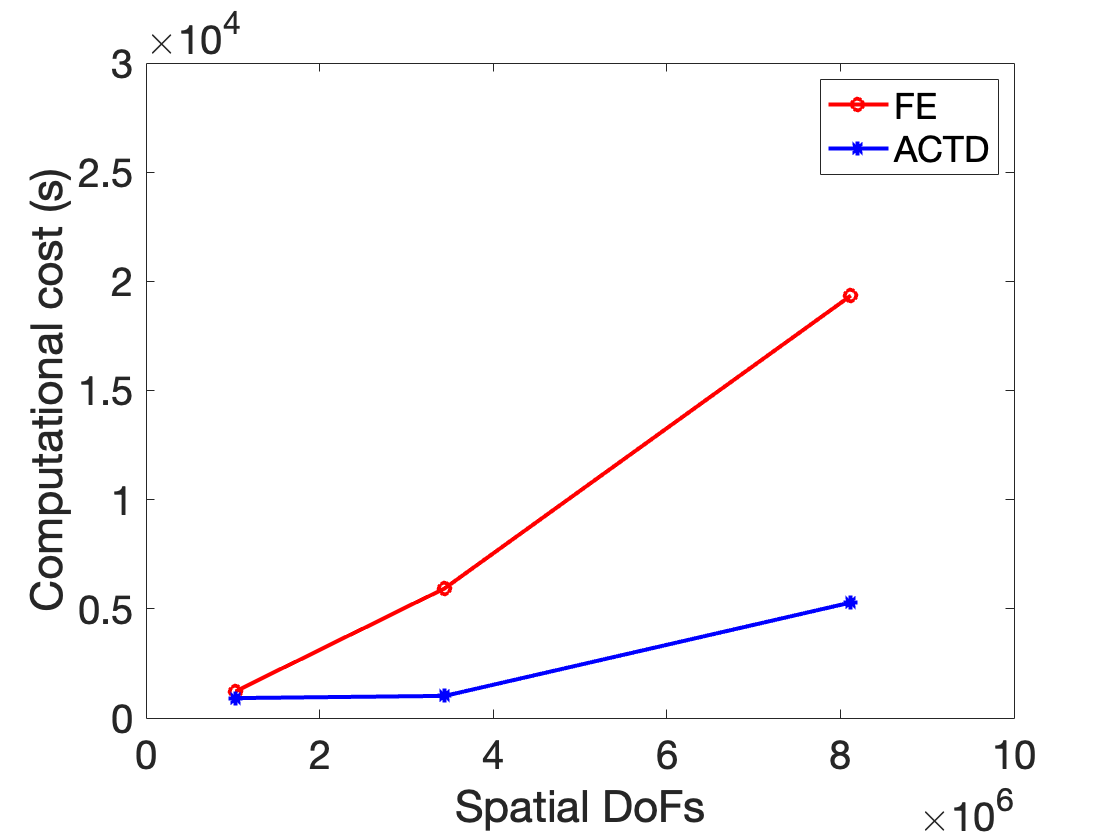}}
\subfigure[Cost for each time step]{\includegraphics[scale=0.135]{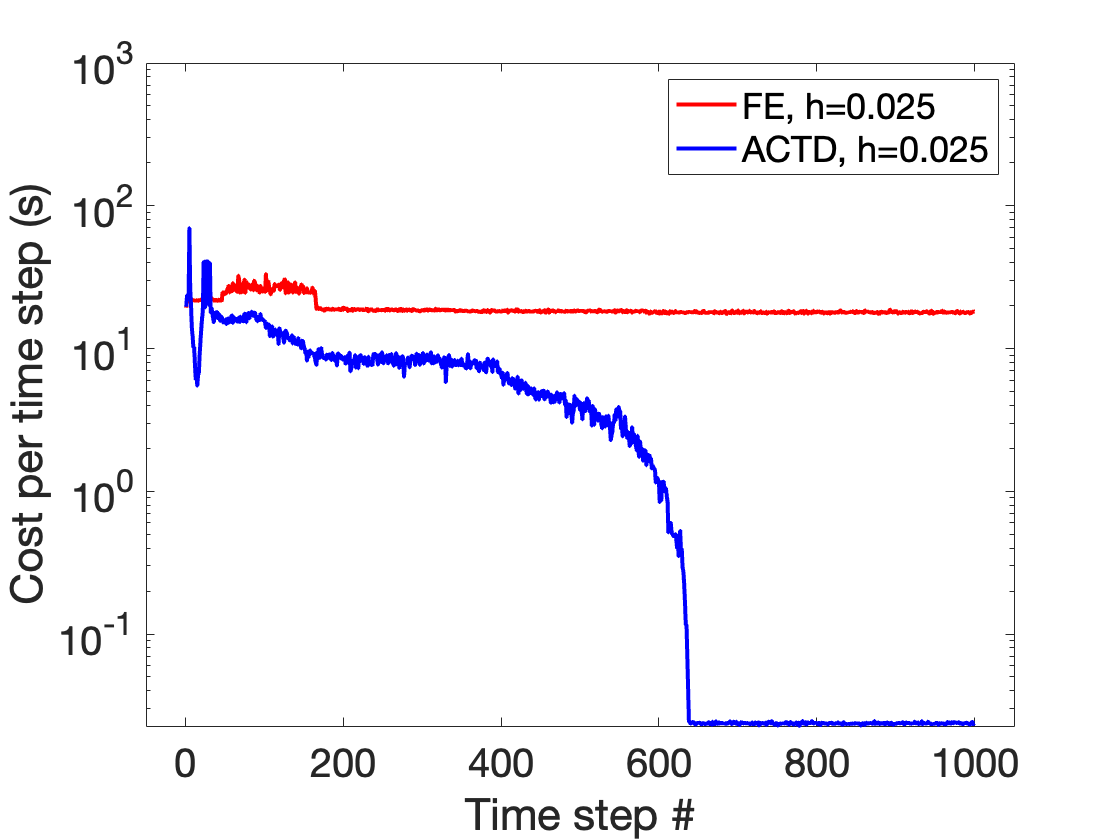}}
\subfigure[Number of modes $M$]{\includegraphics[scale=0.135]{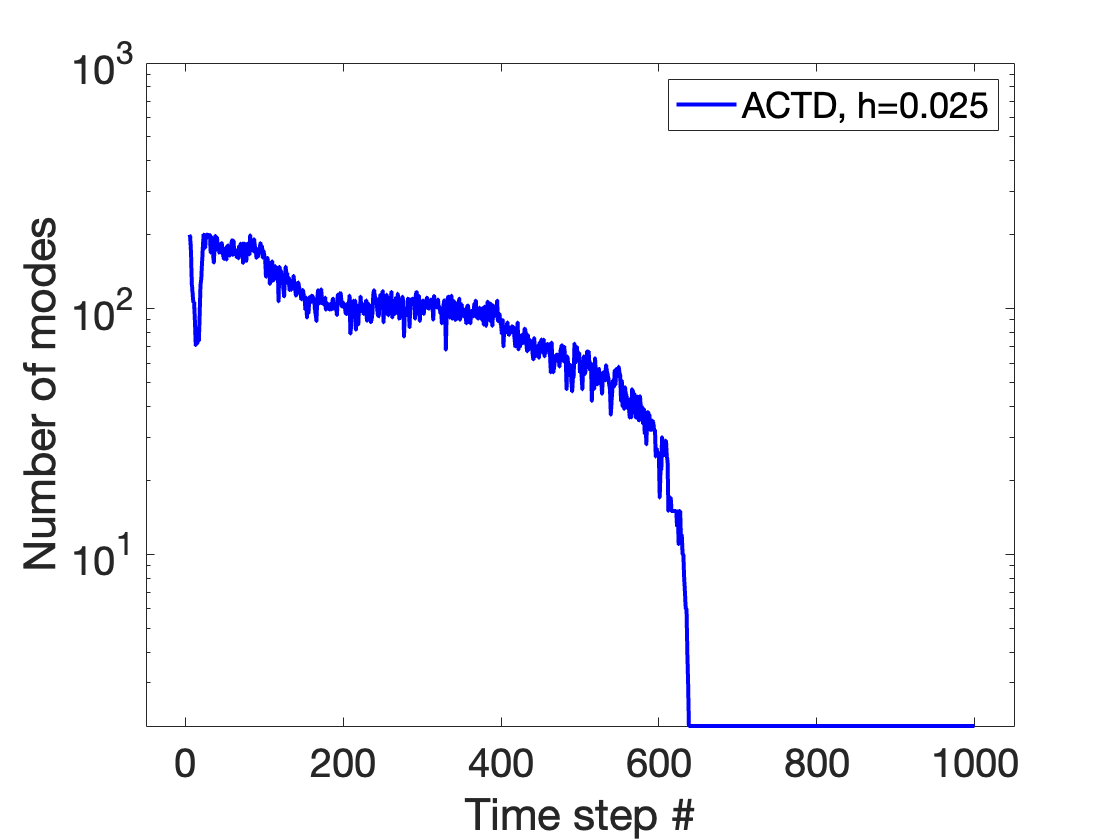}}
\caption{Computational costs of FE and ACTD in 3D}
\label{fig:cpucost3D}
\end{figure}

\subsection{Influence of different controlling parameters}
Here we investigate the influence of the different controlling parameters of the convolution approximation on the CTD solution and computational cost. We use the 3D example with the mesh size $h=h_0=0.05$ for this study. We use the ACTD solution algorithm with $M_c=200$. As a large dilation parameter is always suggested for the convolution approximation, a constant large dilation parameter $a=12$ is used. Four sets of the polynomial order $p$ and patch size $s$ are investigated:  ($p=1, s=1$), ($p=1, s=2$), ($p=2, s=2$), ($p=3, s=3$).

\figurename~\ref{fig:u3dp} illustrates the solutions obtained different controlling parameters at the time $t=0.2$. As we can see, the solutions with different parameters do not change significantly even if higher order approximations are used. This is expected, as the resolution of mesh seems good enough to produce accurate and smooth solutions with ($p=1, s=1$). For coarser meshes, the difference should be more significant. In terms of computational cost, we observed a similar trend with different parameters, as shown in \figurename~\ref{fig:cpucost3Dp}. This means the controlling parameters (even higher order approximations) can be freely chosen without changing significantly the cost. This is a unique advantage of the CTD method, due to the decomposition of the 3D problem and the use of convolution approximation. Hence, we can confirm the great potential of the CTD method for achieving high order approximations while keeping affordable costs,  {although high order derivatives are not required in the formulation used for solving the AC equation.}  

 {We remark that  in general the computational cost should be higher for high order approximations, because the final stiffness matrix \eqref{eq:ux} for CTD would become very dense with high order approximations, similarly to FE. However, the size of the stiffness matrix is very small for CTD and therefore the reduced sparsity of the stiffness matrix does not have a significant effect in the examples.}

\begin{figure}[htbp]
\centering
\subfigure[$p=1, s=1$]{\includegraphics[scale=0.11]{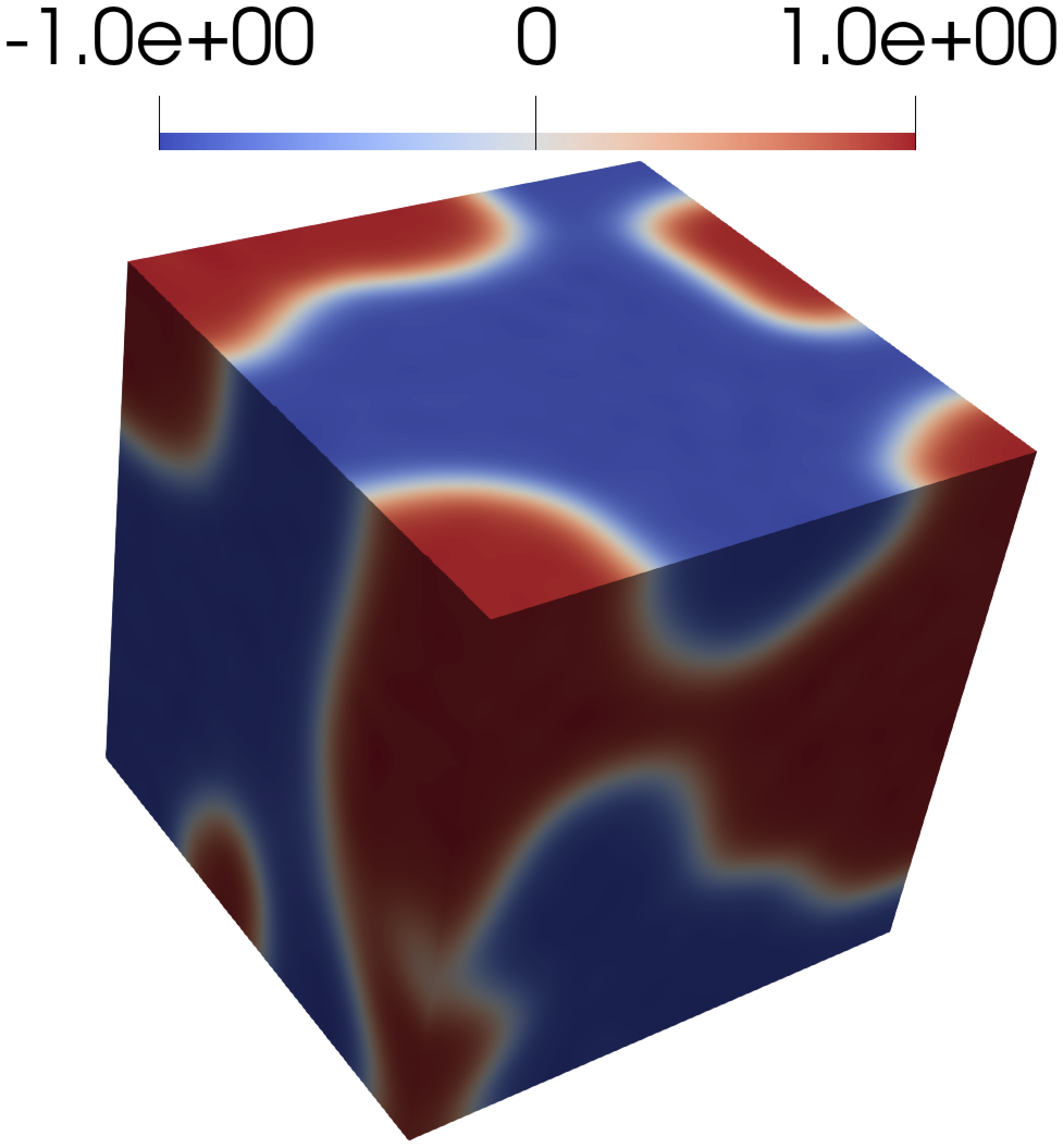}}
\subfigure[$p=1, s=2$]{\includegraphics[scale=0.11]{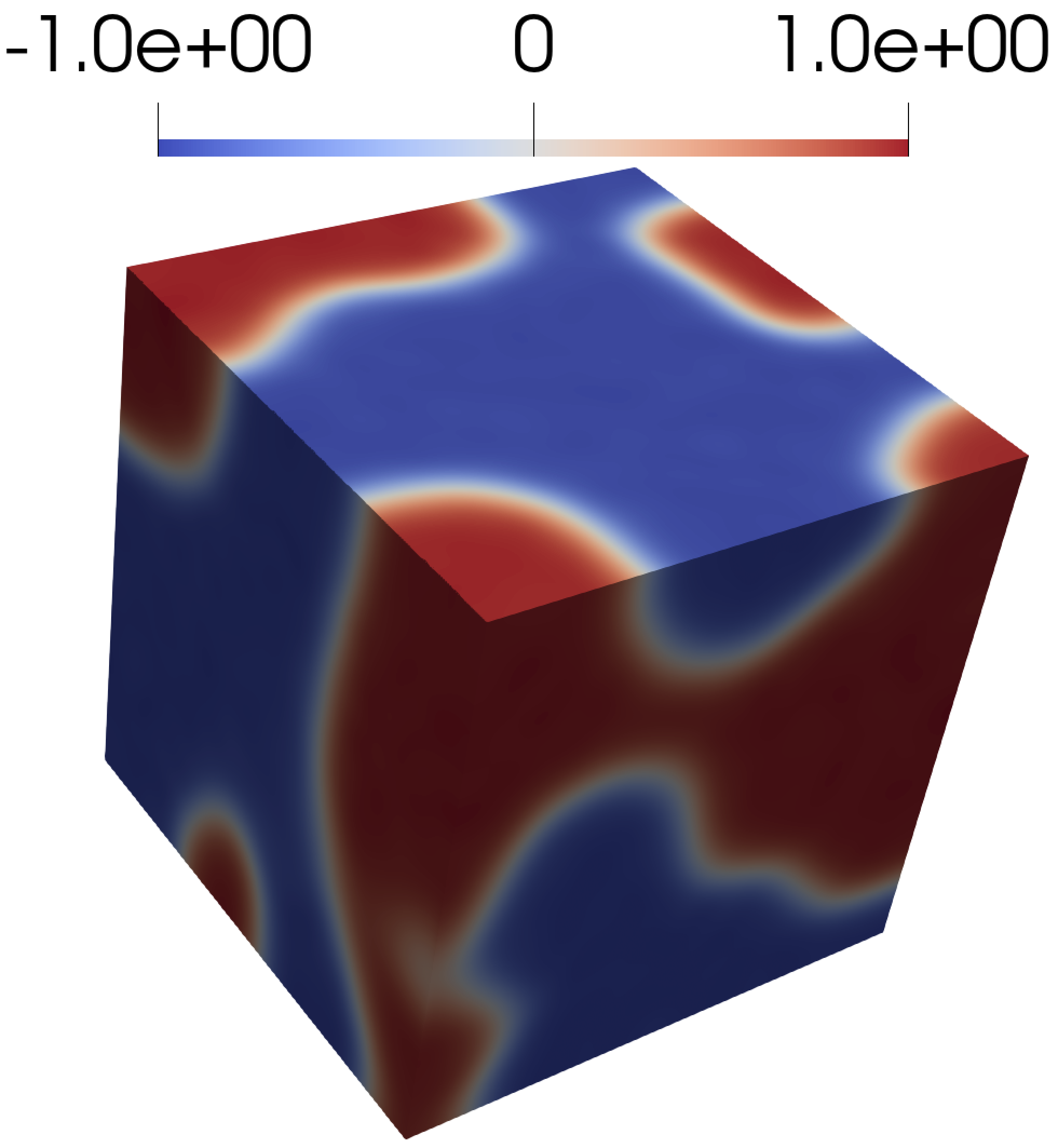}}
\subfigure[$p=2, s=2$ ]{\includegraphics[scale=0.11]{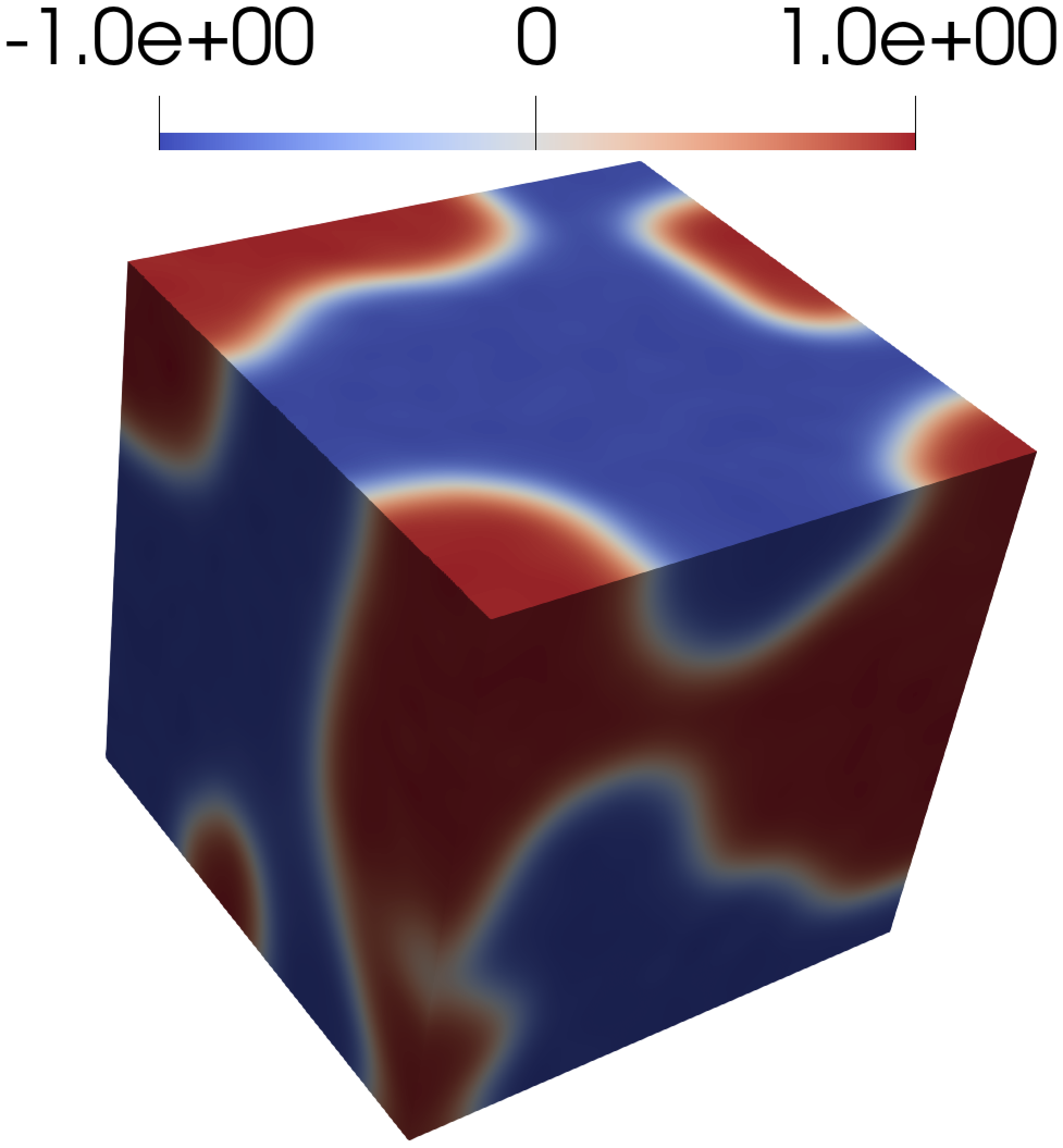}}
\subfigure[$p=3, s=3$ ]{\includegraphics[scale=0.11]{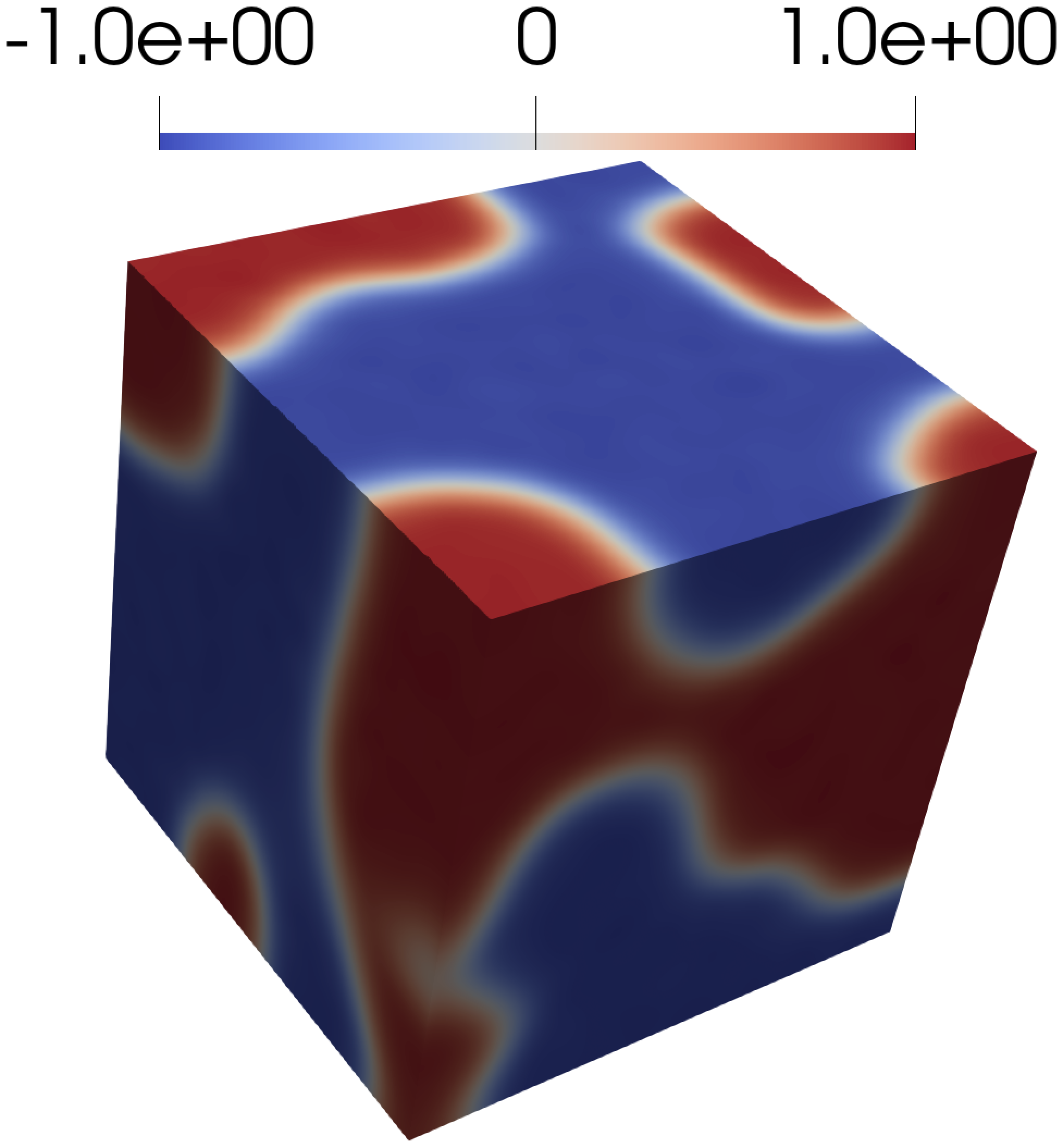}}
\caption{ACTD solutions with different controlling parameters at $t=0.2$ }
\label{fig:u3dp}
\end{figure}

\begin{figure}[htbp]
\centering
\subfigure[Mesh $h=h_0$]{\includegraphics[scale=0.135]{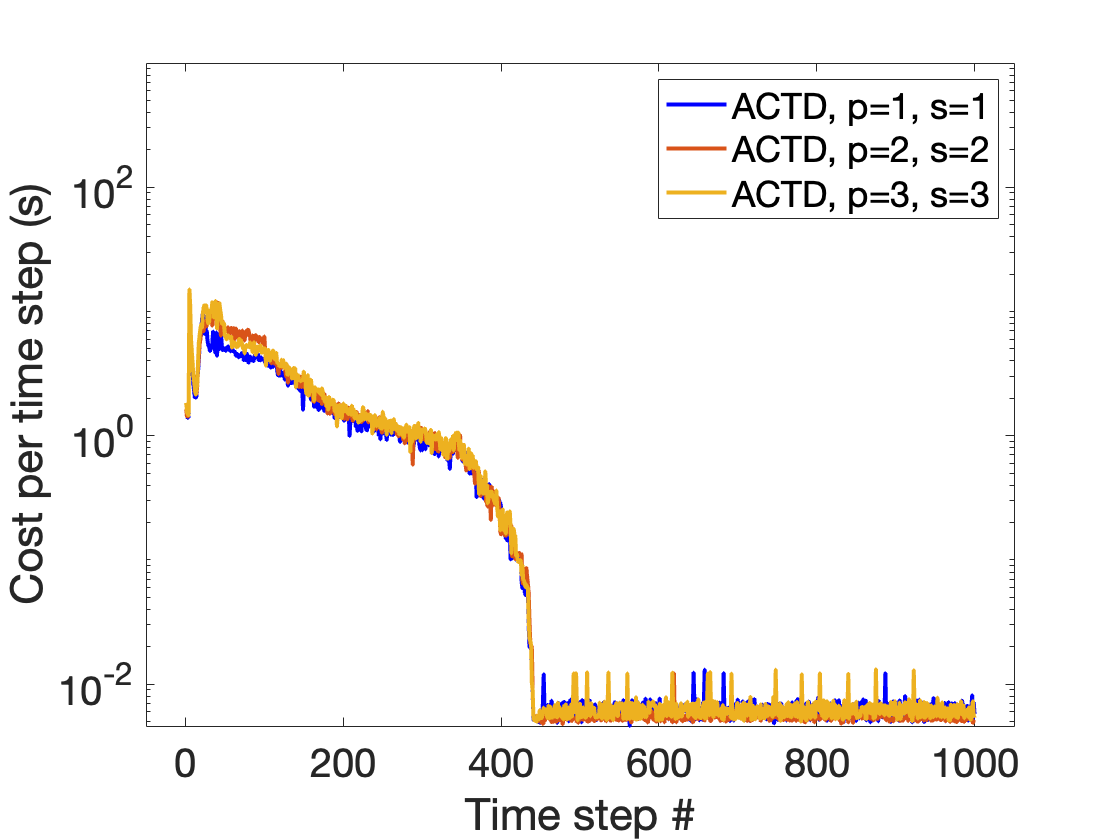}}
\subfigure[Mesh $h=h_0/1.5$]{\includegraphics[scale=0.135]{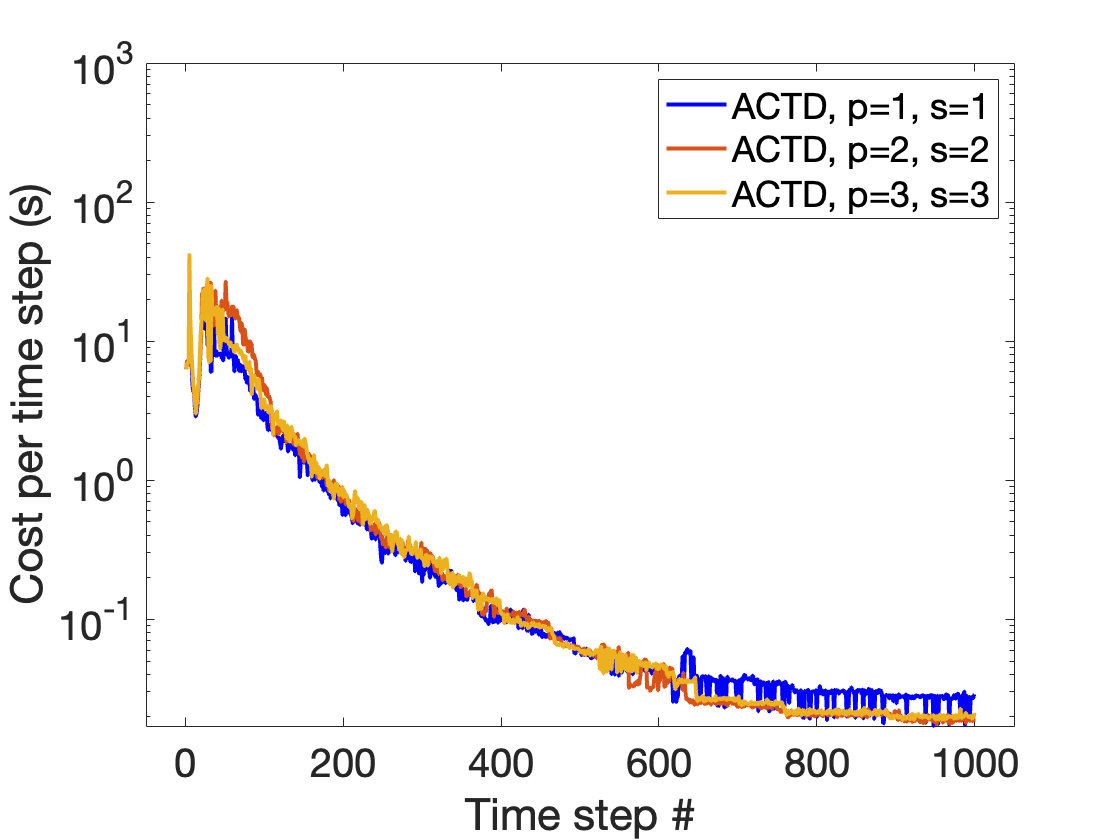}}
\caption{Computational costs of ACTD with different controlling parameters}
\label{fig:cpucost3Dp}
\end{figure}

\subsection{Discussions}
The success of the proposed CTD (including ACTD) method relies on the fact that the solutions to the AC equation can be decomposed into a small or moderate number of separated modes.  {In the presented numerical examples,} we observed that the first few time steps can require many modes to represent the solutions.  {This is due to the randomly generated initial condition that makes the solution very complex at the first time step}. Our proposed adaptive algorithm can effectively overcome this challenge. An alternative way is the extend tensor decomposition method \cite{lu2024extended}, in which local discrete enrichment is involved naturally with the decomposition to reduce the number of separated modes. This will be investigated in our future work.

Regarding the potential of CTD for high order solutions, the application to the Cahn-Hilliard (CH) equation \cite{cahn1958free} could be interesting, which is closely related to the AC equation but requires higher order derivatives. The developed CTD method can be adopted without particular difficulties.  The CTD method is expected to provide a unified and powerful solution framework for both AC and CH problems. Results on the CH equation will be presented in our future work.

\section{Conclusion}
We have developed the CTD based model reduction framework for efficiently solving  the AC equation. The method is based on the use of convolution approximation and a spatial decomposition of the solution field and enables solving a multidimensional problem through several 1D problems. Combining with the semi-implicit time integration scheme (or other appropriate energetically stable time integration), the developed model reduction method can efficiently solve the AC equation with relatively large time steps without violating the energy law. Orders-of-magnitude speedups can be expected with the method, compared to traditional FE. With the numerical examples (including 2D and 3D cases), we found speedups around 7-100$\times$.  This factor can further increase if larger meshes are used, demonstrating a great potential of the method for solving extremely high resolution problems. Furthermore, due to the use of convolution approximation, the method also enables an easy and relatively cheap way to implement high order approximations without changing the mesh. Numerical studies have confirmed this point.

 {In the future, we can compare and potentially combine the proposed method with other efficient solution methods, such as  multigrid approaches \cite{liu2022efficient,leonor2024go}.} In terms of applications, the proposed model reduction framework can be used to simulate many physical phenomena governed by the AC-type equation. An example of our ongoing work is the application to simulate the grain growth in welding and additive manufacturing materials. This framework will make the challenging large size high-resolution phase field simulations for detailed microstructure modeling computationally accessible and provide the essential tool to investigate the manufacturing processes. 

\section*{Acknowledgement}
YL and CY would like to acknowledge the support of University of Maryland Baltimore County through the startup fund and the COEIT Interdisciplinary Proposal Program Award. 

\appendix
{\section{Proof of energy stability condition}\label{apdx:proof}}
 {
\begin{thm}
The solution to the formulation \eqref{eq:AC-discrete-semi-stable} is energetically stable for arbitrary $\Delta t$, if $\alpha$ is chosen according to
\begin{equation}
    \alpha \geq \frac{\max(w')}{2} {.}
\end{equation}
\end{thm}
\begin{proof}
The proof follows the original work of \cite{shen2010numerical}. Let us rewrite the weak form \eqref{eq:AC-weak} using the stabilized semi-implicit formulation
\begin{equation}
\label{eq:AC-semi-stable}
    \int_\Omega  \delta u \ (u^{k+1}-u^{k})(\frac{1}{\Delta t}+\alpha L)\ d \Omega + \int_\Omega \delta u \  L w({u}^{k}) \ d \Omega +\int_\Omega  \nabla\delta u\cdot  L\kappa \nabla{u}^{k+1} \ d \Omega =0  {.}
\end{equation}
Considering that $\delta u =u^{k+1}-u^{k}$, we have 
\begin{equation}
    \label{eq:AC-semi-du}
    \int_\Omega \ (u^{k+1}-u^{k})^2(\frac{1}{\Delta t}+\alpha L)\ d \Omega + \int_\Omega (u^{k+1}-u^{k}) \  L w({u}^{k}) \ d \Omega +\int_\Omega  (\nabla u^{k+1}- \nabla u^{k})\cdot  L\kappa \nabla{u}^{k+1} \ d \Omega =0  {.}
\end{equation}
The last term of the above equation can be written as
\begin{equation}
    \begin{aligned}  
    \int_\Omega  (\nabla u^{k+1}- \nabla u^{k})\cdot  L\kappa \nabla{u}^{k+1} \ d \Omega\ = &\quad\frac{1}{2}\int_\Omega  (\nabla u^{k+1}- \nabla u^{k})^2  L\kappa \ d \Omega \\
    &+\frac{1}{2} \int_\Omega  (\nabla u^{k+1})^2  L\kappa \ d \Omega\\
    &-\frac{1}{2} \int_\Omega  (\nabla u^{k})^2  L\kappa \ d \Omega {.}
    \end{aligned}
\end{equation}
Additionally, using Taylor expansion, we have
\begin{equation}
    \begin{aligned}  
    \int_\Omega (u^{k+1}-u^{k}) \  L w({u}^{k}) \ d \Omega\ = &\quad\int_\Omega  L\mathcal{F}(u^{k+1})    \ d \Omega \\
    &-\int_\Omega  L\mathcal{F}(u^{k})   \ d \Omega\\
    &-\frac{1}{2} \int_\Omega  (u^{k+1}-u^{k})^2  L w'({u}^{k}) \ d \Omega {.}
    \end{aligned}
\end{equation}
Therefore, Eq. \eqref{eq:AC-semi-du} becomes
\begin{equation}
    \label{eq:AC-semi-expan}
    \begin{aligned} 
    \int_\Omega \ (u^{k+1}-u^{k})^2(\frac{1}{\Delta t}+\alpha L)\ d \Omega\ &+ \int_\Omega  L\mathcal{F}(u^{k+1})    \ d \Omega \\
    &-\int_\Omega  L\mathcal{F}(u^{k})   \ d \Omega\\
    &-\frac{1}{2} \int_\Omega  (u^{k+1}-u^{k})^2  L w'({u}^{k}) \ d \Omega\\
    &+\frac{1}{2}\int_\Omega  (\nabla u^{k+1}- \nabla u^{k})^2  L\kappa \ d \Omega \\
    &+\frac{1}{2} \int_\Omega  (\nabla u^{k+1})^2  L\kappa \ d \Omega\\
    &-\frac{1}{2} \int_\Omega  (\nabla u^{k})^2  L\kappa \ d \Omega\\
    &=0 {.}
    \end{aligned}
\end{equation}
Rearrange the above equation and consider that $E =\int_\Omega (\mathcal{F}+\frac{1}{2} \kappa (\nabla u)^2)\ d \Omega$, then
\begin{equation}
    \begin{aligned} 
    L E (u^{k+1})-L E (u^{k})=&-\int_\Omega \ (u^{k+1}-u^{k})^2(\frac{1}{\Delta t}+\alpha L)\ d \Omega \\
    &+\frac{1}{2} \int_\Omega  (u^{k+1}-u^{k})^2  L w'({u}^{k}) \ d \Omega\\
    &-\frac{1}{2}\int_\Omega  (\nabla u^{k+1}- \nabla u^{k})^2  L\kappa \ d \Omega {.}
    \end{aligned}
\end{equation}
Therefore,
\begin{equation}
    \begin{aligned} 
    L E (u^{k+1})-L E (u^{k})\leq\int_\Omega \ (u^{k+1}-u^{k})^2(\frac{1}{2} L w'({u}^{k})-\frac{1}{\Delta t}-\alpha L)\ d \Omega  {.}
    \end{aligned}
\end{equation}
Since $L>0$, a sufficient condition for $E (u^{k+1})-E (u^{k})\leq 0$ is that
\begin{equation}
\label{eq:stable-condtion}
    \begin{aligned} 
   \frac{1}{2} L w'({u}^{k})-\frac{1}{\Delta t}-\alpha L\leq 0  {.}
    \end{aligned}
\end{equation}
The above condition holds for arbitrary $\Delta t$, if
\begin{equation}
    \begin{aligned} 
   \frac{1}{2} L w'({u}^{k})-\alpha L\leq 0  {.}
    \end{aligned}
\end{equation}
Hence, we only need
\begin{equation}
    \begin{aligned} 
   \alpha \geq \frac{1}{2} \max (w') {,}
    \end{aligned}
\end{equation}
which is Eq. \eqref{eq:AC-discrete-semi-stable-step}. Furthermore, if we consider $\alpha=0$ in \eqref{eq:stable-condtion}, it gives the condition \eqref{eq:AC-discrete-semi-step} for the original semi-implicit scheme \eqref{eq:AC-discrete-semi}. 
\end{proof}
}

{\section{Radial basis interpolation}\label{apdx:radialbasis}}
Let us consider the following 1D case for illustration purposes
\begin{equation}
    \begin{aligned}           
   {u^{\text{CFE}}}({\xi})=&\sum_{i\in A^e}{N}_{{i}}({\xi})\underbrace{\sum_{j\in A^i_s} {W}^{{\xi}_i}_{a,j} ({\xi}) u_j}_{u^i({\xi})} {,}\\
    \end{aligned}
\end{equation}
 {where $u^{i}({\xi})$ is the part of interpolation done by the patch function $W$ centering around the $i$-th node with $i\in A^e$. We can now focus on $u^{i}({\xi})$}
\begin{equation}
\label{eq:radialbasis}     
    u^{i}({\xi})=\sum_{j\in A^i_s} {W}^{{\xi}_i}_{a,j} ({\xi}) u_j,
\end{equation}
where the supporting node set of $W$ is $A^i_s$ with a given patch size $s$.  \figurename~\ref{fig:1DCFEMnodes} illustrates a 1D  {CFE} element with {the} patch size $s=1$ and the two-node shape functions for $N_i$. Now the question is how to compute  $W$ based on the given supporting nodes.

\begin{figure}[!htbp]
\centering
\includegraphics[scale=0.5]{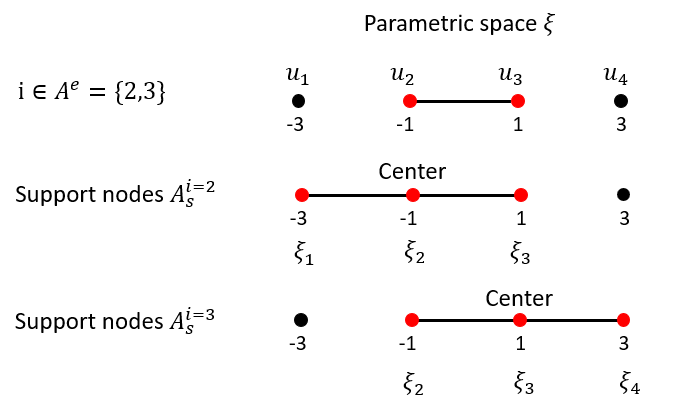}
\caption{Supporting nodes for a 1D convolution element with {the} patch size $s=1$ }
\label{fig:1DCFEMnodes}
\end{figure}

Assuming the nodal solution value for the 4 nodes in \figurename~\ref{fig:1DCFEMnodes} is $[u_1,u_2,u_3,u_4]$, we illustrate the radial basis interpolation procedure for the part centering around $i=2$. In this case, the parametric coordinates for the support nodes are $\{-3,-1,1\}$. Then we can consider the radial basis interpolation $u^{i=2}({\xi})$ has the following form
\begin{equation}      
    u^{i}({\xi})=\mathbf{\Psi}_a (\xi)\mathbf{k}+\mathbf{p}(\xi)\mathbf{l},
\end{equation}
where $\mathbf{\Psi}_a (\xi) $ is  a defined kernel function, which can be the reproducing kernel  or cubic spline kernel \cite{liu1995reproducing,chen2017reproducing} with the dilation parameter $a$, $\mathbf{p}(\xi)$ is the polynomial basis vector of {order $p$}, $\mathbf{k}=[k_1, k_2, k_3]^T$ and {$\mathbf{l}$ is the coefficient vector that helps to enforce the reproducing condition and the Kronecker delta property. Therefore, the size of $\mathbf{l}$ depends on the order of $\mathbf{p}$. In the case of a second order polynomial shown below,  $\mathbf{l}=[l_1, l_2, l_3]^T$.}  We give here {the} specific example for $\mathbf{\Psi}_a $  and $\mathbf{p}(\xi)$ using a cubic spline {kernel} and {the} second-order polynomial, {in which we can see that $a$ is  the dilation parameter that controls the window size (non-zero domain) of the cubic spline}
\begin{equation}      
\begin{aligned} 
   \mathbf{\Psi}_a (\xi)=[&{\Psi}_a (\xi-\xi_1),\  {\Psi}_a (\xi-\xi_2) , \ {\Psi}_a (\xi-\xi_3)] {,}\\
   \text{where}\quad {\Psi}_a (\xi-\xi_I): &= {\Psi}_a (z) \quad \text{with} \quad z=\frac{|\xi-\xi_I|}{a} {,}\\
   {\Psi}_a (z)&=\begin{cases}
    \frac{2}{3} -4z^2+4z^3 \quad \forall z \in [0, \frac{1}{2}]\\
    \frac{4}{3} -4z+4z^2- \frac{4}{3}z^3 \quad \forall z \in [\frac{1}{2}, 1]\\
    0 \quad \forall z \in (1,+\infty)
   \end{cases},
\end{aligned}
\end{equation}
and 
\begin{equation}      
\begin{aligned} 
   \mathbf{p}=[1,\ \xi,\ \xi^2] {.}
\end{aligned}
\end{equation}
Now we can compute $\mathbf{k}$ and $\mathbf{l}$ by enforcing the {conditions below} 
\begin{equation}   
\label{eq:reproducecondition}
\begin{aligned} 
\begin{cases}
    u^{i}({\xi_1})= u_1 \\
     u^{i}({\xi_2})= u_2 \\
    u^{i}({\xi_3})= u_3 \\
     {\sum_i {k}_i= 0}\\
     [\xi_1, \xi_2, \xi_3]\ \mathbf{k}=0\\
      [\xi_1^2, \xi_2^2, \xi_3^2]\ \mathbf{k}=0\\
   \end{cases}.
\end{aligned}
\end{equation}
Solving the above equations gives the solution to $\mathbf{k}$ and $\mathbf{l}$, which reads
\begin{equation}      
\begin{aligned} 
\begin{cases}
    \mathbf{k}=\mathbf{K}\boldsymbol{u}\\
    \mathbf{l}=\mathbf{L}\boldsymbol{u}\\
   \end{cases},
\end{aligned}
\end{equation}
with 
\begin{equation}      
\begin{aligned} 
\begin{cases}
    \boldsymbol{u}=[u_1,\ u_2,\ u_3]^T\\
    \mathbf{L}=(\mathbf{P}^T\mathbf{R}_0^{-1}\mathbf{P})^{-1}\mathbf{P}^T\mathbf{R}_0^{-1}\\
    \mathbf{K}=\mathbf{R}_0^{-1}(\mathbf{I}-\mathbf{P}\mathbf{L})
   \end{cases},
\end{aligned}
\end{equation}
and 
\begin{equation}      
\label{eq:moment}
\begin{cases}
    \mathbf{R}_0=\begin{pmatrix}
     \mathbf{\Psi}_a (\xi_1)\\
      \mathbf{\Psi}_a (\xi_2)\\
       \mathbf{\Psi}_a (\xi_3)
    \end{pmatrix}=\begin{pmatrix}
     {\Psi}_a (\xi_1-\xi_1) &\  {\Psi}_a (\xi_1-\xi_2) &\  {\Psi}_a (\xi_1-\xi_3)\\
      {\Psi}_a (\xi_2-\xi_1) &\  {\Psi}_a (\xi_2-\xi_2) &\  {\Psi}_a (\xi_2-\xi_3)\\
       {\Psi}_a (\xi_3-\xi_1) &\  {\Psi}_a (\xi_3-\xi_2) &\  {\Psi}_a (\xi_3-\xi_3)
    \end{pmatrix}\\
    \\
    \mathbf{P}=\begin{pmatrix}
     \mathbf{p} (\xi_1)\\
      \mathbf{p} (\xi_2)\\
       \mathbf{p} (\xi_3)
    \end{pmatrix}=\begin{pmatrix}
     1 &\  \xi_1 &\  \xi_1^2\\
       1 &\  \xi_2 &\  \xi_2^2\\
        1 &\  \xi_3 &\  \xi_3^2\\
    \end{pmatrix}\\
   \end{cases},
\end{equation}
Finally, the radial basis interpolation with the computed coefficients reads
\begin{equation}      
  \begin{aligned}
      u^{i}({\xi}) = \mathbf{\Psi}_a (\xi)\mathbf{k}+\mathbf{p}(\xi)\mathbf{l} &= \mathbf{\Psi}_a (\xi)\mathbf{K}\boldsymbol{u}+\mathbf{p}(\xi)\mathbf{L}\boldsymbol{u}\\
      &= (\mathbf{\Psi}_a (\xi)\mathbf{K}+\mathbf{p}(\xi)\mathbf{L})\boldsymbol{u}\\
      &= \boldsymbol{W}(\xi)\boldsymbol{u}\\
      &= {W}^{{\xi}_i}_{a,1}(\xi) u_1+{W}^{{\xi}_i}_{a,2} (\xi)u_2+{W}^{{\xi}_i}_{a,3}(\xi) u_3\\
      &= \sum_{j\in A^i_s} {W}^{{\xi}_i}_{a,j} ({\xi}) u_j ,  
  \end{aligned}
\end{equation}
where ${W}^{{\xi}_i}_{a,j}$ is obtained by identifying the corresponding coefficient of $u_j$. By analogy, we can compute the other convolution patch functions $W$ with the support $A^{i=3}_s$. {The} detailed mathematical derivation and analysis of the radial basis interpolation can be found in \cite{schaback2001characterization}.

The generalization of {the} above procedure to 2D cases is straightforward. {By} assuming $K$ is the number of nodes in the support $A^{i}_s$, the 2D cubic spline kernel $\mathbf{\Psi}_a (\boldsymbol{\xi})$  and second-order polynomial $\mathbf{p}(\boldsymbol{\xi})$ can be defined as
\begin{equation}      
\begin{aligned} 
   &\mathbf{\Psi}_a (\boldsymbol{\xi}) =[{\Psi}_a (\boldsymbol{\xi}-\boldsymbol{\xi}_1),\  {\Psi}_a (\boldsymbol{\xi}-\boldsymbol{\xi}_2),\dots , \ {\Psi}_a (\boldsymbol{\xi}-\boldsymbol{\xi}_K)] {,}\\
   \text{with}\quad &{\Psi}_a (\boldsymbol{\xi}-\boldsymbol{\xi}_I): ={\Psi}_a (\xi-\xi_I){\Psi}_a (\eta-\eta_I) {,}\\
\end{aligned}
\end{equation}
and 
\begin{equation}      
\begin{aligned} 
   \mathbf{p}=[1,\ \xi,\ \xi^2,\ \eta,\ \xi\eta,\ \eta^2] {.}
\end{aligned}
\end{equation}
The remaining equations from \eqref{eq:reproducecondition} to \eqref{eq:moment} can be adapted accordingly. 

{
\section{{Illustration of the 1D convolution shape functions}}
\label{apdx:1DCFEshapefunction}
For a better understanding of the convolution shape function $\tilde{{N}}_k$, we illustrate here a 1D convolution approximation with the patch size $s=1$. Specifically, we use the example of \figurename~\ref{fig:1DCFEMnodes} for the supporting nodes. The convolution patch function ${W}^{{\xi}_i}_{a,j}$ can be precomputed using the procedure described in \ref{apdx:radialbasis}. Recall that the general 1D convolution approximation is written as
\begin{equation}
\label{eq:CFEM-1D}       
    u^{\text{CFE}}({\xi})=\sum_{i\in A^e}{N}_{{i}}({\xi})\sum_{j\in A^i_s} {W}^{{\xi}_i}_{a,j} ({\xi}) u_j {.}
\end{equation}
In this specific example, the FE shape function nodal support set is $A^e=\{2,3\}$, and the nodal patch is $A^{i=2}_s=\{1,2,3\}$, $A^{i=3}_s=\{2,3,4\}$. Eq. \eqref{eq:CFEM-1D} then becomes 
\begin{equation}
    \begin{aligned}           
    u^{\text{CFE}}({\xi}) & = \sum_{i\in A^e}{N}_{{i}}({\xi})\sum_{j\in A^i_s} {W}^{{\xi}_i}_{a,j} ({\xi})\ u_j\\
    & = {N}_{{2}}{W}^{{\xi}_2}_{a,1}\ u_1+({N}_{{2}}{W}^{{\xi}_2}_{a,2}+{N}_{{3}}{W}^{{\xi}_3}_{a,2})\ u_2\\
    &+({N}_{{2}}{W}^{{\xi}_2}_{a,3}+{N}_{{3}}{W}^{{\xi}_3}_{a,3})\ u_3+{N}_{{3}}{W}^{{\xi}_3}_{a,4}\ u_4 \\
    & = \sum_{k\in A^e_s} \tilde{{N}}_k({\xi})\ u_k,
    \end{aligned}
\end{equation}
where $A^e_s=\underset{i\in A^e}{\bigcup} A^{i}_s =\{1,2,3,4\}$. Therefore, there are in total 4 convolution shape functions for $s=1$. If $s=2$, we can expect $6$ shape functions, as shown in \figurename~\ref{fig:1Dshape} and \ref{fig:CFEshape-global}. More details about how to deal with  elements close to the boundary and irregular meshes can be found in \cite{lu2023convolution}.

\section{{Derivation of the discretized  formulation for the CTD solution}}
\label{apdx:CTDderivation}
Let us consider the following weak form problem in 2D with $\Omega = \Omega_x \times \Omega_y$
\begin{equation}
\label{eq:CTDappend}
    \int_\Omega  \delta u \ (u^{k+1}-u^{k})(\frac{1}{\Delta t}+\alpha L)\ d \Omega + \int_\Omega \delta u \  L w({u}^{k}) \ d \Omega +\int_\Omega  \nabla\delta u\cdot  L\kappa \nabla{u}^{k+1} \ d \Omega =0  {.}
\end{equation}
The decomposition of solution leads to 
\begin{equation}
\begin{aligned}
        u^{k+1}&=\sum_{m=1}^{M-1} u^{(m)}_x u^{(m)}_y +u_x u_y {,}\\ 
        \delta u&=\delta u_x u_y +  u_x \delta u_y {,}\\ 
    u^{k}&=\sum_{m=1}^{M_{k}} u^{(m,k)}_x u^{(m,k)}_y  {,}
\end{aligned}
\end{equation}
where $u^{(m,k)}_x$ and $u^{(m,k)}_y $ are the CTD solution at the previous time step, $u^{(m)}_x$ and $u^{(m)}_y $, $\forall m\leq M-1$, are assumed known by previous computations. The $M$-th mode $u_x, u_y$ can be computed alternatively by assuming one of them is fixed. 

Starting from computing $u_x$, we consider now $u_y$ is a constant function and known by the initial guess or previous computations, then $\delta u_y=0$ and $\delta u=\delta u_x u_y$, the weak form problem \eqref{eq:CTDappend} becomes
\begin{equation}
\begin{aligned}
    \int_\Omega  \delta u_x u_y\ &(\sum_{m=1}^{M-1} u^{(m)}_x u^{(m)}_y +u_x u_y-\sum_{m=1}^{M_{k}} u^{(m,k)}_x u^{(m,k)}_y )(\frac{1}{\Delta t}+\alpha L)\ d \Omega \\
    &+ \int_\Omega \delta u_x u_y \  L w({u}^{k}) \ d \Omega +\int_\Omega  \nabla(\delta u_x u_y)\cdot  L\kappa \nabla(\sum_{m=1}^{M-1} u^{(m)}_x u^{(m)}_y +u_x u_y) \ d \Omega =0    {.}
\end{aligned}
\end{equation}
By rearrangement and considering $\int_\Omega = \int_{\Omega_x} \int_{\Omega_y}$, we have
\begin{equation}
\begin{aligned}
   (\frac{1}{\Delta t}+\alpha L)&\sum_{m=1}^{M-1}  \int_{\Omega_x}  \delta u_x u^{(m)}_x\ dx  \int_{\Omega_y}   u_y u^{(m)}_y\ dy  +  (\frac{1}{\Delta t}+\alpha L)\int_{\Omega_x}  \delta u_x u_x\ dx \int_{\Omega_y}   u_y  u_y\ dy \\
   &-(\frac{1}{\Delta t}+\alpha L)\sum_{m=1}^{M_{k}}  \int_{\Omega_x}  \delta u_x u^{(m,k)}_x \ dx \int_{\Omega_y}   u_y u^{(m,k)}_y \ dy\\
    &+ \int_\Omega \delta u_x u_y \  L w({u}^{k}) \ d \Omega\\
    &+L\kappa \sum_{m=1}^{M-1}\int_{\Omega_x}  \frac{\partial\delta u_x}{\partial x} \frac{\partial u^{(m)}_x}{\partial x}\ dx \int_{\Omega_y} u_y  u^{(m)}_y\ dy\\
    &+ L\kappa \sum_{m=1}^{M-1}\int_{\Omega_x}  \delta u_x  u^{(m)}_x \ dx \int_{\Omega_y} \frac{\partial u_y}{\partial y} \frac{\partial u^{(m)}_y}{\partial y} \ dy\\
    &+L\kappa \int_{\Omega_x}  \frac{\partial\delta u_x}{\partial x} \frac{\partial u_x}{\partial x}\ dx \int_{\Omega_y} u_y  u_y\ dy\\ 
    &+ L\kappa \int_{\Omega_x}  \delta u_x  u_x \ dx \int_{\Omega_y} \frac{\partial u_y}{\partial y} \frac{\partial u_y}{\partial y} \ dy\\
    &=0  {.}   
\end{aligned}
\end{equation}
Considering the convolution approximation \eqref{eq:cfeshape1D} and let $\Tilde{\boldsymbol{B}}_x=\frac{d\Tilde{\boldsymbol{N}}_x}{dx}$ and $\Tilde{\boldsymbol{B}}_y=\frac{d\Tilde{\boldsymbol{N}}_y}{dy}$, the above equation becomes
\begin{equation}
\begin{aligned}
   (\frac{1}{\Delta t}+\alpha L)&\sum_{m=1}^{M-1}   \delta {\boldsymbol{u}}^T_x\int_{\Omega_x}  \Tilde{\boldsymbol{N}}^T_x \Tilde{\boldsymbol{N}}_x dx\ \boldsymbol{u}^{(m)}_x {\boldsymbol{u}}^T_y \int_{\Omega_y}   \Tilde{\boldsymbol{N}}^T_y \Tilde{\boldsymbol{N}}_y  dy\  \boldsymbol{u}^{(m)}_y \\
   &+  (\frac{1}{\Delta t}+\alpha L) \delta {\boldsymbol{u}}^T_x\int_{\Omega_x}  \Tilde{\boldsymbol{N}}^T_x \Tilde{\boldsymbol{N}}_x dx \ \boldsymbol{u}_x {\boldsymbol{u}}^T_y \int_{\Omega_y}   \Tilde{\boldsymbol{N}}^T_y \Tilde{\boldsymbol{N}}_y  dy\  \boldsymbol{u}_y  \\
   &-(\frac{1}{\Delta t}+\alpha L)\sum_{m=1}^{M_{k}}  \delta {\boldsymbol{u}}^T_x\int_{\Omega_x}  \Tilde{\boldsymbol{N}}^T_x \Tilde{\boldsymbol{N}}_x dx \ \boldsymbol{u}^{(m,k)}_x {\boldsymbol{u}}^T_y \int_{\Omega_y}   \Tilde{\boldsymbol{N}}^T_y \Tilde{\boldsymbol{N}}_y  dy\  \boldsymbol{u}^{(m,k)}_y \\
    &+ \delta {\boldsymbol{u}}^T_x \int_\Omega \Tilde{\boldsymbol{N}}^T_x \Tilde{\boldsymbol{N}}_y\  L w({u}^{k}) \ dx dy \ {\boldsymbol{u}}_y\\
    &+L\kappa \sum_{m=1}^{M-1}\delta {\boldsymbol{u}}^T_x\int_{\Omega_x}  \Tilde{\boldsymbol{B}}^T_x \Tilde{\boldsymbol{B}}_x dx \ \boldsymbol{u}^{(m)}_x {\boldsymbol{u}}^T_y \int_{\Omega_y}   \Tilde{\boldsymbol{N}}^T_y \Tilde{\boldsymbol{N}}_y  dy\  \boldsymbol{u}^{(m)}_y \\
    &+ L\kappa \sum_{m=1}^{M-1}\int_{\Omega_x}  \delta {\boldsymbol{u}}^T_x\int_{\Omega_x}  \Tilde{\boldsymbol{N}}^T_x \Tilde{\boldsymbol{N}}_x dx \ \boldsymbol{u}^{(m)}_x {\boldsymbol{u}}^T_y \int_{\Omega_y}   \Tilde{\boldsymbol{B}}^T_y \Tilde{\boldsymbol{B}}_y  dy\  \boldsymbol{u}^{(m)}_y \\
    &+L\kappa\ \delta {\boldsymbol{u}}^T_x\int_{\Omega_x}  \Tilde{\boldsymbol{B}}^T_x \Tilde{\boldsymbol{B}}_x dx \ \boldsymbol{u}_x {\boldsymbol{u}}^T_y \int_{\Omega_y}   \Tilde{\boldsymbol{N}}^T_y \Tilde{\boldsymbol{N}}_y  dy\  \boldsymbol{u}_y \\ 
    &+ L\kappa\ \delta {\boldsymbol{u}}^T_x\int_{\Omega_x}  \Tilde{\boldsymbol{N}}^T_x \Tilde{\boldsymbol{N}}_x dx \ \boldsymbol{u}_x {\boldsymbol{u}}^T_y \int_{\Omega_y}   \Tilde{\boldsymbol{B}}^T_y \Tilde{\boldsymbol{B}}_y  dy\  \boldsymbol{u}_y \\
    &=0     {.}
\end{aligned}
\end{equation}
With the definition \eqref{eq:KxxMxx}, we have 
\begin{equation}
\begin{aligned}
   (\frac{1}{\Delta t}+\alpha L)&\sum_{m=1}^{M-1}   \delta {\boldsymbol{u}}^T_x\bold{M}_{xx}  \boldsymbol{u}^{(m)}_x {\boldsymbol{u}}^T_y \bold{M}_{yy}  \boldsymbol{u}^{(m)}_y \\
   &+  (\frac{1}{\Delta t}+\alpha L) \delta {\boldsymbol{u}}^T_x\bold{M}_{xx} \boldsymbol{u}_x {\boldsymbol{u}}^T_y \bold{M}_{yy} \boldsymbol{u}_y  \\
   &-(\frac{1}{\Delta t}+\alpha L)\sum_{m=1}^{M_{k}}  \delta {\boldsymbol{u}}^T_x\bold{M}_{xx}  \boldsymbol{u}^{(m,k)}_x {\boldsymbol{u}}^T_y \bold{M}_{yy} \boldsymbol{u}^{(m,k)}_y \\
    &+ \delta {\boldsymbol{u}}^T_x \bold{M}_{xy} {\boldsymbol{u}}_y\\
    &+L\kappa \sum_{m=1}^{M-1}\delta {\boldsymbol{u}}^T_x\bold{K}_{xx}  \boldsymbol{u}^{(m)}_x {\boldsymbol{u}}^T_y \bold{M}_{yy}  \boldsymbol{u}^{(m)}_y \\
    &+ L\kappa \sum_{m=1}^{M-1}  \delta {\boldsymbol{u}}^T_x\bold{M}_{xx} \boldsymbol{u}^{(m)}_x {\boldsymbol{u}}^T_y \bold{K}_{yy} \boldsymbol{u}^{(m)}_y \\
    &+L\kappa\ \delta {\boldsymbol{u}}^T_x\bold{K}_{xx}  \boldsymbol{u}_x {\boldsymbol{u}}^T_y \bold{M}_{yy} \boldsymbol{u}_y \\ 
    &+ L\kappa\ \delta {\boldsymbol{u}}^T_x\bold{M}_{xx} \boldsymbol{u}_x {\boldsymbol{u}}^T_y \bold{K}_{yy}  \boldsymbol{u}_y \\
    &=0    {.} 
\end{aligned}
\end{equation}
Rearranging the above equation gives 
\begin{equation}
    \bold{K}_x\boldsymbol{u}_x=\bold{Q}_x  {,}
\end{equation}
with $\bold{K}_x$ and $\bold{Q}_x$ defined by \eqref{eq:KxQx}.

By analogy, we can derive the discretized equation for solving $u_y$, by considering  $u_x$ is a constant function and known by the initial guess or previous computations, then $\delta u_x=0$ and $\delta u= u_x \delta u_y$. In this case, the weak form problem \eqref{eq:CTDappend} becomes
\begin{equation}
\begin{aligned}
    \int_\Omega  u_x \delta  u_y\ &(\sum_{m=1}^{M-1} u^{(m)}_x u^{(m)}_y +u_x u_y-\sum_{m=1}^{M_{k}} u^{(m,k)}_x u^{(m,k)}_y )(\frac{1}{\Delta t}+\alpha L)\ d \Omega \\
    &+ \int_\Omega u_x \delta  u_y \  L w({u}^{k}) \ d \Omega +\int_\Omega  \nabla( u_x \delta u_y)\cdot  L\kappa \nabla(\sum_{m=1}^{M-1} u^{(m)}_x u^{(m)}_y +u_x u_y) \ d \Omega =0    {.}
\end{aligned}
\end{equation}
This will lead to the following equation 
\begin{equation}
\begin{aligned}
   (\frac{1}{\Delta t}+\alpha L)&\sum_{m=1}^{M-1}    {\boldsymbol{u}}^T_x\bold{M}_{xx}  \boldsymbol{u}^{(m)}_x \delta {\boldsymbol{u}}^T_y \bold{M}_{yy}  \boldsymbol{u}^{(m)}_y \\
   &+  (\frac{1}{\Delta t}+\alpha L)  {\boldsymbol{u}}^T_x\bold{M}_{xx} \boldsymbol{u}_x \delta{\boldsymbol{u}}^T_y \bold{M}_{yy} \boldsymbol{u}_y  \\
   &-(\frac{1}{\Delta t}+\alpha L)\sum_{m=1}^{M_{k}}   {\boldsymbol{u}}^T_x\bold{M}_{xx}  \boldsymbol{u}^{(m,k)}_x \delta {\boldsymbol{u}}^T_y \bold{M}_{yy} \boldsymbol{u}^{(m,k)}_y \\
    &+ \delta {\boldsymbol{u}}^T_y \bold{M}_{xy}^T {\boldsymbol{u}}_x\\
    &+L\kappa \sum_{m=1}^{M-1}{\boldsymbol{u}}^T_x\bold{K}_{xx}  \boldsymbol{u}^{(m)}_x \delta {\boldsymbol{u}}^T_y \bold{M}_{yy}  \boldsymbol{u}^{(m)}_y \\
    &+ L\kappa \sum_{m=1}^{M-1} {\boldsymbol{u}}^T_x\bold{M}_{xx} \boldsymbol{u}^{(m)}_x  \delta {\boldsymbol{u}}^T_y \bold{K}_{yy} \boldsymbol{u}^{(m)}_y \\
    &+L\kappa\  {\boldsymbol{u}}^T_x\bold{K}_{xx}  \boldsymbol{u}_x \delta{\boldsymbol{u}}^T_y \bold{M}_{yy} \boldsymbol{u}_y \\ 
    &+ L\kappa\  {\boldsymbol{u}}^T_x\bold{M}_{xx} \boldsymbol{u}_x \delta {\boldsymbol{u}}^T_y \bold{K}_{yy}  \boldsymbol{u}_y \\
    &=0     {.}
\end{aligned}
\end{equation}
Rearranging the above equation gives 
\begin{equation}
    \bold{K}_y\boldsymbol{u}_y=\bold{Q}_y  {,}
\end{equation}
with $\bold{K}_y$ and $\bold{Q}_y$ defined by \eqref{eq:KyQy}.

\bibliographystyle{model1-num-names}
\bibliography{CTDAC.bib}







\end{document}